\documentclass[11pt,reqno]{amsart}

\usepackage[margin=1in]{geometry}
\usepackage{ifpdf}
\usepackage[numbers,sort&compress]{natbib}

\usepackage{hyperref}
\hypersetup{
  colorlinks=true,
  breaklinks=true,
  plainpages=false,
  citecolor=blue,
  linkcolor=blue,
  urlcolor=blue,
  bookmarksopen=true,
  bookmarksnumbered=true,
  bookmarksdepth=8,
  unicode=true,
  pdftitle={Whittle-Mat\'{e}rn Fields with Variable Smoothness},
  pdfauthor={Hamza Ruzayqat, Wenyu Lei, David Bolin, George Turkiyyah and Omar Knio}
}

\makeatletter
\newtheoremstyle{thmstyletwo}
{12pt plus2pt minus1pt}{12pt plus2pt minus1pt}
{\itshape}{0pt}{\normalfont\bfseries}{}{1em}
{\thmname{#1}\thmnumber{\@ifnotempty{#1}{ }{#2}}\thmnote{ {\the\thm@notefont(#3)}}}
\newtheoremstyle{thmstylethree}
{12pt plus2pt minus1pt}{12pt plus2pt minus1pt}
{\normalfont}{0pt}{\bfseries}{}{1em}
{\thmname{#1}\thmnumber{\@ifnotempty{#1}{ }\@upn{#2}}\thmnote{ {\the\thm@notefont(#3)}}}
\makeatother

\PassOptionsToPackage{final}{changes}

\usepackage[utf8]{inputenc}   
\usepackage[T1]{fontenc}      
\DeclareUnicodeCharacter{2060}{}
\usepackage{amsfonts}
\usepackage{graphicx}
\usepackage{epstopdf}
\usepackage{amssymb}
\usepackage{amsmath}

\usepackage{bm}

\usepackage{enumerate}

\usepackage{comment}

\usepackage[commandnameprefix=ifneeded]{changes}
\definechangesauthor[name={Revision}, color=blue]{rev}
\makeatletter
\@ifpackagewith{changes}{final}{%
  \newenvironment{revblock}{}{}%
  \excludecomment{revdelblock}%
}{%
  \newenvironment{revblock}{\color{blue}}{}%
}
\makeatother

\usepackage{float}       
\usepackage{placeins}    
\usepackage{afterpage}   
\usepackage{caption}     
\usepackage{multirow}
\usepackage[nameinlink]{cleveref}    

\ifpdf
  \DeclareGraphicsExtensions{.eps,.pdf,.png,.jpg}
\else
  \DeclareGraphicsExtensions{.eps}
\fi

\theoremstyle{thmstyletwo}
\newtheorem{theorem}{Theorem}[section]
\newtheorem{lemma}[theorem]{Lemma}
\newtheorem{corollary}[theorem]{Corollary}
\newtheorem{prop}[theorem]{Proposition}
\newtheorem{proposition}[theorem]{Proposition}

\theoremstyle{thmstylethree}
\newtheorem{definition}[theorem]{Definition}
\newtheorem{remark}[theorem]{Remark}

\crefname{prop}{Proposition}{Propositions}
\crefname{proposition}{Proposition}{Propositions}
\crefname{corollary}{Corollary}{Corollaries}
\crefname{hypothesis}{Hypothesis}{Hypotheses}
\crefname{claim}{Claim}{Claims}
\crefname{definition}{Definition}{Definitions}

\newcommand{\sref}[1]{\texorpdfstring{\Cref{#1}}{}}

\usepackage{etoolbox}   

\AtBeginEnvironment{align}{\vspace{-6pt}}
\AtEndEnvironment{align}{\vspace{-8pt}}
\AtBeginEnvironment{align*}{\vspace{-8pt}}
\AtEndEnvironment{align*}{\vspace{-8pt}}

\usepackage{amsopn}

\makeatletter
\newcommand*{\addFileDependency}[1]{
  \typeout{(#1)}
  \@addtofilelist{#1}
  \IfFileExists{#1}{}{\typeout{No file #1.}}
}
\makeatother

\title{Whittle-Mat\'{e}rn Fields with Variable Smoothness}

\author{Hamza Ruzayqat\textsuperscript{$\ast,\dagger$}}
\author{Wenyu Lei\textsuperscript{$\ddagger$}}
\author{David Bolin\textsuperscript{$\ast$}}
\author{George Turkiyyah\textsuperscript{$\ast$}}
\author{Omar Knio\textsuperscript{$\ast$}}

\date{}

\begin{document}

\begin{abstract}
We introduce and analyze a nonlocal generalization of Whittle--Mat\'ern Gaussian fields in which the smoothness parameter varies in space through the fractional order, $s=s(x)\in[\underline s,\overline s]\subset(0,1)$. The model is defined via an integral-form operator whose kernel is constructed from the modified Bessel function of the second kind and whose local singularity is governed by the symmetric exponent $\beta(x,y)=(s(x)+s(y))/2$. This variable-order nonlocal formulation departs from the classical constant-order pseudodifferential setting and raises new analytic and numerical challenges. We develop a novel variational framework adapted to the kernel, prove existence and uniqueness of weak solutions on truncated bounded domains, and derive Sobolev regularity of the Gaussian (spectral) solution controlled by the minimal local order, with realizations lying in $H^r(\mathcal G)$ for every $r<2\underline s-\tfrac{d}{2}$ with $d\ge2$ or $\underline s\le1/2$, with a slightly reduced range in the remaining case (i.e. $d=1$ and $\underline s > 1/2$), hence in $L_2(\mathcal G)$ when $\underline s>d/4$. Here $H^r(\mathcal G)$ denotes the Sobolev space on the bounded domain $\mathcal G$. We also present a finite-element sampling method for the integral model, derive error estimates, and provide numerical experiments in one dimension that illustrate the impact of spatially varying smoothness on the solution covariance. Computational aspects and directions for scalable implementations are discussed.
\end{abstract}

\maketitle

\footnotetext[1]{Computer, Electrical and Mathematical Sciences and Engineering Division, King Abdullah University of Science and Technology, Thuwal 23955, KSA}
\footnotetext[2]{Corresponding author. Email: \texttt{hamza.ruzayqat@kaust.edu.sa}}
\footnotetext[3]{School of Mathematical Sciences, University of Electronic Science and Technology of China, No.\ 2006, Xiyuan Ave, West Hi-Tech Zone, Chengdu 611731, China}


\noindent \textbf{MSC classes}: 60G15, 60H15, 35R11, 46E35, 65N30.\\

\section{Introduction}
Gaussian random fields are fundamental models in statistics and machine learning. Gaussian random fields can be specified through their covariance function, and the most commonly used covariance function is the Mat\'ern covariance function \cite{Matern1960},
\begin{align}
\label{eq:mat_cov}
\varrho(x_1,x_2) = \frac{2^{1-\nu} \sigma^2}{\Gamma(\nu)} \left(\kappa \|x_1 - x_2\| \right)^\nu K_\nu (\kappa \|x_1 - x_2\| ), \quad x_1,x_2 \in \mathcal{D}\subseteq\mathbb{R}^d,
\end{align}
where $\Gamma$ denotes the gamma function, $\|\cdot\|$ the Euclidean norm on $\mathbb{R}^d$ and $K_\nu$ is the modified Bessel function of the second kind of order $\nu$. The parameters $\sigma, \nu, \kappa>0$ control
the variance, smoothness, and correlation range of the random field, respectively. 

One drawback of this model is that the parameters are constant over space, which means that it cannot capture non-stationarities, such as a spatially varying practical correlation range. Because non-stationary covariance structures often are observed in applications in statistics, there has been considerable interest over the past 30 years to construct flexible non-stationary extensions of the Mat\'ern covariance function. Some popular approaches are space-deformations \cite{sampson1992nonparametric}, process convolutions \cite{higdon01}, and more explicit constructions which directly formulate valid non-stationary covariance functions \cite{paciorek2006spatial}. One of the most popular approaches is the stochastic partial differential equation (SPDE) approach \cite{Lindgren2011}, which is based on the fact that a stationary solution $u:\mathbb{R}^d\times \Omega \to \mathbb{R}$  ($\Omega$ is a sample space) to the fractional-order SPDE
\begin{align}
\label{eq:frac_spde}
(\kappa^2 - \Delta)^s u = \frac{1}{\mu}\mathcal{W}, \quad \text{ on } \mathbb{R}^d,
\end{align}
has covariance \eqref{eq:mat_cov} \cite{Whittle1963}, with range parameter $\kappa$, smoothness $\nu = 2s - d/2$, and variance
\begin{align}
\sigma^2 = \frac{\Gamma(\nu)}{(4\pi)^{d/2} \kappa^{2\nu}\Gamma(\nu+d/2)\mu^2}.
\end{align}
Here $\Delta$ is the Laplacian, $\mu>0$ is a scaling factor and $\mathcal{W}$ is a spatial Gaussian white noise in $\mathbb{R}^d$. Motivated by the SPDE formulation, \cite{Lindgren2011} proposed extending the Mat\'ern model to accommodate non-stationary behavior and more general spatial domains. Their work sparked a vibrant line of research focused on spatial models constructed via SPDEs; see, for example, \cite{Bakka2019, Bolin2011, Fuglstad2015, Hildeman2021, Bolin2023a, Bolin2023b, Bolin2024}. In particular, the authors derived a connection between Gaussian fields and Gaussian Markov random fields on a bounded domain $\mathcal{D} \subset \mathbb{R}^d$ by using an approximate weak solution to \eqref{eq:frac_spde}, posed on a bounded domain $\mathcal{D} \subset \mathbb{R}^d$, where the operator was equipped with
Neumann boundary conditions. Besides reducing computational cost compared to the covariance-based approximations, \cite{Lindgren2011} also used this idea to extend the Mat\'ern fields to non-stationary models by allowing the parameters $\kappa$ and $\mu$ to be spatially varying functions. This idea was later used to construct the so-called generalized Whittle-Mat\'ern fields, which are formulated by considering the SPDE
\begin{align}
\label{eq:general_Lu=W}
L^s u = \frac{1}{\mu}\mathcal{W}, \quad \text{ on } \mathcal{D},
\end{align}
where the operator is given as $L= \kappa^2 - \nabla \cdot (H\nabla)$, for bounded and non-negative functions $\kappa$, $\mu$ and a (sufficiently nice) matrix-valued function $H$ \cite{Bolin2020, Cox2020}. The term Whittle-Mat\'ern fields is used to denote, in a broader sense, the class of solutions to \eqref{eq:general_Lu=W}, encompassing not only the stationary fields defined on $\mathbb{R}^d$,  but also solutions formulated on bounded domains with boundary conditions and on general manifold domains.

A drawback of the finite element (FE) approximation proposed by \cite{Lindgren2011} is that it is computable only if $2s \in \mathbb{N}$. This restriction has later been removed by combining a FE approximation with a rational approximation of either the fractional operator $L^{-s}$ \cite{Bolin2020} or of the corresponding covariance operator $L^{-2s}$ \cite{Bolin2023a}. Either option results in a computationally efficient approximation which is applicable for $s > d/4$. In contrast to the approach in \cite{Lindgren2011}, where Neumann boundary conditions are employed, the authors in \cite{Bolin2020} consider homogeneous Dirichlet boundary conditions. However, the choice of boundary conditions typically has little practical impact, as the error in the covariance of the field decays rapidly away from the boundary \cite{Potsepaev2010}.

Although the SPDE approach can be used to construct non-stationary Gaussian processes, it is limited to models with constant smoothness. Similarly, most previously proposed methods for constructing non-stationary covariance functions, such as the deformation approach, do not allow for spatially varying smoothness. This is a significant limitation, as many physical phenomena naturally exhibit variability in smoothness across space.
For example if precipitation or temperature is modeled on a global scale, it is likely that these fields are smoother over oceans compared to over land. Motivated by this, we aim in this work to extend the SPDE approach to allow a spatially varying fractional order $s(x)$, which in turn controls the smoothness of the field. To accomplish this, we consider an integral-based formulation of the fractional operator instead of a spectral definition of it as used, for example, in \cite{Lindgren2011} and \cite{Bolin2020}.

The main contributions of this paper are summarized below.
\begin{itemize}
\item We formulate a variable-order Whittle–Mat\'ern integral operator with a heterogeneous Bessel-type kernel on a bounded domain $\mathcal G$. 

\item We develop a rigorous variational framework, including completeness of the adapted energy space, denoted $\mathbb V_{\kappa,s}$, and a well-posed weak formulation.

\item We establish spectral and regularity results, including compact embedding $\mathbb V_{\kappa,s}\hookrightarrow L_2(\mathcal G)$, compact resolvent, two-sided spectral bounds, and Sobolev regularity of the spectral solution by the minimal order $\underline s = \min s(x)$ and prove uniqueness in the class of Gaussian generalized fields.

\item We propose a finite-element sampling algorithm, provide error estimates, and present 1D numerical experiments illustrating the effect of spatially varying smoothness on the solution covariance and comparing with Mat\'ern covariance.
\end{itemize}

The remainder of the paper is organized as follows. In the next subsection we introduce notation used throughout the paper. In \autoref{sec:model_form} we present the integral-form model via a Bessel-kernel representation that underlies our framework for both constant and spatially varying smoothness. In \autoref{sec:weak_form} we construct an energy space adapted to the kernel in the non-constant case, prove that the weak formulation is well posed, and show existence and uniqueness of the Gaussian field solution together with regularity results under suitable sufficient conditions. In \autoref{sec:FEM} we introduce the FE solver for the integral equation, derive error estimates for the FE approximation, describe the quadrature rule used, and present the sampling algorithm. Finally, in \autoref{sec:numerics} we present numerical experiments focused on one spatial dimension to illustrate the proposed algorithm. \added[id=rev]{We emphasize that the model and the analysis of \autoref{sec:model_form}--\autoref{sec:weak_form} are formulated in any dimension $d\ge1$; the restriction to $d=1$ applies only to the finite element solver of \autoref{sec:FEM} and the experiments of \autoref{sec:numerics}; the multidimensional solver and the associated statistical inference are developed in the companion work \cite{Ruzayqat2026sampling}.}

\subsection{Notation}
\label{sec:notation}
For an open, bounded domain $\mathcal{D} \subset \mathbb{R}^d$, we denote the standard inner product on $L_2(\mathcal{D})$ by $\langle u, v\rangle_{L_2(\mathcal{D})} := \int_{\mathcal{D}} u(x) v(x) \, dx, \quad u, v \in L_2(\mathcal{D})$, with the associated norm denoted by $\|u\|_{L_2(\mathcal{D})} := \sqrt{\langle u, u\rangle_{L_2(\mathcal{D})}}$. For a Banach space $\mathbb V$ we write $\langle F,v\rangle_{\mathbb V',\mathbb V}$ for the duality pairing between $F\in\mathbb V'$ and $v\in\mathbb V$ and it denotes the action of the linear bounded functional $F$ on $v$. Let $(\Omega, \mathcal{A}, \mathbb{P})$ be a complete probability space. We denote by $L_2(\Omega; \mathbb{R})$ the space of real-valued random variables with finite second moment, defined by $L_2(\Omega; \mathbb{R}) := \{ X : \Omega \to \mathbb{R} \text{ measurable}:\, \mathbb{E}[|X|^2] < \infty \}$. We denote by $\mathcal{L}(L_2(\mathcal{D}), L_2(\Omega; \mathbb{R}))$ the space of generalized random fields, defined as continuous linear maps from $L_2(\mathcal{D})$ to $L_2(\Omega; \mathbb{R})$. For a map $f \in \mathcal{L}(L_2(\mathcal{D}), L_2(\Omega; \mathbb{R}))$, its action on a test function $\phi \in L_2(\mathcal{D})$ is denoted by $f(\phi)$. For a given function $f(x)$, we denote by $\mathcal{F}(f)$ the Fourier transform, given by $\mathcal{F}(f)(\xi) = \int_{\mathbb{R}^n} f(x)\, e^{-2\pi i x \cdot \xi} \, dx$, for $\xi \in \mathbb{R}^n$.

\section{Model formulation}
\label{sec:model_form}
In this section, we present the integral-form model underlying our framework. We first recall in \autoref{sec:const_smooth} the well-known integral representation of the operator for the constant-smoothness case $s(x)\equiv s$, and then define the associated boundary-value problem on a bounded domain. 
Then, in the subsequent subsection, we extend this formulation to the spatially varying case $s=s(x)$, discuss the corresponding nonlocal operator, and introduce the associated variational and stochastic formulations.

\subsection{Model formulation for constant smoothness}
\label{sec:const_smooth}
Assume $\kappa>0$ and $0<s <1$, then 
\begin{align}
\label{eq:int_form}
(\kappa^2-\Delta)^{s}u :=  \kappa^{2s} u(x) +   \int_{\mathbb{R}^d} \frac{w_{-2s}(x,y)(u(x)-u(y))}{\|x-y\|^{d+2s}} dy, 
\end{align}
by \cite[page 277]{Samko2001}, where
\begin{align*}
w_{-2s}(x,y) := c_s\kappa^\frac{d+2s}{2}~ \|x-y\|^\frac{d+2s}{2} K_{\frac{d+2s}{2}}(\kappa \|x-y\|),
\quad c_s = \frac{2^{1-\frac{d-2s}{2}}}{\pi^{\frac{d}{2}} |\Gamma(-s)|} = \frac{s \, 2^{1-\frac{d-2s}{2}}}{\pi^{\frac{d}{2}} \Gamma(1-s)}.  
\end{align*}
Let $\mathcal{D}$ be a bounded and Lipchitz domain and define $\mathcal{D}^c:= \mathbb{R}^d\setminus \mathcal{D}$ as the complement of $\mathcal{D}$. For a function $u \in L_2(\mathcal{D})$, let $\tilde{u}$ be the zero extension of $u$ on $\mathbb{R}^d$. That is, $\tilde{u}|_{\mathcal{D}} = u$ and $\tilde{u}|_{\mathcal{D}^c} = 0$. We define an operator $A_s$ as
\begin{align*}
A_s u &= (\kappa^2-\Delta)^{s}\tilde{u}= \kappa^{2s} u(x) +   \int_{\mathbb{R}^d} \frac{w_{-2s}(x,y)(\tilde{u}(x)-\tilde{u}(y))}{\|x-y\|^{d+2s}} dy, 
\end{align*}
and consider the equation
\begin{align}
\label{eq:loc_frac_spde}
\begin{array}{rll}
A_s u &=& \frac{1}{\mu}\mathcal{W}, \quad \text{ on }\mathcal{D},\\
u & = & 0, \quad \text{ on }\mathcal{D}^c,
\end{array}
\end{align}
where $\mu>0$ and the white noise $\mathcal W$ is defined as in \cite[Appendix B.2, Definition 6]{Lindgren2011} (see also \cite{Hida1993}), which we include here
\begin{definition}[Gaussian white noise on~$\mathcal G$]\label{def:LN_white}
Let $\mathcal G$ be a compact Riemannian manifold (or bounded domain) with Lebesgue measure. A \emph{Gaussian white noise} $\mathcal W$ on $\mathcal G$ is an $L_2(\mathcal G)$-bounded generalized Gaussian field such that for every finite family $\{\phi_i\}_{i=1}^n\subset L_2(\mathcal G)$ the random vector
$\big\{\mathcal W(\phi_i)\big\}_{i=1}^n$ is jointly Gaussian with $\mathbb E\big[\mathcal W(\phi_i)\big] = 0$, and $ 
\operatorname{Cov}\big(\mathcal W(\phi_i),\mathcal W(\phi_j)\big)
= \langle\phi_i,\phi_j\rangle_{L_2(\mathcal G)}$.
\end{definition}

\begin{remark}
For a test function $\phi\in L_2(\mathcal G)$ the scalar random variable
$\mathcal W(\phi)$ is a well-defined Gaussian.
However, it \emph{does not} automatically follow that there exists a version
$\omega\mapsto \mathcal W(\omega)$ with $\phi \mapsto \langle\phi,\mathcal W(\omega)\rangle_{L_2(\mathcal G)}$ a bounded linear functional on $L_2(\mathcal G)$
for almost every $\omega\in\Omega$. In the sequel, we will provide sufficient conditions for such a statement to hold
(see \sref{thm:W_in_V_prime}).
\end{remark}

\subsection{Model formulation for non-constant smoothness \texorpdfstring{$s(x)$}{}}

To allow for a general definition of the operator $A_s$ with spatially varying fraction we follow a nonlocal calculus framework, see e.g. \cite{Du2013}. For a vector field $v: \mathbb{R}^d \times \mathbb{R}^d \to \mathbb{R}$, we define a general nonlocal divergence operator by $
\mathcal{D}[v](x) := \int_{\mathbb{R}^d} (v(x, y) + v(y, x)) \cdot \alpha(x, y) \, dy$.
In this expression, the vector field $ \alpha(x,y) : \mathbb{R}^d \times \mathbb{R}^d \to \mathbb{R}^d $ possesses the antisymmetric property $\alpha(x,y) = -\alpha(y,x)$. The adjoint operator $ \mathcal{D}^*$ associated with $\mathcal{D}$ in the context of the $L_2(\mathbb{R}^d)$ inner product, which serves as a non-local gradient, is defined through the  identity $\int_{\mathbb{R}^d} u(x)\cdot \mathcal{D}[v](x) dx - \int_{\mathbb{R}^d \times \mathbb{R}^d} \mathcal{D}^*[u](x,y) \cdot v(y, x) dxdy=0$.
The adjoint operator that satisfies the above identity is $\mathcal{D}^*[u](x,y) = - (u(y) - u(x)) \alpha(x, y)$. For a function $u:\mathbb{R}^d \to \mathbb{R}$ and a given symmetric function $a:\mathbb{R}^d\times\mathbb{R}^d \to \mathbb{R}$, \replaced[id=rev]{a direct computation using the definitions of $\mathcal D$ and $\mathcal D^*$ (see  \cite{Du2013}) shows}{ one can shows} that the negative nonlocal Laplacian operator is
\begin{align}
\label{eq:nl-diff}
    \mathcal{L} u &= \mathcal{D}[a ~\mathcal{D}^*[u]](x) =2 \int_{\mathbb R^d} (u(x)-u(y))\gamma (x,y)~ dy
\end{align}
with $\gamma (x,y) := \frac{1}{2} a(x,y) \|\alpha(x,y)\|^2$. Setting $\beta(x,y) := (s(x)+s(y))/2$, we choose $a$ equal to $\tilde{w}_{-2s}$, where
\begin{align*}
\tilde{w}_{-2s}(x,y):=\frac{2^{1-\frac{d}{2}+\beta(x,y)}}{\pi^{\frac{d}{2}} \left|\Gamma(-\beta(x,y))\right|}  \kappa^{\frac{d}{2}+\beta(x,y)}  \|x-y\|^{\frac{d}{2}+\beta(x,y)}   K_{\frac{d}{2}+\beta(x,y)}(\kappa \|x-y\|) 
\end{align*}
and the vector field $\alpha$ to be $ (y-x)/\|y-x\|^{\frac{d}{2} +\beta(x,y) + 1}$, with $0 < \underline{s} \leq s(x) \leq \overline{s} < 1$ (hence $0 < \underline{s} \leq \beta(x,y) \leq \overline{s} < 1$). We then have a nonlocal fractional operator (which corresponds to the negative fractional Laplacian when $s$ is constant) with a varying fraction $s(x)$ given by
 \begin{align*}
 \mathcal{L}^{s(\cdot)} u(x) :=& \int_{\mathbb{R}^d} \frac{\tilde{w}_{-2s}(x,y)(u(x)-u(y))}{\|x-y\|^{d+2\beta(x,y)}} dy\\
 =& \frac{1}{\pi^{\frac{d}{2}}} \int_{\mathbb R^d} \frac{2^{1-\frac{d}{2}+\beta(x,y)}}{\left|\Gamma(-\beta(x,y))\right|} \kappa^{\frac{d}{2}+\beta(x,y)} K_{\frac{d}{2}+\beta(x,y)}(\kappa \|x-y\|) \frac{(u(x)-u(y))}{\|x-y\|^{\frac{d}{2}+\beta(x,y)}}.
 \end{align*}
Considering the extension $\tilde{u}$ of $u$ on $\mathbb{R}^d$, where $u\in L_2(\mathcal{D})$, we have
\begin{align}
\tilde{A}_{s(\cdot)} u &:= \left[\kappa^{2s(x)} + \mathcal{L}^{s(\cdot)}\right] \tilde{u}= \kappa^{2s(x)} u(x) +   \int_{\mathbb{R}^d} \frac{\tilde{w}_{-2s}(x,y)(\tilde{u}(x)-\tilde{u}(y))}{\|x-y\|^{d+2\beta(x,y)}} dy.
\end{align}
\replaced[id=rev]{We then seek to solve a similar equation as in \eqref{eq:loc_frac_spde}. However, in practice, we truncate the exterior $\mathcal{D}^c$ to a bounded subdomain $\mathcal{D}^c_t \subset \mathcal{D}^c$; we require that $\mathcal{D}^c_t$ contains a neighborhood of $\partial\mathcal D$, so that $\overline{\mathcal D}\subset\mathcal G$ (see the beginning of \autoref{sec:weak_form} for the precise geometric setup). This truncation is both natural and necessary. On the modeling side, the Bessel factor $K_{\frac d2+\beta(x,y)}(\kappa\|x-y\|)$ in the kernel of $\tilde A_{s(\cdot)}$ decays exponentially, like $e^{-\kappa r}$ up to algebraic factors, as $r=\|x-y\|\to\infty$ (cf. \cite[Ch.~10]{Olver2010nist}). Interactions between points of $\mathcal D$ and points of the exterior located farther than a few correlation lengths $\kappa^{-1}$ away therefore contribute negligibly to the covariance on $\mathcal D$, and imposing the volume constraint on a sufficiently large buffer $\mathcal D^c_t$ alters the covariance on $\mathcal D$ only slightly; this is consistent with the observation that boundary effects in such models decay rapidly away from the boundary \cite{Potsepaev2010}. On the computational side, the sampling algorithm of \autoref{sec:FEM} requires a bounded, meshable domain. The buffer size $t$ is thus a modeling parameter balancing accuracy against computational cost, and the role of the truncation in the analysis is discussed in \sref{rem:truncation_motivation}.}{We then seek to solve a similar equation as in \eqref{eq:loc_frac_spde}. However, in practice, one should consider truncating $\mathcal{D}^c$ to some finite and bounded domain, $\mathcal{D}^c_t \subset \mathcal{D}^c$.} Let $\mathcal{G} = \mathcal{D} \cup \mathcal{D}^c_t$, we consider solving the following SPDE on $\mathcal{G}$
 \begin{align}
\label{eq:loc_frac_spde1}
\begin{array}{rll}
\tilde{A}_{s(\cdot)} u &=& \frac{1}{\mu}\mathcal{W}, \quad \text{ on }\mathcal{D},\\
u & = & 0, \quad \text{ on }\mathcal{D}^c_t.
\end{array}
\end{align}
Note that when $s$ is a constant, $\tilde{w}_{-2s} = w_{-2s}$ and $\tilde{A}_{s(\cdot)}$ is just $A_s$. Therefore, from now on, we focus on the general case where $s = s(x)$. 
 
\begin{revblock}
\begin{remark}\label{rem:truncation_motivation}
 The truncated problem \eqref{eq:loc_frac_spde1} should be viewed as an approximation of the whole-exterior problem whose covariance on $\mathcal D$ is expected to approach that of the untruncated model as the buffer size $t$ grows, in view of the exponential decay of the kernel discussed in \autoref{sec:model_form}.
\end{remark}
\end{revblock}

\replaced[id=rev]{We clarify how the present model relates to existing Whittle--Mat\'ern fields. (i) On the whole space with constant $s$, the operator $\tilde A_{s(\cdot)}$ coincides with the integral representation of $(\kappa^2-\Delta)^s$ on $\mathbb R^d$, so the model recovers the classical Mat\'ern covariance \eqref{eq:mat_cov} through \eqref{eq:frac_spde}. (ii) On the truncated domain, already with constant $s$ the model differs from the generalized Whittle--Mat\'ern fields of \eqref{eq:general_Lu=W}, which are defined through spectral fractional powers of the second-order elliptic operator $L=\kappa^2-\nabla\cdot(H\nabla)$. In contrast, our weak form involves the kernel integral over $\mathcal G\times\mathcal G$ with the volume constraint on $\mathcal D_t^c$, so the associated operator is of censored (regional) type on $\mathcal G$ rather than a spectral fractional power on $\mathcal D$; because of the truncation it also differs from the restricted integral fractional Laplacian on $\mathcal D$, which would carry the additional tail term $u(x)\int_{\mathbb R^d\setminus\mathcal G}\gamma(x,y)\,dy$ and a constraint on all of $\mathcal D^c$. These three realizations of the fractional Laplacian---spectral powers, restricted integral, and censored/regional---are different operators with different solutions (see, e.g., \cite{Grubb2015}). (iii) With spatially varying $s(x)$ the model is a new member of the broad class of generalized Whittle--Mat\'ern fields, understood as fields defined through a fractional-order operator equation driven by white noise as in \eqref{eq:general_Lu=W}; it is not a special case of the spectral SPDE models, and as the buffer grows, $\mathcal D_t^c\uparrow\mathcal D^c$, it is expected to approach the whole-exterior model (cf. \sref{rem:truncation_motivation}). Finally, we note that for non-convex $\mathcal D$ the covariance between two points depends on their Euclidean distance through the exterior, not on the intrinsic geometry of $\mathcal D$; this is inherent to kernel-based nonlocal models.}

%

\begin{revblock}
\begin{remark}\label{rem:nonlocal_dirichlet}
The condition $u=0$ on $\mathcal D_t^c$ in \eqref{eq:loc_frac_spde1} is the nonlocal Dirichlet (volume) constraint of the nonlocal calculus framework \cite{Du2013}, which our model formulation follows (see \autoref{sec:model_form}). For nonlocal operators with a kernel singularity $|x-y|^{-(d+2\beta)}$, interactions extend over positive distances, so data must be prescribed on sets of positive volume rather than on the boundary, which is a set of measure zero for the nonlocal pairing. Consistently, no boundary condition on $\partial\mathcal G$ is required in \eqref{eq:loc_frac_spde1}. The equation is posed on $\mathcal D$ and the constraint on $\mathcal D_t^c$, and $\partial\mathcal G$ does not enter the weak formulation of \autoref{sec:weak_form}. Since $\mathcal D_t^c$ contains a neighborhood of $\partial\mathcal D$ (see the beginning of \autoref{sec:weak_form}), the constraint region separates $\mathcal D$ from $\partial\mathcal G$. In the constant-order case this is the standard zero exterior condition of nonlocal volume-constrained problems \cite{Du2013,Acosta2017,Ainsworth2017}.
\end{remark}
\end{revblock}

\section{Weak formulation}
\label{sec:weak_form}
We now consider a weak form of \eqref{eq:loc_frac_spde1}. Let $\mathcal D\subset\mathbb R^d$ be bounded Lipschitz, let $\mathcal D^c_t$ be a bounded truncation of the exterior\added[id=rev]{ that contains a neighborhood of $\partial\mathcal D$, so that $\overline{\mathcal D}\subset\mathcal G$} and set $\mathcal G:=\mathcal D\cup\mathcal D^c_t$ (a bounded Lipschitz set). \added[id=rev]{Unless otherwise stated, we work under the following standing assumptions throughout the analysis. We assume $\mathcal D$ is bounded Lipschitz, $\mathcal D_t^c$ contains a neighborhood of $\partial\mathcal D$, $\kappa>0$, and $s\in C(\overline{\mathcal G})$ with $0<\underline s\le s(x)\le\overline s<1$ is log-H\"older continuous, that is, there exists $C_{\log}>0$ with
\begin{align}
\label{eq:log_holder}
|s(x)-s(y)|\,\log\Big(e+\frac{1}{|x-y|}\Big)\le C_{\log},\qquad x,y\in\overline{\mathcal G},\ x\neq y.
\end{align}
Any further hypotheses are stated where they are used.} Define the energy space $\mathbb V_{\kappa, s} := \{v\in L_2(\mathcal{G}) : \|v\|_{\mathbb V_{\kappa, s}}<\infty \text{ and } v = 0 ~\text{ a.e. in } \mathcal{D}^c_t\}$,
where the energy norm $\|.\|_{\mathbb V_{\kappa, s}}$ is given by $\|v\|_{\mathbb V_{\kappa, s}}^2 := \|\kappa^{s(x)}v\|_{L_2(\mathcal{G})}^2 + |v|_{\mathbb V_{\kappa, s}}^2$ with
\begin{align*}
    |v|_{\mathbb V_{\kappa, s}}^2 := \int_{\mathcal{G}}\int_{\mathcal{G}}
    (v(x) - v(y))^2\gamma(x,y) dy dx \quad\text{and}\quad
    \gamma(x,y) = \frac{1}{2}\frac{\tilde{w}_{-2s}(x,y)}{|x-y|^{d+2\beta(x,y)}}.
\end{align*}
We will prove that this space is a Hilbert space and that the standard variational tools apply. Define
\begin{align*}
(u,v)_{\mathbb V_{\kappa, s}} := \int_{\mathcal D}\kappa^{2s(x)} u(x) v(x)\,dx + \int_{\mathcal{G}}\int_{\mathcal{G}} (u(x)-u(y))(v(x)-v(y))\gamma(x,y)\,dy\,dx.
\end{align*}
Then $(\cdot,\cdot)_{\mathbb V_{\kappa, s}}$ is an inner product on $\mathbb V_{\kappa, s}$, $\|u\|_{\mathbb V_{\kappa, s}}^2=(u,u)_{\mathbb V_{\kappa, s}}$ and $\mathbb V_{\kappa, s}$ is a vector space. We now show that $(\mathbb V_{\kappa, s},\|\cdot\|_{\mathbb V_{\kappa, s}})$ is complete and thus a Hilbert space with inner product $(\cdot,\cdot)_{\mathbb V_{\kappa, s}}$.

\begin{theorem}[Completeness]\label{thm:V_complete}
The space $(\mathbb V_{\kappa, s},\|\cdot\|_{\mathbb V_{\kappa, s}})$ is complete.
\end{theorem}

\replaced[id=rev]{We next study the approximation of elements of $\mathbb V_{\kappa, s}$ by smooth, compactly supported functions.}{We also prove that $C_c^\infty(\mathcal G)\cap\{v=0\text{ a.e. in }\mathcal D^c_t\}$ is dense in $\mathbb V_{\kappa, s}$:}
\begin{revblock}
One point requires care here. Membership in $\mathbb V_{\kappa,s}$ forces only $\operatorname{supp}\bar v\subseteq\overline{\mathcal D}$, where $\bar v$ denotes the extension of $v$ by zero, and any smoothing procedure spreads the support slightly. Smooth approximants can therefore not be expected to satisfy the volume constraint exactly; it is recovered only outside a collar around $\partial\mathcal D$ whose measure tends to zero, and the statement below is formulated accordingly.
\end{revblock}
\begin{prop}[Smooth approximation]
\label{prop:density_moll}
\replaced[id=rev]{For $\varepsilon>0$ let $\mathcal D^+_\varepsilon:=\{x\in\mathbb R^d:\operatorname{dist}(x,\mathcal D)\le\varepsilon\}$ be the closed $\varepsilon$-neighbourhood of $\mathcal D$ and let $\mathcal C_\varepsilon:=\mathcal D^+_\varepsilon\setminus\mathcal D$ be the collar of width $\varepsilon$ around $\partial\mathcal D$, so that $|\mathcal C_\varepsilon|\to0$ as $\varepsilon\downarrow0$. Then, for every $v\in\mathbb V_{\kappa, s}$ and every sufficiently small $\varepsilon>0$, there exists $v_\varepsilon\in C_c^\infty(\mathcal G)$ with $\operatorname{supp}v_\varepsilon\subseteq\mathcal D^+_\varepsilon$, hence vanishing on $\mathcal D^c_t\setminus\mathcal C_\varepsilon$, such that $\|v_\varepsilon-v\|_{\mathbb V_{\kappa, s}}\to0$ as $\varepsilon\downarrow0$, the energy norm being understood as the expression defining it, which is meaningful for any function on $\mathcal G$.}{For every $v\in\mathbb V_{\kappa, s}$ there exists a sequence $v_\varepsilon\in C_c^\infty(\mathcal G)$ with $v_\varepsilon=0$ a.e. on $\mathcal D^c_t$ and $\|v_\varepsilon-v\|_{\mathbb V_{\kappa, s}}\to0$.}
\end{prop}

Next we show the equivalence between $\mathbb V_{\kappa,s}$ and $\mathbb V_{\kappa',s}$ for $\kappa\neq \kappa'$.

\begin{theorem}[Equivalence of energy spaces for different $\kappa$]\label{thm:V_kappa_equiv}
Let $\mathcal G$ be as before. Assume $s\in C(\overline{\mathcal G}),\, 0<\underline s\le s(x)\le\overline s<1$, and fix two positive constants $\kappa,\kappa'>0$, $\kappa\neq \kappa'$. Let $\mathbb V_{\kappa,s}$ and $\mathbb V_{\kappa',s}$ denote the energy spaces built from the kernels $\gamma_{\kappa}(x,y)$ and $\gamma_{\kappa'}(x,y)$ corresponding to parameters $\kappa$ and $\kappa'$ (same $s$), and with energy norms
\begin{align*}
\|v\|_{\mathbb V_{\kappa,s}}^2=\|\kappa^{s(\cdot)} v\|_{L_2(\mathcal G)}^2 + 
\iint_{\mathcal G\times\mathcal G} (v(x)-v(y))^2 \gamma_{\kappa}(x,y)\,dy\,dx,
\end{align*}
(and analogously for $\kappa'$). Then there exist positive constants $C_1,C_2$ (depending on $d,\underline s,\overline s,\kappa,\kappa'$ and $|\mathcal G|$) such that for every $v\in L_2(\mathcal G)$, $
C_1\,\|v\|_{\mathbb V_{\kappa',s}} \le \|v\|_{\mathbb V_{\kappa,s}} \le C_2\,\|v\|_{\mathbb V_{\kappa',s}}$. In other words $\mathbb V_{\kappa,s}\simeq \mathbb V_{\kappa',s}$.  ⁠
\end{theorem}

\begin{remark}
\label{rem:V_kappa_equiv}
The constants in the comparability bounds of \sref{thm:V_kappa_equiv} may deteriorate (possibly blow up) as 
$\underline s\downarrow0$ or $\overline s\uparrow1$, or as $|\mathcal G|\to\infty$. The boundedness of $\mathcal G$ is essential; on unbounded domains the far-field exponential factors $e^{-\kappa r}$ and $e^{-\kappa' r}$ that appear in the proof are not uniformly comparable and equivalence may fail. If $\kappa$ and $\kappa'$ vary in space the same conclusion holds provided they are uniformly comparable,
i.e. $0<m_0\le (\kappa(x)/\kappa'(x))^{s(x)}\le M_0<\infty$ for all $x\in\mathcal G$.
\added[id=rev]{The role of the bounded truncated domain in the analysis is discussed further in \sref{rem:truncation_motivation}.}
\end{remark}

Next we define a bilinear form $\mathcal{A}^{s(\cdot)}(u,v)$ \replaced[id=rev]{which is proved to be}{that can be easily shown to be} symmetric, continuous and coercive (see the proof of \sref{thm:form_rep}). Let $v\in \mathbb V_{\kappa, s}$ be a test function, we have that
\begin{align*}
        \int_{\mathcal{G}} \tilde{A}_{s(\cdot)} u(x) v(x) dx & = \int_{\mathcal{D}} \kappa^{2s(x)} u(x) v(x) dx +2\int_{\mathcal{G}}\int_{\mathcal{G}}(u(x)-u(y))\gamma(x,y) v(x) dy dx
    \end{align*}
and by switching the variables in the double integral on the right hand side of the above equation and noting that $\gamma(x,y)$ is symmetric
\begin{align*}
        \int_{\mathcal{G}} \tilde{A}_{s(\cdot)} u(x) v(x) dx & = \int_{\mathcal{D}} \kappa^{2s(x)} u(x) v(x) dx -2\int_{\mathcal{G}}\int_{\mathcal{G}}(u(x)-u(y))\gamma(x,y) v(y) dx dy.
    \end{align*}
Summing these two equations and dividing by 2, we obtain
\begin{align*}
\int_{\mathcal{G}} \tilde{A}_{s(\cdot)} u(x) v(x) \, dx 
&= \int_{\mathcal{D}} \kappa^{2s(x)} u(x) v(x) \, dx + \int_{\mathcal{G}} \int_{\mathcal{G}} 
  (u(x) - u(y))(v(x) - v(y)) \gamma(x,y) \, dy \, dx 
\nonumber \\
&=: \mathcal{A}^{s(\cdot)}(u,v), \quad\qquad (u,v)\in \mathbb V_{\kappa,s} \times \mathbb V_{\kappa,s}
\end{align*}

The weak formulation is to find $u\in\mathbb V_{\kappa, s}$ satisfying
\begin{align}
\label{eq:weak}
    \mathcal A^{s(\cdot)}(u,v) = \frac{1}{\mu}\, \mathcal{W}(v) \quad \forall v\in \mathbb V_{\kappa, s}.
\end{align}

\begin{theorem}[Form representation]\label{thm:form_rep}
Let $\mathcal A^{s(\cdot)}$ and $\mathbb V_{\kappa, s}$ be as above. Then there exists a unique positive self-adjoint operator $A$ on $L_2(\mathcal G)$ with domain $\operatorname{Dom}(A)\subset L_2(\mathcal G)$ such that
\begin{enumerate}[(i)]
\item $\operatorname{Dom}(A^{1/2})=\mathbb V_{\kappa, s}$, $\langle A^{1/2}u,A^{1/2}v\rangle_{L_2(\mathcal G)}=(u,v)_{\mathbb V_{\kappa, s}}$ and $\|A^{1/2}u\|_{L_2(\mathcal G)}=\|u\|_{\mathbb V_{\kappa, s}}$ $\forall u,v\in\mathbb V_{\kappa, s}$.
\item For every deterministic $f\in L_2(\mathcal G)$ the variational problem
\begin{align*}
\mathcal A^{s(\cdot)}(u,v)=\langle f,v\rangle_{L_2(\mathcal G)}\qquad\forall v\in\mathbb V_{\kappa, s}
\end{align*}
has a unique solution $u\in\mathbb V_{\kappa, s}$, namely $u=A^{-1}f$. More generally, for every continuous linear functional $F\in\mathbb V_{\kappa, s}'$ there exists a unique $u\in\mathbb V_{\kappa, s}$ such that $
\mathcal A^{s(\cdot)}(u,v)=\langle F,v\rangle_{\mathbb V_{\kappa, s}',\mathbb V_{\kappa, s}}$, $\forall v\in\mathbb V_{\kappa, s}$,
and this solution is given by $u=A^{-1}F$ (here $A$ is understood as the isomorphism $A:\mathbb V_{\kappa, s}\to\mathbb V_{\kappa, s}'$ induced by the form).
\label{eq:gen_sol_def}
\end{enumerate}
\end{theorem}
The stochastic case, where the right-hand side is the white noise $\mathcal W$, is treated in \sref{prop:gen_weak_solution_rhs=W}. We defer this result because its proof requires embedding and regularity properties of the energy space (most importantly the continuous embedding $\mathbb V_{\kappa,s} \hookrightarrow L_2(\mathcal G)$) which will be established later.

We next compare the nonlocal energy form $\mathcal A^{s(\cdot)}(\cdot,\cdot)$ with standard fractional Sobolev forms associated with fractional Laplacian with constant fractional order. Using the near-/far-field bounds for the kernel from \sref{lem:gamma_two_regime} in the appendix, we obtain two-sided form inequalities which imply a continuous embedding $\mathbb V_{\kappa, s}\hookrightarrow H^{\overline s}(\mathcal G)$.

\begin{prop}\label{prop:form_compare_compact}
Under the standing assumptions (in particular $s\in C(\overline{\mathcal G})$, $0<\underline s\le s\le\overline s<1$, and $\kappa>0$), 
there exist constants $c_1',c_2'>0$ such that \replaced[id=rev]{for all $v\in L_2(\mathcal G)$,}{for all $v\in C_c^\infty(\mathcal G)$}
\begin{align}\label{eq:form_two_sided}
c_1' \mathcal B_-(v,v) \;\le\; \mathcal A^{s(\cdot)}(v,v) \;\le\; c_2' \mathcal B_+(v,v),
\end{align}
where $\mathcal B_-(v,v):=|v|_{H^{\underline s}(\mathcal G)}^2 + \|v\|_{L_2(\mathcal G)}^2$, 
$\mathcal B_+(v,v):=|v|_{H^{\overline s}(\mathcal G)}^2 + \|v\|_{L_2(\mathcal G)}^2$,
and the fractional seminorms
\begin{align*}
|v|_{H^{\underline s}(\mathcal G)}^2
= \iint_{\mathcal G\times\mathcal G} \frac{(v(x)-v(y))^2}{r^{d+2\underline s}}\,dy\,dx, \quad |v|_{H^{\overline s}(\mathcal G)}^2
= \iint_{\mathcal G\times\mathcal G} \frac{(v(x)-v(y))^2}{r^{d+2\overline s}}\,dy\,dx.
\end{align*}
In particular the left inequality in \eqref{eq:form_two_sided} yields the continuous embedding $
\mathbb V_{\kappa, s} \hookrightarrow H^{\underline s}(\mathcal G)$,
hence (by Rellich--Kondrachov) the embedding $\mathbb V_{\kappa, s}\hookrightarrow L_2(\mathcal G)$ is compact. 
\end{prop}
Consequently the operator $A$ associated with $\mathcal A^{s(\cdot)}$ has compact resolvent and a discrete spectrum consisting of eigenvalues of finite multiplicity with an orthonormal eigenbasis in $L_2(\mathcal G)$.

\begin{remark}
\label{rem:continuity_s}
Continuity of $s$ throughout the article is a convenient technical assumption but is not essential for the variational formulation or for the main existence/uniqueness arguments. With modest additional effort (approximation of $s$ or working with essential bounds and form convergence) all main results extend to measurable $s$ satisfying $0<\underline s\le s(x)\le\overline s<1$ a.e.. \added[id=rev]{The exceptions are the results whose proofs compare two distinct values of $s$, namely \sref{prop:density_moll} and \sref{prop:reg_variable}, which genuinely use the log-H\"older condition \eqref{eq:log_holder}.}
\end{remark}

In \sref{thm:form_rep} we proved the existence of a unique solution when the forcing is deterministic. Next we construct a generalized Gaussian weak solution when the forcing is the white noise $\mathcal W$. The meaning of the identity $u(\varphi)=\frac{1}{\mu}\,\mathcal W\big(A^{-1}\varphi\big)$ relies on the mapping properties of $A^{-1}$ together with the continuous embedding of $\mathbb V_{\kappa,s}$ into $L_2(\mathcal G)$, which guarantee that $A^{-1}\varphi$ belongs to the domain on which $\mathcal W$ acts and that the resulting random variables are well defined.
\begin{prop}[Stochastic RHS]\label{prop:gen_weak_solution_rhs=W}
Assume the hypotheses of \sref{thm:form_rep} and let $A$ be the positive self-adjoint operator associated with the coercive form $\mathcal A^{s(\cdot)}$ on $L^2(\mathcal G)$.  Define the map
\begin{align}
\label{eq:gen_weak_sol}
u : L^2(\mathcal G)\longrightarrow L^2(\Omega;\mathbb R),\qquad
u(\varphi) := \frac{1}{\mu}\,\mathcal W\big(A^{-1}\varphi\big).
\end{align}
Then $u$ is a centered Gaussian generalized random field, in the sense that for every $v\in\mathbb V_{\kappa,s}$ the two $L^2(\Omega;\mathbb R)$--random variables $\langle A u, v\rangle_{\mathbb V_{\kappa,s}',\mathbb V_{\kappa,s}}$, and $\frac{1}{\mu}\,\mathcal W(v)$ are equal. Equivalently, the map $\mathbb V_{\kappa,s}\ni v \longmapsto \langle A u, v\rangle_{\mathbb V_{\kappa,s}',\mathbb V_{\kappa,s}} \in L^2(\Omega;\mathbb R)$ is a well-defined continuous linear mapping and coincides with the map $\mathbb V_{\kappa,s}\ni v \longmapsto \frac{1}{\mu}\,\mathcal W(v)\in L^2(\Omega;\mathbb R).$
\end{prop}

Let $(\lambda_j)_{j\ge1}$ denote the eigenvalues of $A$ in nondecreasing order, counted with multiplicity. \replaced[id=rev]{Then, by}{Then one can easily show, e.g. by} standard min-max (Courant--Fischer) inequalities (see e.g. \cite[Thm. XIII.1]{Reed1978}), the corresponding eigenvalues satisfy, for all $j$, $
c_1 \mu_j^- \le \lambda_j \le c_2 \mu_j^+$,
where $(\mu_j^\pm)$ are eigenvalues for the self-adjoint operators $B_{\pm}$ associated to $\mathcal B_\pm$. Weyl asymptotics for constant-order fractional Laplacians on bounded Lipschitz domains \cite[Theorem 5.3 and Example 5.4]{Grubb2023} (see also \cite{Grubb2015}) yield $\mu_j^- \simeq j^{2\underline s/d}$, and $\mu_j^+ \simeq j^{2\overline s/d}$,
meaning there exist positive constants $C_1,C_2, C_1', C_2'$ and $j_0$ such that for $j\ge j_0$,
\begin{align*}
C_1 j^{2\underline s/d} \le \mu_j^- \le C_2 j^{2\underline s/d},\quad
C_1' j^{2\overline s/d} \le \mu_j^+ \le C_2' j^{2\overline s/d},
\end{align*}
which implies that there exist constants $C_-,C_+>0$ such that for all $j\ge j_0$,
\begin{align}
\label{eq:lambda_growth}
C_- j^{2\underline s/d} \le \lambda_j \le C_+ j^{2\overline s/d}.
\end{align}

This last inequality (in particular the left inequality) will help us prove that the weak solution $u\in H^r(\mathcal G)$ almost surely for all nonnegative reals $r$ with $r \;<\; 2\underline s - \frac{d}{2}$ as in the below theorem.

\begin{theorem}[Spectral/generalized Gaussian solution and Sobolev regularity]\label{thm:spectral_underline_s}
Let $\mathcal A^{s(\cdot)}$ be the closed, symmetric, coercive form on $L_2(\mathcal G)$ with form domain
$\mathbb V_{\kappa, s}=\operatorname{Dom}(\mathcal A^{s(\cdot)})=\operatorname{Dom}(A^{1/2})$ with $A$ is the associated positive self-adjoint operator from \sref{thm:form_rep}.
Denote $A$'s eigenpairs by $0<\lambda_1 \le\lambda_2\le\cdots$, $ A\phi_j= \lambda_j\phi_j$, where $(\phi_j)_{j\ge1}$ is an orthonormal basis of $L_2(\mathcal G)$ with respect to the inner product $\langle\cdot,\cdot\rangle_{L_2(\mathcal G)}$. Let $\mathcal W$ be Gaussian white noise on $\mathcal G$ (as in \sref{def:LN_white}) and set $\xi_j:=\mathcal W(\phi_j)$ so that $\xi_j\stackrel{i.i.d.}{\sim}N(0,1)$ as real random variables in $L_2(\Omega; \mathbb R)$. Then the following hold.

\begin{enumerate}[(i)]
\item (Spectral generalized weak solution.) The series $ u \;:=\; \frac{1}{\mu}\sum_{j=1}^\infty \xi_j\,\lambda_j^{-1}\,\phi_j$
defines a centered Gaussian generalized random field, in the sense that for every $\varphi\in L_2(\mathcal G)$ the scalar series (defined as the action of $u$ on $\varphi \in L_2(\mathcal G)$) $
u(\varphi) \;:=\; \frac{1}{\mu}\sum_{j=1}^\infty \xi_j \lambda_j^{-1}\,\langle\phi_j,\varphi\rangle_{L_2(\mathcal G)}$ converges in $L_2(\Omega; \mathbb R)$ and the map $\varphi\mapsto u(\varphi)$ is linear and continuous from $L_2(\mathcal G)$ into $L_2(\Omega; \mathbb R)$.

\item (Sobolev regularity.) \replaced[id=rev]{Set $r_\star := 2\underline s - d/2$ if $d\ge2$, or if $d=1$ and $\underline s\le 1/2$. If $d=1$ and $\underline s>1/2$, assume in addition that $\mathcal D$ and $\mathcal G$ are $C^\infty$ domains and set
\begin{equation}\label{eq:r_star}
r_\star := \underline s+(r_g-\underline s)\Big(1-\frac{1}{2\underline s}\Big),\qquad r_g:=\min\big(2\underline s,\,\underline s+\tfrac12\big).
\end{equation}
Then for every real $r<r_\star$, one has $\mathbb E\big[\|u\|_{H^r(\mathcal G)}^2\big]<\infty$, hence the spectral generalized weak solution $u\in H^r(\mathcal G)$ almost surely. In particular, $u\in L_2(\mathcal G)$ a.s.\ whenever $\underline s>d/4$.}{For every real $r$ with $r \;<\; 2\underline s - d/2$, one has $\mathbb E\big[\|u\|_{H^r(\mathcal G)}^2\big]<\infty$, hence the spectral generalized weak solution $u\in H^r(\mathcal G)$ almost surely. In particular, $u\in L_2(\mathcal G)$ a.s.\ whenever $\underline s>d/4$.}

\item (Equation in the generalized sense and uniqueness.) Assume $\underline s>d/4$. Then, the weak solution $u$ satisfies $
A u \;=\; (1/\mu)\mathcal W$
in the sense of generalized random fields, meaning that \replaced[id=rev]{for every $\varphi\in \operatorname{Dom}(A)$ the identity
$u(A\varphi) \;=\; (1/\mu)\,\mathcal W(\varphi)$}{for every $\varphi\in L_2(\mathcal G)$ the identity $
\langle A u,\varphi\rangle_{L_2(\mathcal G)} \;=\; (1/\mu)\,\mathcal W(\varphi)$}
holds in $L_2(\Omega; \mathbb R)$\added[id=rev]{, where $u(\cdot)$ denotes the action of the generalized field $u$ from (i); note that $Au$ itself need not be a function, so the equation is transferred to the test function through the self-adjointness of $A$}. Moreover, $u$ is the unique centered Gaussian generalized solution of this equation.

\end{enumerate}
\end{theorem}

For each fixed $j$ lets define the linear functional $F_j\in\mathbb V_{\kappa,s}'$ by $
F_j(v):=\langle \phi_j,v\rangle_{L_2(\mathcal G)}$, with $v\in\mathbb V_{\kappa,s}$ and $\phi_j$ the orthonormal basis in \sref{thm:spectral_underline_s}. This functional is linear and continuous on $\mathbb V_{\kappa,s}$. Indeed, the embedding 
$\mathbb V_{\kappa,s}\hookrightarrow L^2(\mathcal G)$ is continuous as $\|v\|_{L^2(\mathcal G)}^2 \le \kappa^{-2\underline s}\|\kappa^{s(\cdot)}v\|_{L^2(\mathcal G)}^2
\le \kappa^{-2\underline s}\|v\|_{\mathbb V_{\kappa,s}}^2$,
so for every $v\in\mathbb V_{\kappa,s}$, $
|F_j(v)|
= |\langle \phi_j,v\rangle_{L^2(\mathcal G)}|
\le \|\phi_j\|_{L^2(\mathcal G)}\,\|v\|_{L^2(\mathcal G)}
\le \kappa^{-\underline s}\|\phi_j\|_{L^2(\mathcal G)}\,\|v\|_{\mathbb V_{\kappa,s}} = \kappa^{-\underline s}\,\|v\|_{\mathbb V_{\kappa,s}}.
$ Hence $F_j\in\mathbb V_{\kappa,s}'$ with 
$\|F_j\|_{\mathbb V_{\kappa,s}'} \le \kappa^{-\underline s}$. We can get an exact equality because there exists a unique $w_j\in\mathbb V_{\kappa,s}$ such that $F_j(v)=(w_j,v)_{\mathbb V_{\kappa,s}}= \langle A^{1/2} w_j, A^{1/2} v\rangle_{L_2(\mathcal G)}
= \langle A w_j, v\rangle_{L_2(\mathcal G)}$ \replaced[id=rev]{by the Riesz representation theorem in the Hilbert space $\mathbb V_{\kappa,s}$ (with the energy inner product)}{by Riesz representation theorem} and \sref{thm:form_rep}(i). Thus $A w_j=\phi_j$ (as elements of $L_2(\mathcal G)$), and hence $
w_j = A^{-1}\phi_j = \lambda_j^{-1}\phi_j$.
Consequently, \replaced[id=rev]{again by the Riesz representation theorem in $\mathbb V_{\kappa,s}$}{by Riesz representation theorem}, the dual norm of $F_j$ equals the $\mathbb V_{\kappa,s}$-norm of $w_j$, namely $\|F_j\|_{\mathbb V_{\kappa,s}'} = \|w_j\|_{\mathbb V_{\kappa,s}}= \|A^{1/2}w_j\|_{L_2(\mathcal G)}
= \|\lambda_j^{-1/2}\phi_j\|_{L_2(\mathcal G)}
= \lambda_j^{-1/2} 
$.
Let $\{\xi_j\}$ be as in \sref{thm:spectral_underline_s}. Let $\mathcal W_N \;:=\; \sum_{j=1}^N \xi_j \phi_j$, be the finite projection of white noise onto the span$\{\phi_1,\dots,\phi_N\}$.  
As a linear functional on $\mathbb V_{\kappa,s}$ we have $\mathcal W_N(v) \;=\; \sum_{j=1}^N \xi_j F_j(v)$. Its Riesz representer in $\mathbb V_{\kappa,s}$ is therefore $w_N \;=\; \sum_{j=1}^N \xi_j w_j \;=\; \sum_{j=1}^N \xi_j \lambda_j^{-1}\phi_j$, and the dual norm satisfies $
\|\mathcal W_N\|_{\mathbb V_{\kappa,s}'}^2 \;=\; \|w_N\|_{\mathbb V_{\kappa,s}}^2 = \big\langle A w_N, w_N\big\rangle_{L_2(\mathcal G)}.
$
Since $A w_N = \sum_{j=1}^N \xi_j \lambda_j^{-1} A\phi_j  = \sum_{j=1}^N \xi_j \phi_j$, we have $\|\mathcal W_N\|_{\mathbb V_{\kappa,s}'}^2=\sum_{j=1}^N \xi_j^2 \lambda_j^{-1}$. Taking the expectation, we obtain $\mathbb E\!\left[\|\mathcal W_N\|_{\mathbb V_{\kappa,s}'}^2\right]=\sum_{j=1}^N \lambda_j^{-1}$. By monotone convergence and \eqref{eq:lambda_growth},  
$\mathbb E\!\left[\|\mathcal W\|_{\mathbb V_{\kappa,s}'}^2\right]=\sum_{j=1}^\infty \lambda_j^{-1} < \infty \quad\text{provided }\underline{s}>d/2$.
\added[id=rev]{Conversely, $\|\mathcal W\|_{\mathbb V_{\kappa,s}'}^2=\sum_{j\ge1}\xi_j^2\lambda_j^{-1}$ is a series of independent nonnegative random variables, so by the Kolmogorov three-series theorem it converges almost surely if and only if $\sum_{j\ge1}\lambda_j^{-1}<\infty$; when the sum diverges, $\mathcal W\notin\mathbb V_{\kappa,s}'$ almost surely. Divergence is guaranteed when $\overline s\le d/2$, since the upper bound in \eqref{eq:form_two_sided} yields $\lambda_j\lesssim j^{2\overline s/d}$ by the min--max principle.}
This is summarized in the following theorem

\begin{theorem}[Spectral criterion for $\mathcal W\in\mathbb V_{\kappa,s}'$]
\label{thm:W_in_V_prime}
\added[id=rev]{Assume $d=1$.} Assume the standing hypotheses ensuring that
$\mathbb V_{\kappa,s}\hookrightarrow L_2(\mathcal G)$ is compact and that the
associated self-adjoint operator $A$ has discrete spectrum
$0<\lambda_1\le\lambda_2\le\cdots$ with $L_2$-orthonormal eigenfunctions
$(\phi_j)_{j\ge1}$. Let $\mathcal W$ be Gaussian white noise on $\mathcal G$ in the sense of \sref{def:LN_white}.
\added[id=rev]{Then the following hold.}
\begin{enumerate}[(i)]
\item \added[id=rev]{$\mathcal W\in\mathbb V_{\kappa,s}'$ almost surely if and only if $\sum_{j=1}^\infty \lambda_j^{-1}<\infty$, in which case $\mathbb E\!\left[\|\mathcal W\|_{\mathbb V_{\kappa,s}'}^2\right]=\sum_{j=1}^\infty \lambda_j^{-1}$.}
\item \added[id=rev]{If $\underline s>\tfrac d2$, then $\sum_{j=1}^\infty \lambda_j^{-1}<\infty$, hence $\mathcal W\in\mathbb V_{\kappa,s}'$ almost surely.}
\item \added[id=rev]{If $\overline s\le\tfrac d2$, then $\sum_{j=1}^\infty \lambda_j^{-1}=\infty$, hence $\mathcal W\notin\mathbb V_{\kappa,s}'$ almost surely.}
\end{enumerate}
\deleted[id=rev]{Then $\mathcal W\in\mathbb V_{\kappa,s}'$ almost surely if and only if $\underline s > \frac{d}{2}$. Moreover, $\mathbb E\!\left[\|\mathcal W\|_{\mathbb V_{\kappa,s}'}^2\right]
\;=\;
\sum_{j=1}^\infty \lambda_j^{-1} <\infty  \text{ provided }\underline s > \frac{d}{2}.
$}
\end{theorem}

\begin{revblock}
In the intermediate regime $\underline s\le\tfrac d2<\overline s$, the two-sided spectral bounds $j^{2\underline s/d}\lesssim\lambda_j\lesssim j^{2\overline s/d}$ decide neither convergence nor divergence of $\sum_j\lambda_j^{-1}$; whether $\mathcal W\in\mathbb V_{\kappa,s}'$ then depends on the fine structure of $s(\cdot)$, for instance on the measure of the region where $s$ is close to $\underline s$.
\end{revblock}

In summary, there exists a version
$\omega\mapsto \mathcal W(\omega)$ with $\phi \mapsto \langle\phi,\mathcal W(\omega)\rangle_{L_2(\mathcal G)}$ a bounded linear functional on $L_2(\mathcal G)$
for almost every $\omega\in\Omega$ \replaced[id=rev]{under the sufficient condition $\underline s>d/2$, which can only hold when $d=1$.}{only under the stronger condition $\underline s>d/2$, which is only true when $d=1$.}

\begin{prop}
\label{prop:pathwise_coincidence}
Assume \added[id=rev]{$d=1$ and} $\underline s>d/2$, then the action of the generalized weak solution $u$ obtained in \sref{prop:gen_weak_solution_rhs=W} on a test function $\varphi \in L_2(\mathcal G)$ is equal to $\langle u, \varphi \rangle_{L_2(\mathcal G)}$ a.s. 
\end{prop}

\begin{revblock}
\begin{remark}\label{rem:cov_symmetry}
The Mat\'ern covariance \eqref{eq:mat_cov} is symmetric, $\varrho(x_1,x_2)=\varrho(x_2,x_1)$, and the variable-smoothness model retains this property. Indeed, the kernel $\gamma(x,y)$ defining the energy seminorm in \autoref{sec:weak_form} is symmetric by construction, since $\beta(x,y)=(s(x)+s(y))/2$ is symmetric and $\tilde w_{-2s}(x,y)$ depends on $(x,y)$ only through $\beta(x,y)$ and $\|x-y\|$; the bilinear form $\mathcal A^{s(\cdot)}$ is therefore symmetric and the associated operator $A$ of \sref{thm:form_rep} is positive self-adjoint. The covariance functional of the generalized solution $u$ of \sref{thm:spectral_underline_s} thus satisfies
\begin{align*}
\mathbb E\big[u(\varphi)\,u(\psi)\big] \;=\; \frac{1}{\mu^2}\,\big\langle A^{-1}\varphi,\,A^{-1}\psi\big\rangle_{L_2(\mathcal G)} \;=\; \mathbb E\big[u(\psi)\,u(\varphi)\big],\qquad \varphi,\psi\in L_2(\mathcal G).
\end{align*}
In particular, whenever $\underline s>d/4$ so that $u\in L_2(\mathcal G)$ almost surely, the covariance function $\varrho(x,y):=\mathbb E[u(x)u(y)]$, defined for almost every $(x,y)\in\mathcal G\times\mathcal G$, satisfies $\varrho(x,y)=\varrho(y,x)$, as is also apparent from its spectral representation $\varrho(x,y)=\frac1{\mu^2}\sum_{j\ge1}\lambda_j^{-2}\phi_j(x)\phi_j(y)$. For non-constant $s(x)$ the field is non-stationary, i.e., $\varrho$ is no longer a function of $x-y$ alone, but the symmetry is structural---it follows from the self-adjointness of $A$ and it reduces to the Mat\'ern symmetry when $s$ is constant.
\end{remark}
\end{revblock}

\begin{revblock}
\begin{remark}\label{rem:riesz_gelfand}
The variational framework uses the Riesz representation theorem in two different Hilbert spaces, which we summarize through the Gelfand triple $\mathbb V_{\kappa,s}\hookrightarrow L_2(\mathcal G)\cong L_2(\mathcal G)'\hookrightarrow\mathbb V_{\kappa,s}'$. (i) Riesz representation in $L_2(\mathcal G)$ identifies $L_2(\mathcal G)$ with its dual, so that $L_2(\mathcal G)$ is canonically embedded into $\mathbb V_{\kappa,s}'$; together with Kato's first representation theorem this yields the $L_2$-realization $A$ of \sref{thm:form_rep}. (ii) Riesz representation in $\mathbb V_{\kappa,s}$ with respect to the energy inner product identifies $\mathbb V_{\kappa,s}$ with its dual and yields the isometry $A:\mathbb V_{\kappa,s}\to\mathbb V_{\kappa,s}'$, $Au:=\mathcal A^{s(\cdot)}(u,\cdot)$, which is an isomorphism by coercivity (Lax--Milgram); on $\operatorname{Dom}(A)$ this coincides with the $L_2$-realization from (i). Consequently, $A^{-1}$ extends from $L_2(\mathcal G)$ to the isomorphism $A^{-1}:\mathbb V_{\kappa,s}'\to\mathbb V_{\kappa,s}$, and $u=\tfrac1\mu A^{-1}F$ is well defined for every $F\in\mathbb V_{\kappa,s}'$. For the stochastic right-hand side this gives two consistent constructions. If $\underline s>d/2$, then $\mathcal W\in\mathbb V_{\kappa,s}'$ almost surely (\sref{thm:W_in_V_prime}) and $u=\tfrac1\mu A^{-1}\mathcal W$ pathwise (\sref{prop:pathwise_coincidence}); if $\underline s\le d/2$, the same object is defined as a generalized field by $u(\varphi):=\tfrac1\mu\mathcal W(A^{-1}\varphi)$ (\sref{prop:gen_weak_solution_rhs=W}), which requires no pointwise extension of $A^{-1}$ since $A^{-1}\varphi\in L_2(\mathcal G)$ for $\varphi\in L_2(\mathcal G)$. The Riesz arguments used for the functionals $F_j$ and $\mathcal W_N$ above are all of type (ii).
\end{remark}
\end{revblock}

\section{Finite element approximation}
\label{sec:FEM}

The weak form \eqref{eq:weak} involves a heterogeneous, spatially singular kernel and nonlocal volume constraints that couple each interior degree of freedom to the whole domain.  A direct finite element discretization therefore produces a \emph{dense} stiffness matrix, since storing the full matrix and performing standard matrix--vector operations would incur $\mathcal O(N^2)$ memory and $\mathcal O(N^2)$ (or worse $\mathcal O(N^3)$) computational cost for $N$ degrees of freedom, which quickly becomes prohibitive in higher dimensions.  Addressing this bottleneck calls for tailored spatial and algebraic compression strategies (e.g.\ kernel localization / truncation, sparse approximations, hierarchical matrices and fast multipole-type techniques) and a specially designed FEM assembly/solver pipeline. The development and analysis of such scalable methods is substantial and is the subject of a companion article \cite{Ruzayqat2026sampling}. In the present work we therefore employ a \replaced[id=rev]{standard}{straightforward} FEM implementation as a proof-of-concept and \replaced[id=rev]{specialize the finite element implementation and the numerical experiments in}{restrict numerical experiments in} \autoref{sec:numerics} to the one-dimensional setting\added[id=rev]{ ($d=1$)}\replaced[id=rev]{, where the computational cost remains manageable and we}{, where the computational cost remains manageable and one can clearly} illustrate the model behaviour and convergence properties.

\subsection{Discrete bilinear form}
\label{subsec:discrete_bilinear}
We consider a finite element mesh that discretizes both the interior and exterior polytope regions, $\mathcal{D}$ and $\mathcal{D}^c_t$. Let $\mathcal{T} := \mathcal{T}(\mathcal{G})$ denote the resulting mesh on $\mathcal{G} \subset \mathbb R$. Let $\mathcal{T}^{\text{int}}$ and $\mathcal{T}^{\text{ext}}$ be the respective subdivisions of the interior and exterior domains. The set of vertices associated with $\mathcal{T}$ is denoted as $\mathcal{N} = \{ \mathbf{x}_i \}_{i=1}^{N_{all}}$. Among these, we label the first $N$ nodes (with $N < N_{all}$) as the interior nodes belonging to $\mathcal{D}$. Let $\mathbb V_{\kappa,s}(\mathcal{T}) \subset \mathbb V_{\kappa,s}$ be the finite element space of continuous, piecewise polynomial functions of order $p$ defined over $\mathcal{T}$. For each node $\mathbf{x}_i \in \mathcal{N}$, we denote the corresponding nodal basis function by $\psi_i \in \mathbb V_{\kappa,s}(\mathcal{T})$. Using this notation, the discrete solution on $\mathcal D$ can be expressed as $U(x) = \sum_{i=1}^{N} u_i \psi_i(x)$, where the coefficient vector is given by $\underline{U} = (u_1, u_2, \dots, u_N)^T \in \mathbb{R}^{N}$. 
The discrete bilinear form may then be written as

\begin{align*}
\mathcal{A}(U,V) = \sum_{\tau\in\mathcal T^{int}} \mathcal A_{\tau}^1(U,V)+ \sum_{\tau,\tau'\in\mathcal T} \mathcal A_{\tau,\tau'}^2(U,V) = \sum_{\tau \in \mathcal T^{int}} b_\tau(V), \qquad \forall \, V\in \mathbb V_{\kappa, s}(\mathcal T),
\end{align*}
where 
{\allowdisplaybreaks
\begin{align}
    \mathcal A_{\tau}^1(U,V) &:= \int_{\tau\in\mathcal T^{int}} \kappa^{2s(x)} U(x) V(x) dx,\label{eq:loc-mass} \\
   \mathcal A_{\tau,\tau'}^2(U,V) & := \int_{\tau\in\mathcal T}\int_{\tau'\in\mathcal T} (U(x)-U(y))(V(x)-V(y))\gamma(x,y) dy dx, \label{eq:loc-stiff} \\
b_{\tau}(V) &:=\frac{1}{\mu}\,\mathcal W(V),
\label{eq:loc-rhs}
\end{align}
}
and the term $\mathcal W(V)$ as defined and computed in \autoref{subsec:sampling}.

\begin{revblock}
\begin{prop}[Discrete well-posedness]\label{prop:discrete_wellposed}
For almost every $\omega\in\Omega$ the discrete problem above admits a unique solution $U\in\mathbb V_{\kappa,s}(\mathcal T)$; equivalently, the stiffness matrix $A=A^1+A^2$ assembled from \eqref{eq:loc-mass} and \eqref{eq:loc-stiff} is symmetric positive definite.
\end{prop}
\end{revblock}

\subsection{Convergence}
Let $\Pi_h u \in \mathbb V_{\kappa,s}(\mathcal{T})$ denote the Scott-Zhang interpolant \cite{ScottZhang1990}. Then, if $u\in H^r(\mathcal G)$ with
$r\in[0,p+1]$, for every $t\in[0,r]$, there exist constants $C, C'$ independent of $h$ and $u$ such that
\begin{align}
\label{eq:scott_zhang}
    \inf_{u_h\in \mathbb V_{\kappa,s}(\mathcal{T})}\|u-u_h\|_{H^t(\mathcal G)} \leq C\|u-\Pi_h u\|_{H^t(\mathcal G)} \le C'\,h^{\,r-t}\,|u|_{H^r(\mathcal G)} .
\end{align}
By \sref{prop:form_compare_compact}, for every $u\in \mathbb V_{\kappa,s}$, we have the sample-wise deterministic approximation bounds $
 \|u\|_{\mathbb V_{\kappa,s}} \leq C\big( |u|_{H^{\overline s}(\mathcal G)}^2 + \|u\|_{L_2(\mathcal G)}^2\big)$. So setting $t=0$ and $t=\overline s$ in \eqref{eq:scott_zhang}, we obtain that for $r>\overline s$
\begin{align*}
\big(\mathbb E\|u-u_h\|^2_{\mathbb V_{\kappa,s}}
\big)^{1/2}\le C\, h^{\min(p+1,r)-\overline s}\,\big(\mathbb E\|u\|^2_{H^r(\mathcal G)}\big)^{1/2}.
\end{align*} 
However, since $u \in H^r(\mathcal G)$ a.s. for \replaced[id=rev]{$r<r_\star\le2\underline s - d/2$}{$r<2\underline s - d/2$}, the error rate $\mathbb E\|u-u_h\|^2_{\mathbb V_{\kappa,s}} = \mathcal O(h^{\min(p+1,r)-\overline s})$ fails to hold in dimension $d=2$ since no such $r$ satisfies $\overline s < r< 2\underline{s}-d/2$. In dimension $d=1$, the estimate is valid only under special choices of $s(x)$. On the other hand, we can still consider the $L_2$ strong convergence error, which we give in the following result. \added[id=rev]{A natural companion question is the $L_2$ regularity of the deterministic solution operator $A^{-1}$; for $g\in L_2(\mathcal G)$, the solution $w_g\in\mathbb V_{\kappa,s}$ of
\begin{equation}\label{eq:dual_problem}
\mathcal A^{s(\cdot)}(v,w_g)=\langle g,v\rangle_{L_2(\mathcal G)}\qquad\text{for all } v\in\mathbb V_{\kappa,s},
\end{equation}
which exists and is unique by the Lax--Milgram theorem and equals $A^{-1}g$, should gain Sobolev regularity over the energy space. For constant order $s$, such estimates follow from the $\mu$-transmission calculus of \cite{Grubb2015} where the solution admits the decomposition $u=\operatorname{dist}(x,\partial\mathcal D)^{s}v_0+w$ with $v_0\in H^{s-\epsilon}$ and $w\in H^{2s-\epsilon}$, so that $\|u\|_{H^{\min(2s,\,s+1/2)-\epsilon}}\le C_\epsilon\|f\|_{L_2}$; the ceiling $s+1/2$ is caused by the boundary singularity $\operatorname{dist}(x,\partial\mathcal D)^{s}\notin H^{s+1/2}$. The following proposition extends this regularity theory to the variable-order operator $A$ by a localization-and-freezing argument; its proof is given in \autoref{app:proof_reg_variable} in the appendix.}

\begin{revblock}
\begin{prop}[Elliptic regularity for the variable-order operator]\label{prop:reg_variable}
Let $s:\overline{\mathcal G}\to[\underline s,\overline s]\subset(0,1)$ and $\mathcal G$ satisfy the standing assumptions (in particular the log-H\"older condition \eqref{eq:log_holder}). For every $\varepsilon>0$, the following hold for $u=A^{-1}f$ with $f\in L_2(\mathcal G)$.
\begin{enumerate}[(i)]
\item (Interior regularity.) Assume $\underline s>d/4$. For every $\mathcal D'\Subset\mathcal D$, set $\delta:=\operatorname{dist}(\mathcal D',\partial\mathcal D)>0$ and $\underline{s}':=\min_{\overline{\mathcal D'}}s$; then there exists $C_{\varepsilon,\delta}>0$ such that
\begin{equation}\label{eq:var_reg_interior}
\|u\|_{H^{2\underline{s}'-\varepsilon}(\mathcal D')} \le C_{\varepsilon,\delta}\,\|f\|_{L_2(\mathcal G)}.
\end{equation}
\item (Global regularity.) If, in addition, $\mathcal D$ and $\mathcal G$ are $C^\infty$ domains, then there exists $C_\varepsilon>0$ such that
\begin{equation}\label{eq:var_reg_global}
\|u\|_{H^{\min(2\underline s,\;\underline s+1/2)-\varepsilon}(\mathcal G)} \le C_\varepsilon\,\|f\|_{L_2(\mathcal G)}.
\end{equation}
\end{enumerate}
\end{prop}

\begin{remark}\label{rem:boundary_reg}
The $C^\infty$ assumption on $\partial\mathcal D$ is used only in the boundary charts when proving the global regularity, via \cite[Theorem~7.1]{Grubb2015}, whose $\mu$-transmission pseudodifferential calculus requires smooth coordinate diffeomorphisms; the interior estimate (i) requires no regularity of $\partial\mathcal D$. The threshold $\underline s+\tfrac12$ is intrinsic where it reflects $\operatorname{dist}(x,\partial\mathcal D)^{s_0}\in H^\sigma(0,1)$ iff $\sigma<s_0+\tfrac12$, a property that holds whenever $\partial\mathcal D\in C^{1,1}$. A $C^{1,1}$ H\"older-space analogue exists \cite[Theorem~1.2]{RosOtonSerra2014}, but it requires $f\in L_\infty$ (rather than $f\in L_2$) and yields only $u/\operatorname{dist}(\cdot,\partial\mathcal D)^s\in C^\alpha$ for small $\alpha<\min(s,1-s)$, which is too weak for the Sobolev estimate \eqref{eq:var_reg_global}. A Sobolev-space boundary regularity result for $C^{1,1}$ or $C^{k,\alpha}$ domains matching Grubb's $H^{s_0+1/2-\varepsilon}$ bound would allow the $C^\infty$ hypothesis to be relaxed.
\end{remark}
\end{revblock}

\begin{prop}[Strong $L_2$ Error - FE convergence rates]
\label{prop:FEM_MS_rate}
Assume the standing hypotheses as before. Let $\mathbb V_{\kappa,s}(\mathcal{T})$ be a Lagrange finite element space of degree $p>0$ on a family of shape-regular meshes with maximum mesh-size $h\le1$. \replaced[id=rev]{For any fixed $r$ with $0< r < r_\star$, where $r_\star$ is as in \sref{thm:spectral_underline_s}(ii) (in particular, $r_\star=2\underline s-\tfrac d2$ unless $d=1$ and $\underline s>\tfrac12$, in which case $r_\star$ is given by \eqref{eq:r_star} and $C^\infty$ domains are assumed),}{For any fixed $r$ with $0< r < 2\underline s - \tfrac d2$,} there exists $C>0$ (independent of $h$ and $u$) such that
\begin{equation}\label{eq:FEM_L2_error}
\added[id=rev]{\Big(\mathbb E\inf_{v_h\in \mathbb V_{\kappa,s}(\mathcal T)}\|u-v_h\|_{L_2(\mathcal G)}^2\Big)^{1/2}
\le C\, h^{\min(p+1,r)} = C h^r.
}
\end{equation}
\end{prop}
In practice, the observed convergence rate is typically greater than or equal to
$2\underline s - \tfrac d2$, indicating that the choice $r = 2\underline s - \tfrac d2$
is not sharp. Indeed, this
threshold arises exclusively from the regularity result stated in
\sref{thm:spectral_underline_s}, which ensures that
$u \in H^r(\mathcal G)$ for all $0 \le r < \added[id=rev]{r_\star}$.
However, this lower bound does not exclude the possibility that
$u$ possesses additional Sobolev regularity, i.e.,
$u \in H^{r^*}(\mathcal G)$ for some $r^* \ge 2\underline s - \tfrac d2$.
Consequently, the theoretical rate derived from this minimal regularity
may underestimate the actual convergence behavior. \added[id=rev]{For $d=1$ and $\underline s>\tfrac12$, the reduction from $2\underline s-\tfrac12$ to $r_\star$ stems from the eigenfunction estimate available in the variable-order setting, where the lift $\phi_j=\lambda_jA^{-1}\phi_j$ combined with \sref{prop:reg_variable}(ii) bounds $\|\phi_j\|_{H^\sigma(\mathcal G)}$ by a full power of $\lambda_j$, rather than the calibrated growth $\lambda_j^{\sigma/(2\underline s)}$ that a characterization of $H^\sigma(\mathcal G)$ as the domain of a fractional power of $A$ would provide, and interpolation then costs the factor $1-\tfrac{1}{2\underline s}$ in \eqref{eq:r_star}. We therefore expect the reduction to be an artifact of the proof rather than a property of the solution.} Furthermore, since $2\underline s - d/2 < 2$,
increasing the polynomial degree $p$ beyond one does not yield any
additional improvement in the observed rate.

\subsection{Computation of the singular near-field integrals}
\label{subsec:singular_integrals}

In this section, we describe the numerical strategy for computing the near-field integrals that exhibit singular behavior, specifically those arising in the term
\begin{align}
\label{eq:A2}
\mathcal B(U,V) := \sum_{\substack{\tau \in \mathcal{T}^\text{int}, \tau'\in \mathcal{T} \\ \tau \cap \tau' \neq \emptyset}} \! \!
    \mathcal{A}_{\tau,\tau'}^2(U,V)
\end{align}
appearing in \eqref{eq:loc-stiff}. We focus on the one-dimensional case ($d = 1$). Let $\hat{\tau}$ be the reference interval $[0,1]$. We denote by $T_i$ and $T_j$ the affine maps that transform $\hat{\tau}$ to mesh elements $\tau := [x_i, x_{i+1}]$ and $\tau' := [x_j, x_{j+1}]$, respectively, with associated Jacobians $J_i$ and $J_j$. Substituting the expansion of $U$ into the bilinear form and testing with $V = \psi_j$, we obtain integrals of the form
\begin{align}
\label{eq:sing_int}
\mathcal{A}^2_{\tau,\tau'}(\psi_i,\psi_j) &=  |J_i||J_j|\int_{\hat{\tau}}\int_{\hat{\tau}} \left(\psi_i\left(T_i(\hat{x})\right) - \psi_i\left(T_j(\hat{y})\right)\right)\left(\psi_j\left(T_i(\hat{x})\right) - \psi_j\left(T_j(\hat{y})\right)\right) \\
&\qquad \qquad \qquad \qquad \times \gamma\left(T_i(\hat{x}),T_j(\hat{y})\right) \, d\hat{y} \, d\hat{x}, \nonumber
\end{align}
where the kernel $\gamma(x,y)$ has the form
\begin{align*}
\gamma(x,y) = \frac{K_{\nu(x,y)}(\kappa |x-y|)}{|x-y|^{\nu(x,y)}} \,C(x,y).
\end{align*}
with the smooth prefactor 
\[
C(x,y)=\frac{1}{2\sqrt{\pi}}\frac{2^{\nu(x,y)}}{|\Gamma(-\beta(x,y))|}\,\kappa^{\nu(x,y)},
\]
and $\nu(x,y)=\tfrac12+\beta(x,y)$.
Define the auxiliary (regularized) function
\[
\Phi(x,y) \;:=\; C(x,y)\;K_{\nu(x,y)}(\kappa r)\;r^{\nu(x,y)},\qquad r:=|x-y| .
\]
We split the domain $[0,1]^2$ into two triangles and rewrite the integral above as
\begin{align*}
\mathcal{A}^2_{\tau,\tau'}(\psi_i,\psi_j)
&=|J_i||J_j|\int_0^1\int_0^{\hat x}
\Delta\psi_i(\hat x,\hat y)\Delta\psi_j(\hat x,\hat y)\;\frac{\Phi(T_i(\hat x),T_j(\hat y))}{r(T_i(\hat{x}),T_j(\hat{y}))^{2\nu}}\,d\hat y\,d\hat x\\
&\quad + |J_i||J_j|\int_0^1\int_0^{\hat y}
\Delta\psi_i(\hat x,\hat y)\Delta\psi_j(\hat x,\hat y)\;\frac{\Phi(T_i(\hat x),T_j(\hat y))}{r(T_i(\hat{x}),T_j(\hat{y}))^{2\nu}}d\hat x\,d\hat y \\
&=: \mathcal{A}^{2,T_1}_{\tau,\tau'}(\psi_i,\psi_j) + \mathcal{A}^{2,T_2}_{\tau,\tau'}(\psi_i,\psi_j),
\end{align*}
with $\Delta\psi_i(\hat x,\hat y):=(\psi_i(T_i(\hat x))-\psi_i(T_j(\hat y)))$
and $
\Delta\psi_j(\hat x,\hat y):=(\psi_j(T_i(\hat x))-\psi_j(T_j(\hat y)))$.
The singularity arises when $|T_i(\hat{x}) - T_j(\hat{y})| \to 0$, which occurs in two main scenarios.
\begin{itemize}
  \item[(i)] $\tau$ and $\tau'$ share a vertex, e.g., $\tau = [x_i, x_{i+1}]$ and $\tau' = [x_{i+1}, x_{i+2}]$.
  \item[(ii)] $\tau = \tau'$, i.e., the integrals are computed over the same element.
\end{itemize}
For case (i), we define affine maps such that $T_i(0) = T_{i+1}(0)$
\begin{align*}
T_i(\hat{x}) = x_{i+1} + (x_i - x_{i+1})\hat{x}, \quad T_{i+1}(\hat{y}) = x_{i+1} + (x_{i+2} - x_{i+1})\hat{y}.
\end{align*}
To resolve the singularity, we apply a Duffy-type transformation, a classical tool in the boundary element literature \cite{Sauter2011,Delia2021}; see also \cite{Acosta2017,Ainsworth2017,Ainsworth2018} for similar treatments in the constant-fractional setting. After transformation, the integral is approximated using tensor-product Gaussian quadrature.
Assuming for simplicity a uniform mesh with size $h$, then $|T_{i+1}(\hat{y}) - T_i(\hat{x})| =  h |\hat{y}+\hat{x}|$. Using a Duffy-type transformation $\hat{x} = \xi, \quad \hat{y} = \xi \eta$, with $d\hat{x} \, d\hat{y} = \xi \, d\eta \, d\xi$, which maps the first triangular domain $T_1=\{(\hat{x}, \hat{y}) \in [0,1]^2 : \hat{y} \leq \hat{x}\}$ to the unit square $[0,1]^2$,
the $\mathcal{A}^{2,T_1}_{\tau,\tau'}(\psi_i,\psi_j)$ part becomes
\begin{align*}
\mathcal{A}^{2,T_1}_{\tau,\tau'}(\psi_i,\psi_j)
=h^2
\int_0^1\int_0^1
\Delta\psi_i\big(T_i(\xi),T_j(\xi\eta)\big)\,
\Delta\psi_j\big(T_i(\xi),T_j(\xi\eta)\big)\;
\frac{\Phi\big(T_i(\xi),T_j(\xi\eta)\big)}{r\big(T_i(\xi),T_j(\xi\eta)\big)^{2\nu}}\;
\xi\,d\eta\,d\xi.
\end{align*}
The second part $\mathcal{A}^{2,T_2}_{\tau,\tau'}(\psi_i,\psi_j)$ is treated similarly using $\hat{x} = \xi \eta, \quad \hat{y} = \xi $, with $d\hat{x} \, d\hat{y} = \xi \, d\eta \, d\xi$.

For linear reference shape functions $\hat{\psi}^1(z) =1 - z$ and $\hat{\psi}^{2}(z)=z$, we have $\Delta\hat{\psi}_i^l(\xi,\xi\eta)\Delta\hat{\psi}_j^m(\xi,\xi\eta) = \pm \xi^2(1-\eta)^2$, $l,m\in\{1,2\}$ and $r\big(T_i(\xi), T_j(\xi\eta)\big)=h\xi(1+\eta)$, therefore, the integral becomes
\begin{align*}
    \mathcal{A}^{2,T_1}_{\tau,\tau'}(\psi_i,\psi_j)
=\pm
\int_0^1\int_0^1 h^{2-2\nu}  \xi^{3-2\nu}(1-\eta)^{2} (1+\eta)^{-2\nu}
\Phi\big(T_i(\xi),T_{i+1}(\xi\eta)\big)\,d\eta\,d\xi.
\end{align*}
However, to improve the quadrature approximation rates, we apply the additional changes of variables $
 \xi = \zeta^{1/(3 - 2\overline{s})}$. 
This transformation smooths the integrand further, leading to
\begin{align}
\label{eq:transformed-integrand_adj}
\mathcal{A}^{2,T_1}_{\tau,\tau'}(\psi_i,\psi_j)
&=
\frac{\pm 1}{3-2\overline s}
\int_0^1\int_0^1
h^{1-2\beta}\zeta^{\frac{2(\overline s-\beta)}{3-2\overline s}}
(1-\eta)^2(1+\eta)^{-(1+2\beta)}
\\
&\qquad \qquad\qquad\times
\Phi\!\Big(
T_i\!\big(\zeta^{\frac{1}{3-2\overline s}}\big),
T_{i+1}\!\big(\zeta^{\frac{1}{3-2\overline s}}\eta\big)
\Big)
\,d\eta\,d\zeta . \nonumber
\end{align}
with nonnegative powers in $\zeta$ and $t$. The resulting integral is then approximated using a tensor-product Gaussian quadrature rule, which is highly efficient for smooth integrands on rectangular domains (see, e.g., \cite{Quarteroni2007}). Case (ii), when $\tau = \tau'$, is handled similarly. In this case, $|T_{i}(\hat{y}) - T_i(\hat{x})| =  h |\hat{y}-\hat{x}|$. An extra change of variables $\eta = 1 - t^{1/(2 - 2\overline{s})}$ is used to get
\begin{align} \label{eq:transformed-integrand_ident} \mathcal{A}^{2,T_1}_{\tau,\tau'}(\psi_i,\psi_j) &= \frac{\pm 1}{(2-2\overline{s})(3 - 2\overline{s})} \int_0^1\int_0^1 h^{1 - 2\beta} \zeta^{\frac{2\overline{s} - 2\beta}{3 - 2\overline{s}}} t^{\frac{2\overline{s} - 2\beta}{2 - 2\overline{s}}} \\
&\qquad \qquad \qquad \qquad \qquad \times \Phi\Big(T_i(\zeta^{\frac{1}{3-2\overline s}}),T_{i}(\zeta^{\frac{1}{3-2\overline s}}(1-t^{\frac{1}{2-2\overline s}})))\Big) \, d\zeta \, dt. \nonumber
\end{align}

Using the argument from \cite[Section~4]{Lei2023}, we can derive the following local quadrature error and a detailed proof is provided in Appendix.

\begin{proposition}\label{prop:local-quad-error}
  \added[id=rev]{Assume that $s$ is real-analytic on $\overline{\mathcal G}$, admitting a bounded analytic extension to a complex neighborhood of $\overline{\mathcal G}$.} Denote $Q_{\tau,\tau'}^n$ to be the $n$th-order tensor-product Gaussian quadrature approximation of $\mathcal A^2_{\tau,\tau'}(\psi_i,\psi_j)$ based on the form \eqref{eq:transformed-integrand_adj} or \eqref{eq:transformed-integrand_ident}. Then there exist $\rho\in (\tfrac12,1)$ and a positive constant $C$ not depending on $h$ satisfying
$|\mathcal A^2_{\tau,\tau'}(\psi_i,\psi_j) - Q_{\tau,\tau'}^n| \le C h^{1-2\overline s}(2\rho)^{-2n} .$
\end{proposition}

\begin{revblock}
\begin{remark}\label{rem:quad_analytic}
The analyticity of $s$ is needed only for the exponential quadrature rate; it plays no role in the variational theory of \autoref{sec:weak_form}, which requires log-H\"older continuity. Note also that the exponents of $\zeta$ and $t$ in \eqref{eq:transformed-integrand_adj} and \eqref{eq:transformed-integrand_ident} are nonnegative but in general noninteger, so the integrand has a branch point at the origin and does not extend analytically to any Bernstein ellipse containing $0$. The rate $(2\rho)^{-2n}$ therefore holds for $\rho$ restricted by the branch point, which limits the achievable $\rho$ but not the geometric character of the convergence; see the proof in the appendix.
\end{remark}
\end{revblock}

\begin{corollary}
  Summing up the above local errors for all $\tau$ and $\tau'$, \replaced[id=rev]{the quadrature error for}{we can show that the quadrature error for} \eqref{eq:A2} \replaced[id=rev]{is}{can be} bounded by $Ch^{-2\overline s}(2\rho)^{-2n}$ (\replaced[id=rev]{see}{e.g.} \cite[Proposition~1]{Lei2023}). Moreover, given the mesh size $h< 1$, let the quadrature order $n \ge c \log(1/h) (2\overline s + 2\underline s-d/2)$ for some fixed constant $c$. The quadrature error for \eqref{eq:A2} can be further bounded by $Ch^{2\underline s-d/2}$. This implies that the error from the finite element approximation is dominant. 
\end{corollary}

If the finite element approximation error has the rate $O(h^{r})$, we can choose $n \ge C\log(1/h)( r +  2\overline s)$ so that the finite element error is always dominant.

\subsection{Sampling procedure}
\label{subsec:sampling}
By the definition of white noise in \autoref{sec:const_smooth}, the random vector $b$, whose components are given by $b_i = \left\langle \mathcal{W}, \psi_i \right\rangle_{L_2(\mathcal{D})}$ with $\psi_i$ denoting a finite element basis function, satisfies $\mathbb{E}[b_i] = 0$ and $\mathbb{E}[b_ib_j]=\left\langle\psi_i,\psi_j\right\rangle_{L_2(\mathcal{D})}$.
In other words, $b \sim \mathcal{N}(0, M)$, where $M$ is the finite element mass matrix with entries $M_{ij} = \left\langle\psi_i,\psi_j\right\rangle_{L_2(\mathcal{D})}$. This matrix needs to be assembled only once. \added[id=rev]{Here the $L_2(\mathcal D)$ inner products coincide with the $L_2(\mathcal G)$ ones. The basis functions satisfy $\psi_i=0$ on $\mathcal D_t^c$ since $\mathbb V_{\kappa,s}(\mathcal T)\subset\mathbb V_{\kappa,s}$, so $\langle\psi_i,\psi_j\rangle_{L_2(\mathcal D)}=\langle\psi_i,\psi_j\rangle_{L_2(\mathcal G)}$, and the action of $\mathcal W$ on $\psi_i$ is the white noise on $\mathcal G$ tested against $\psi_i$. The sampling problem written with $L_2(\mathcal D)$ inner products is therefore identical to the restriction of the weak form on $\mathcal G$ to $\mathbb V_{\kappa,s}(\mathcal T)$, and no approximation is involved in using $\mathcal D$ instead of $\mathcal G$ here.} A Cholesky factorization of $M$ yields $M = LL^T$, where $L \in \mathbb{R}^{N \times N}$ is lower triangular.

Once the global stiffness matrix $A=A^1+A^2$ (where $A^1$ and $A^2$ are assembled from the local contributions in \eqref{eq:loc-mass} and \eqref{eq:loc-stiff}, respectively), are available, one-dimensional Gaussian random fields can be generated efficiently. Let $m \in \mathbb{N}$ be the desired number of samples. For each sample, one draws a standard normal vector $z\sim \mathcal{N}(0, I)$, computes the white noise realization via $b = L z$, and solves the linear system $A u = b$ to obtain the sampled field $u$. The assembly of the mass and stiffness matrices $M$ and $A$ is performed in parallel over all elements $\tau$, enabling significantly faster computation. Furthermore, the sampling procedure can also be done in parallel.
%

%
%
%
%
%

\section{Numerical Illustrations}
\label{sec:numerics}

In this section, we present results on sampling Whittle–Matérn random fields with variable smoothness by solving the SPDE \eqref{eq:loc_frac_spde1} for various values of $\kappa$ and different choices of $s(x)$ in 1D. We also show some convergence results. We ran all computations in MATLAB on a Precision 7920 Tower Workstation equipped with 52 cores and 512 GB of memory.

\subsection{Sampling}
Let $\mathcal{D}=[-R_{int},R_{int}]$ denote the interior interval and $\mathcal{D}_t^c=[-R_{ext},R_{ext}]\setminus \mathcal{D}$ be the exterior interval. In this section, we take $R_{int}=3$ and $R_{ext}=4$. The parameters are fixed as $\mu=1$, number of samples $m=1000$, and discretization step-size $h=2^{-8}$ in the whole domain. We note that it is possible to employ a graded mesh in larger exterior domains $\mathcal D_t^c$. Across all cases below, assembling the global stiffness matrix required approximately 36 seconds (using all available cores), while generating $m = 1000$ samples took roughly 10 seconds; so in total the code took around 45 seconds to generate $1000$ samples. 

We first examine the constant-order case with $s = 0.5$. 
As shown in the top-left panel of \autoref{fig:cov_comp}, 
for sufficiently large values of $\kappa$ (in our experiments, $\kappa =2.5$), the covariance $C_s := A^{-1} M A^{-T}$ of the FEM solutions to the SPDE~\eqref{eq:loc_frac_spde} 
closely matches the Matérn covariance function. Next, we examine the variable case and consider three representative choices of $s(x)$; see \autoref{fig:1d_s_cases}.
{
\begin{align}
\qquad &\text{Case 1: Step function } s(x) =
\left\{\begin{array}{ll}
\underline{s}, & x \leq 0, \\
\overline{s}, & x > 0,
\end{array}\right.  \nonumber\\
& \text{Case 2: Gaussian bump } s(x) = \left\{\renewcommand\arraystretch{0.7}\begin{array}{ll}
\underline{s} 
+ (\overline{s}-\underline{s}) \frac{ \exp\{-(\frac{x}{\sigma})^2\} - \exp\left\{-(\frac{R_{int}}{\sigma})^2\right\} }{ 1 - \exp\left\{-(\frac{R_{int}}{\sigma})^2\right\} }, & |x| \le R_{int}, \\[1em]
\underline{s}, & |x| > R_{int}, 
\end{array}\right.
\label{eq:s_cases}\\
&\text{Case 3: Oscillatory Ramp } s(x) = \left\{\renewcommand\arraystretch{0.7}\begin{array}{ll}
 a, 
   & x \le -R_{int}, \\
 a
   + \frac{b-a}{2R_{int}}\,(x+R_{int})
   + \omega \sin\!\left(\frac{4 \pi\, (x+R_{int})}{R_{int}} \right), 
   & |x| < R_{int}, \\
 b, 
   & x \ge R_{int} \nonumber,
\end{array}\right.
\end{align}
}
\begin{figure}[!htbp] 
\begin{center} 
		\includegraphics[width=0.23\textwidth]{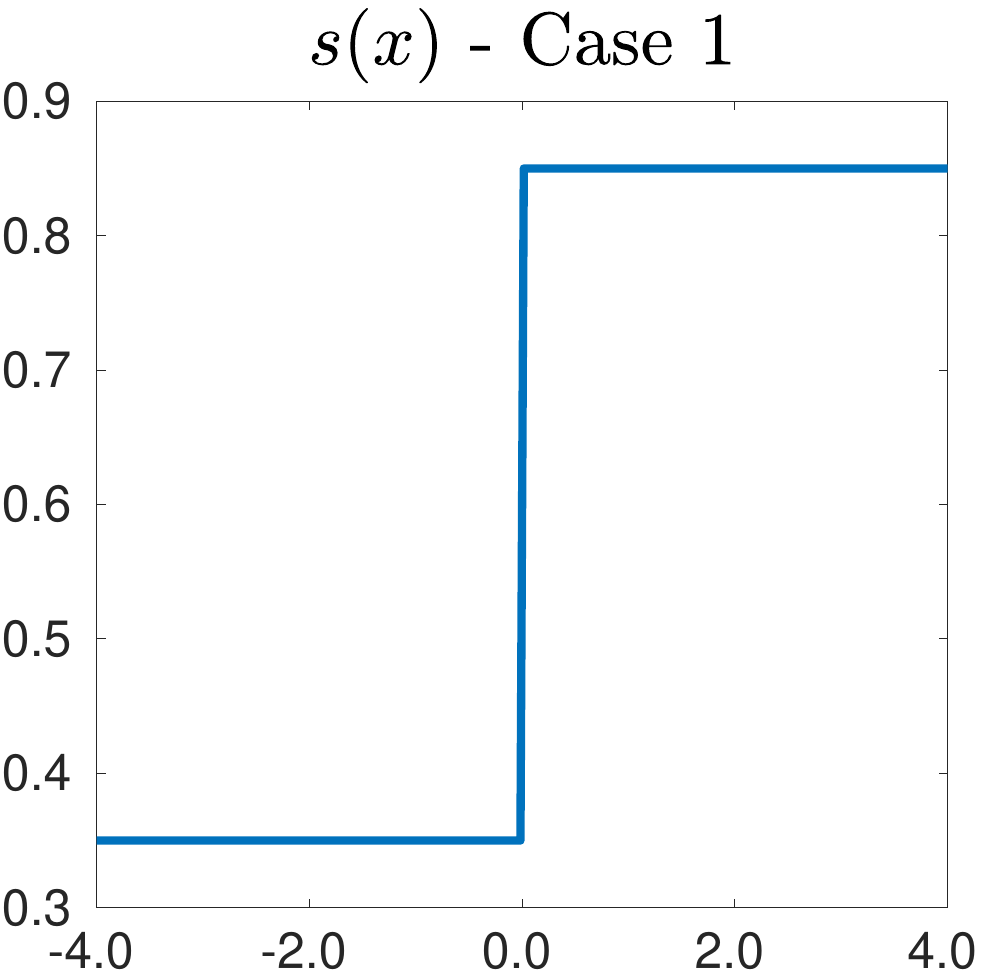} \quad \includegraphics[width=0.23\textwidth]{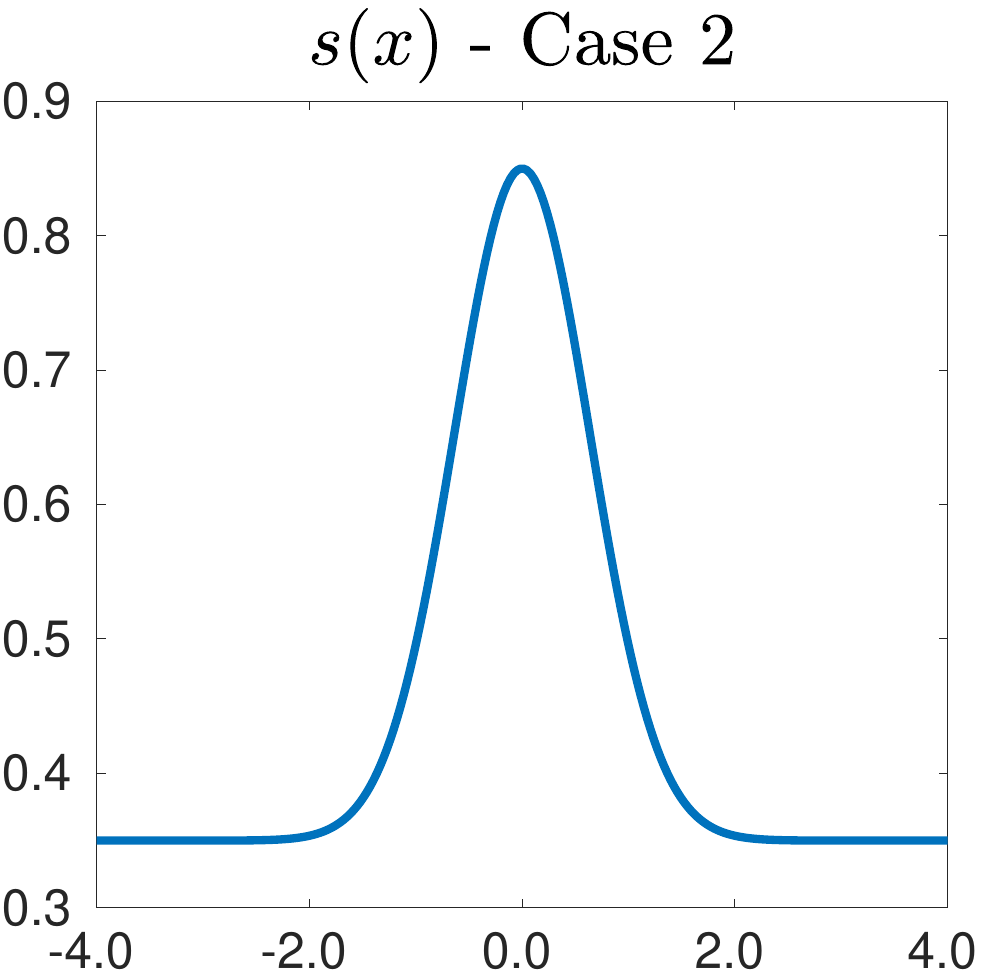} \quad \includegraphics[width=0.23\textwidth]{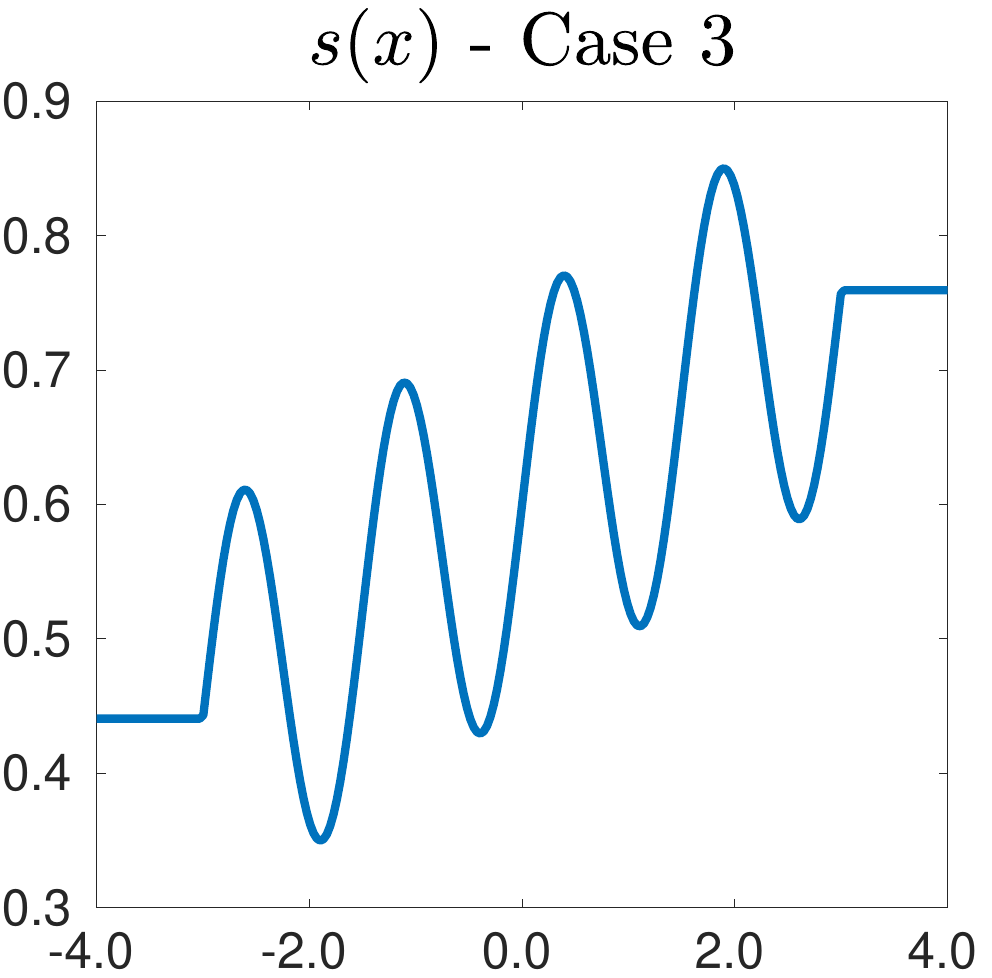}  
\end{center} 
	\caption{We consider three different functions of $s(x)$. The left, middle and right panels display a step function, a Gaussian bump and an oscillatory ramp, respectively.} \label{fig:1d_s_cases} 
\end{figure}
\begin{figure}[!htbp] 
\begin{center} 
	\includegraphics[width=0.48\textwidth]{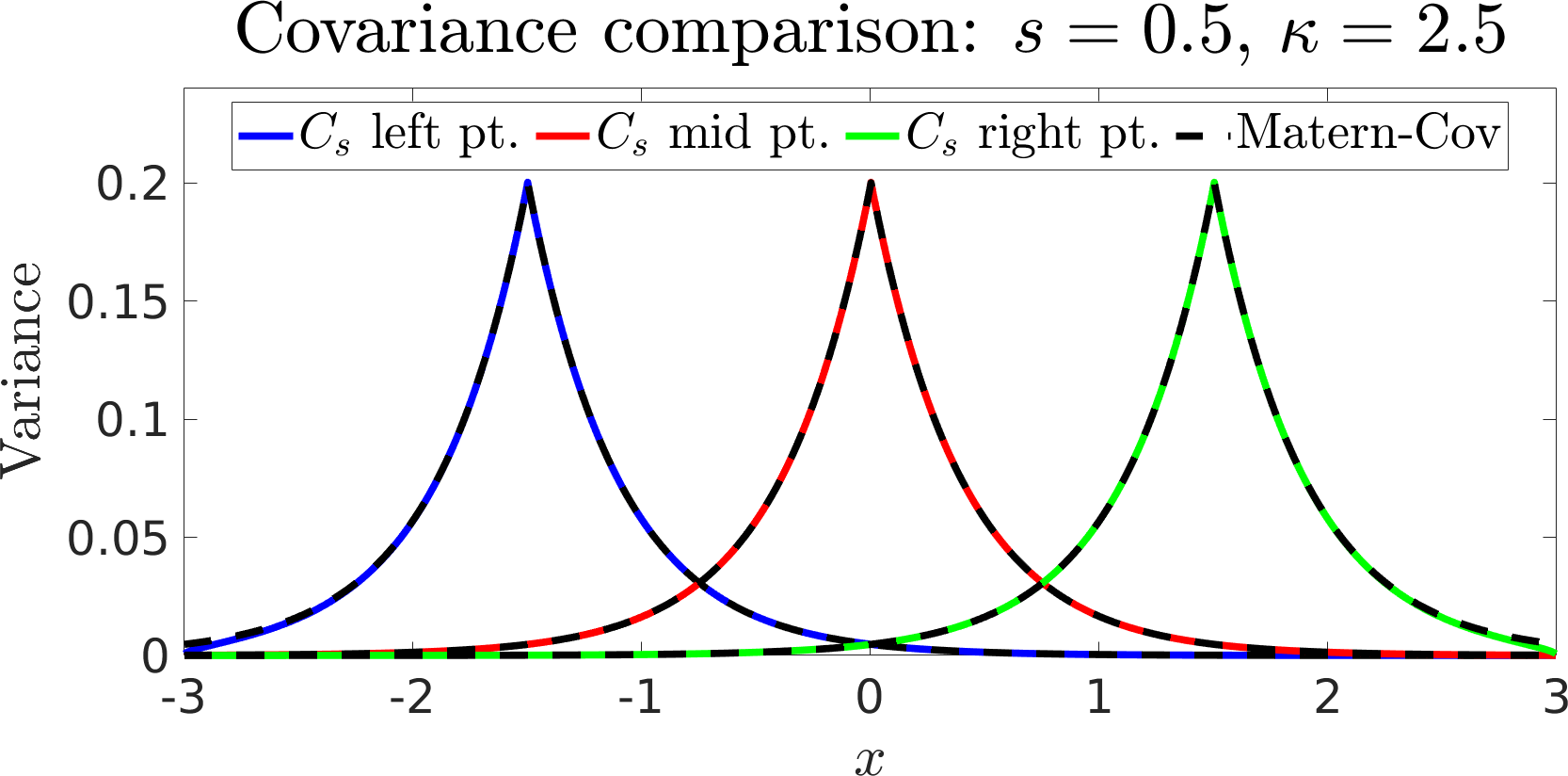} 
	\includegraphics[width=0.48\textwidth]{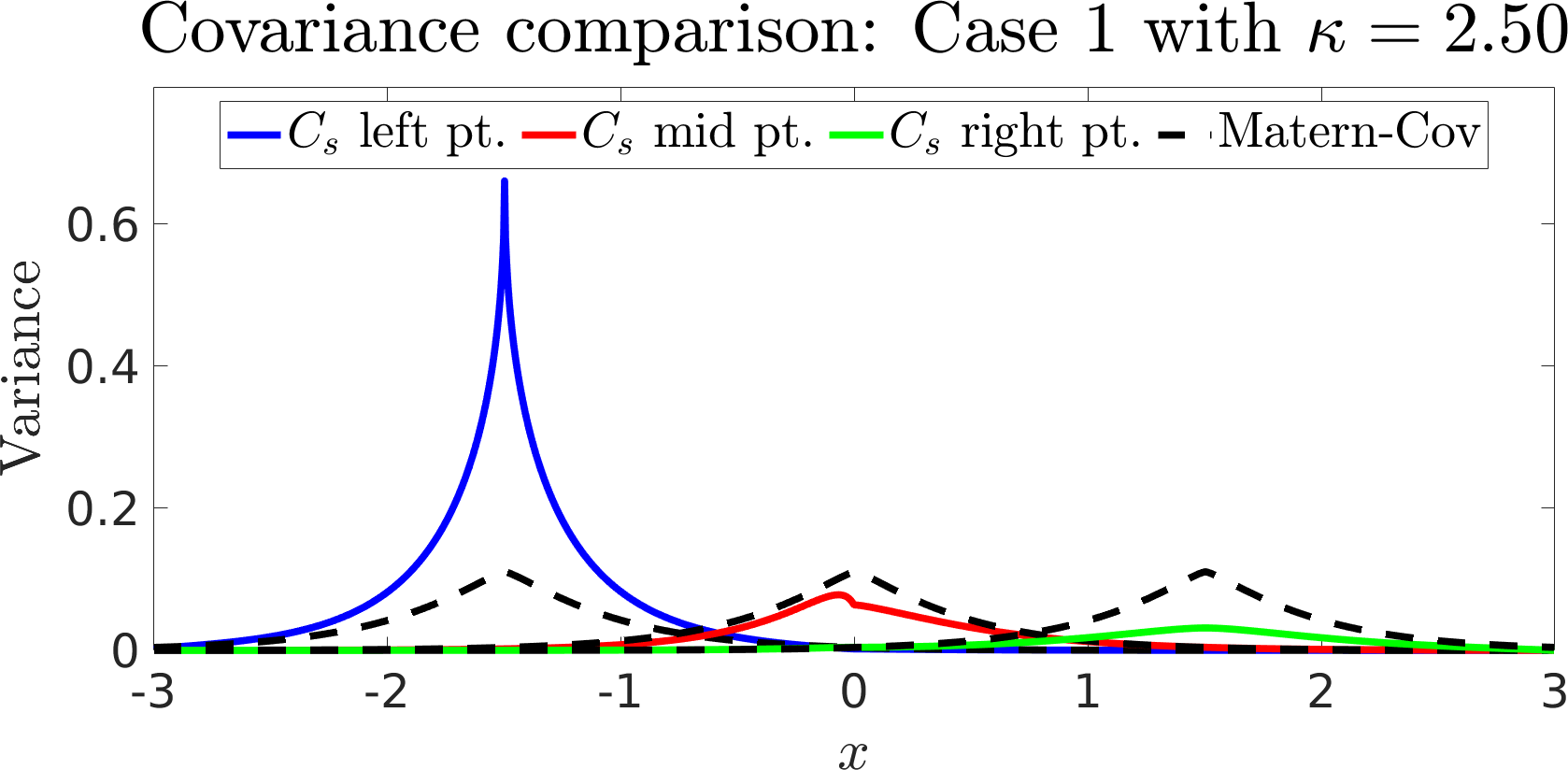}\\	\includegraphics[width=0.48\textwidth]{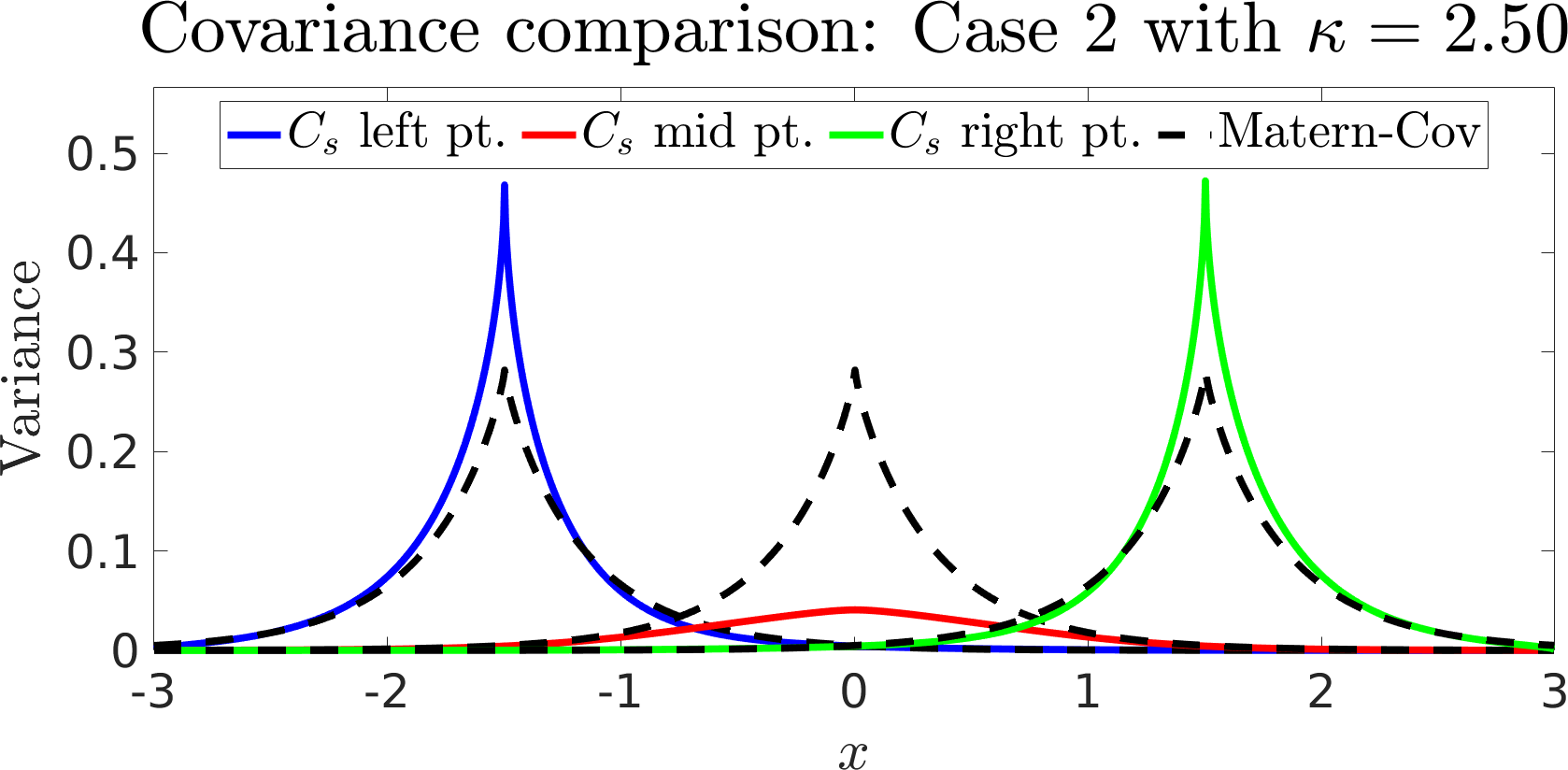}
	\includegraphics[width=0.48\textwidth]{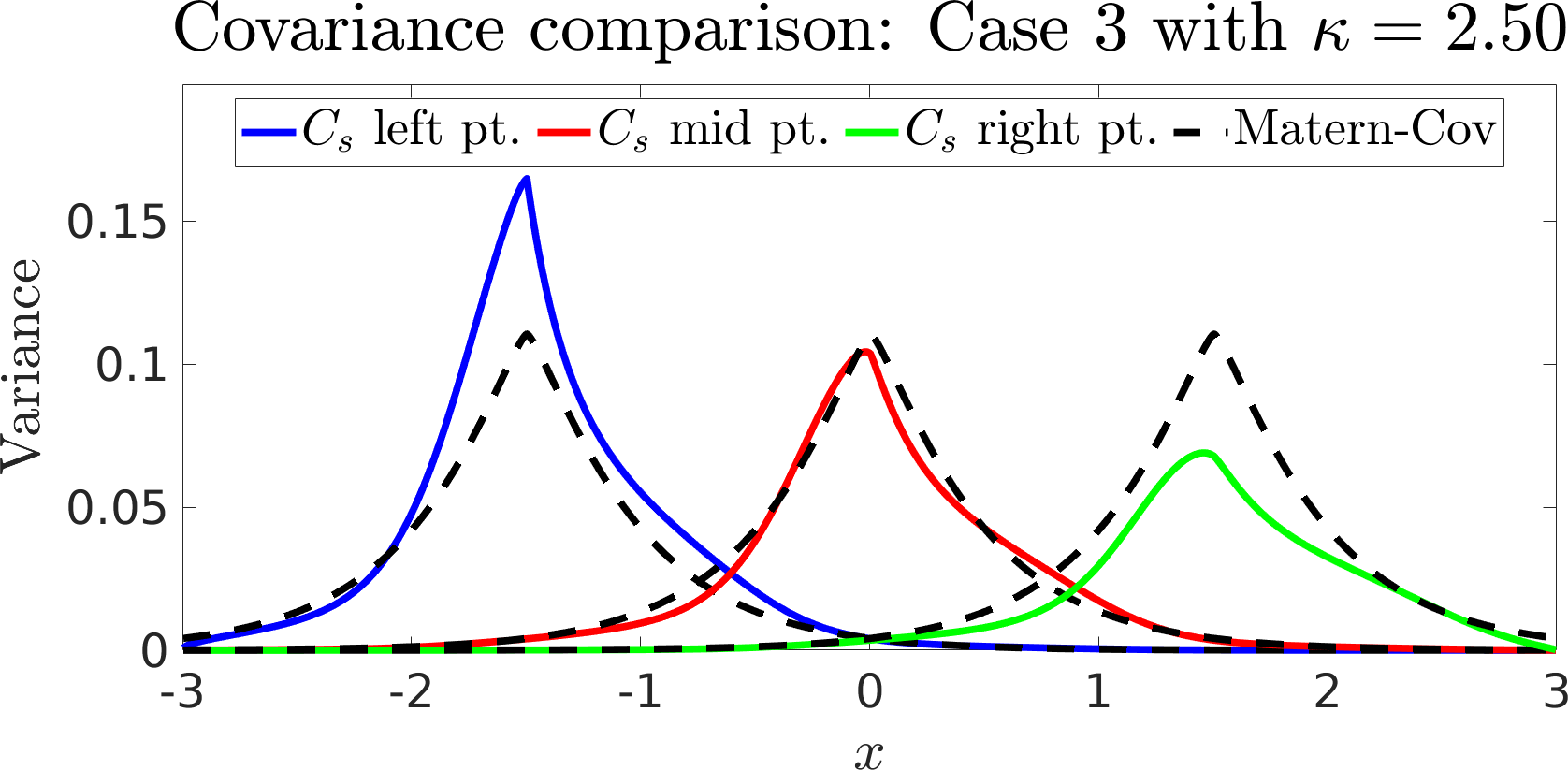}  
\end{center} 
	\caption{Comparison between the Mat\'{e}rn covariance and the FEM solution covariance $C_s = A^{-1} M A^{-T}$ with $\kappa = 2.5$. For the constant case (top left), the Mat\'{e}rn covariance is computed with $\nu=2s-1/2$ while for the variable cases Cases~1--3 it is computed with $\nu = 2\,\left\langle s(x)\right\rangle-d/2$.} \label{fig:cov_comp} 
\end{figure}
where we have $\underline s = 0.35$, $\overline{s}=0.85$, $\sigma = 0.3 R_{int}$, $a=0.44075$, $b=0.7594$ and $\omega=0.15$.

We plot the solution covariance $C_s$ at three different locations in the interior domain, namely, we plot $C_s(-1.5,\cdot)$, $C_s(0,\cdot)$ and $C_s(1.5,\cdot)$ and compare with the Mat\'ern covariance for $\kappa=2.5$. Since the smoothness here is varying in space, we compare with the Mat\'ern covariance computed with $\nu = 2\,\left\langle s(x)\right\rangle-d/2$, where $\left\langle s(x)\right\rangle = \frac{1}{|\mathcal G|} \int_{\mathcal{G}} s(x) dx
$. The comparison for Case 1 (the step function $s$) is shown in \autoref{fig:cov_comp} (top-right panel). We also plot 100 samples of the 1000 generated samples of the random Gaussian fields for different values of $\kappa$, namely 0.25 and 2.5; the results are shown in \autoref{fig:var_s_samples_k=2.5_and0.25} (top panel). \replaced[id=rev]{We note that $C_s = A^{-1} M A^{-T}$ is the exact covariance matrix of the discrete FEM solution rather than a Monte Carlo estimate, and it is exactly symmetric, since the stiffness matrix $A$ is symmetric positive definite and hence $A^{-T}=A^{-1}$; this mirrors the covariance symmetry of the continuous model established in \sref{rem:cov_symmetry}.}{We note that the samples covariance $C_s = A^{-1} M A^{-T}$ is symmetric, since the stiffness matrix $A$ is symmetric positive definite and hence $A^{-T}=A^{-1}$; this mirrors the exact covariance symmetry of the model, so the plotted covariances are symmetric up to Monte Carlo sampling error.}

\begin{figure}[!htbp] 
\begin{center} 
	\includegraphics[width=0.45\textwidth]{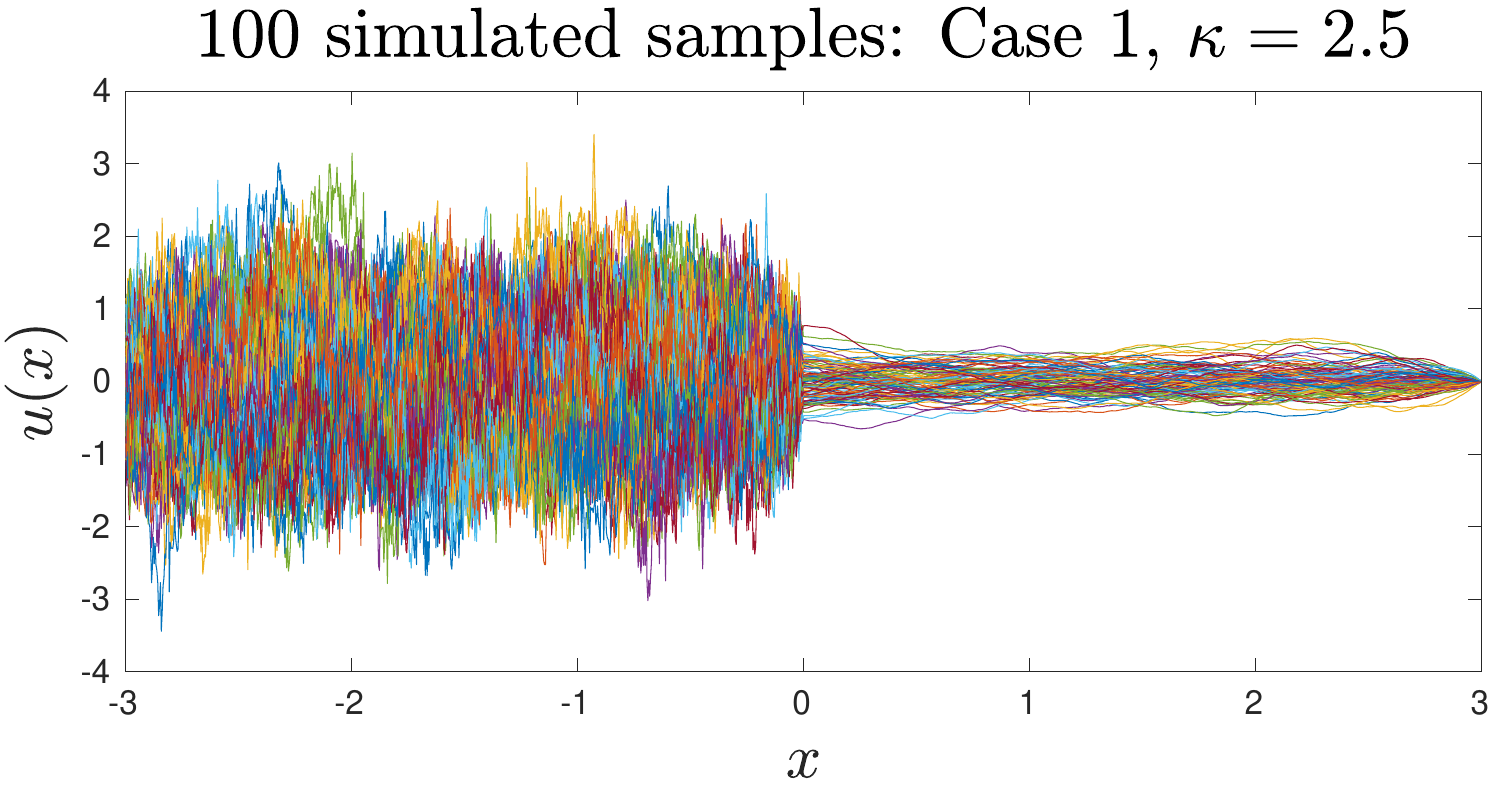} \includegraphics[width=0.45\textwidth]{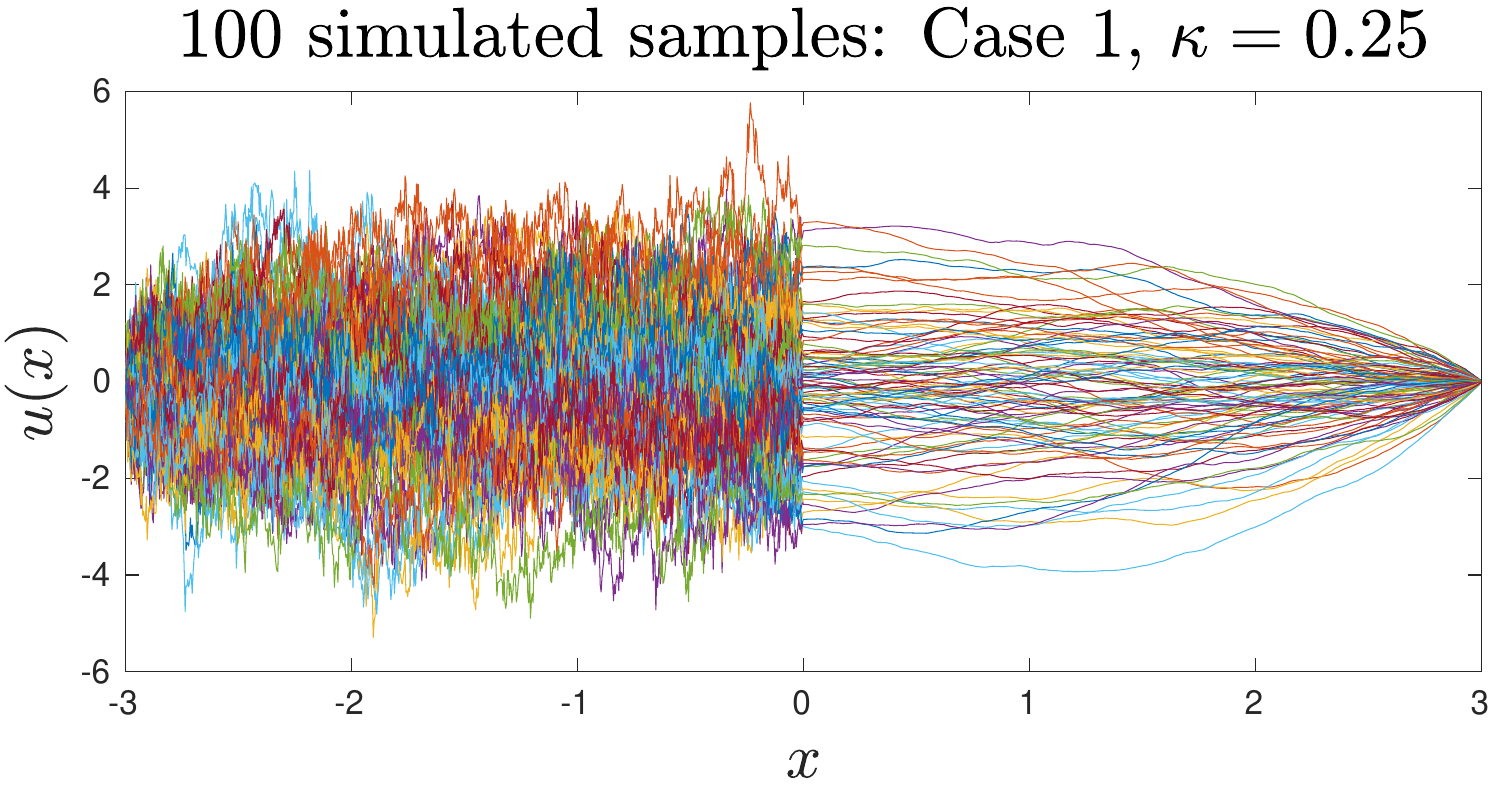} \\	
\includegraphics[width=0.45\textwidth]{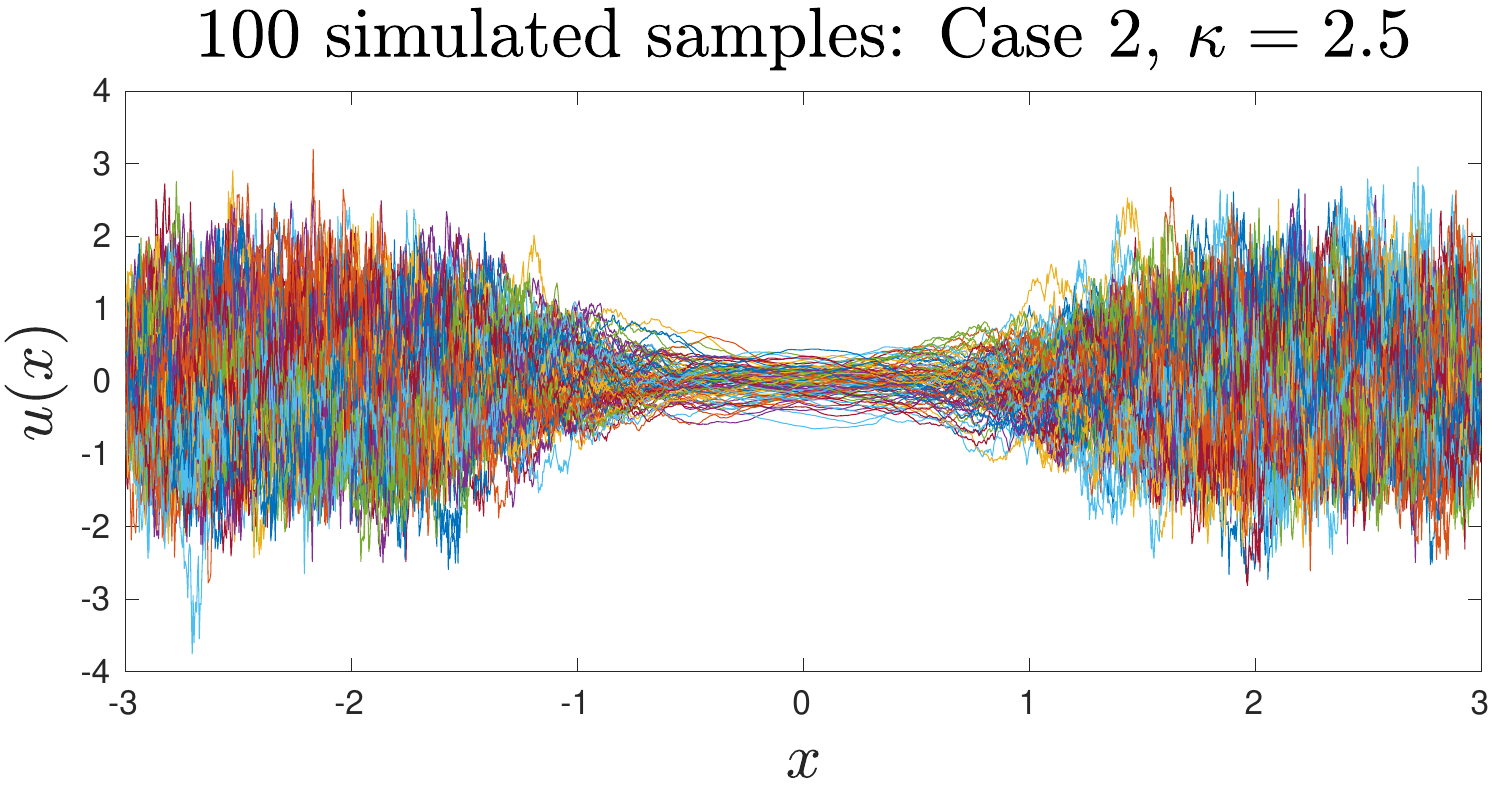}
	\includegraphics[width=0.45\textwidth]{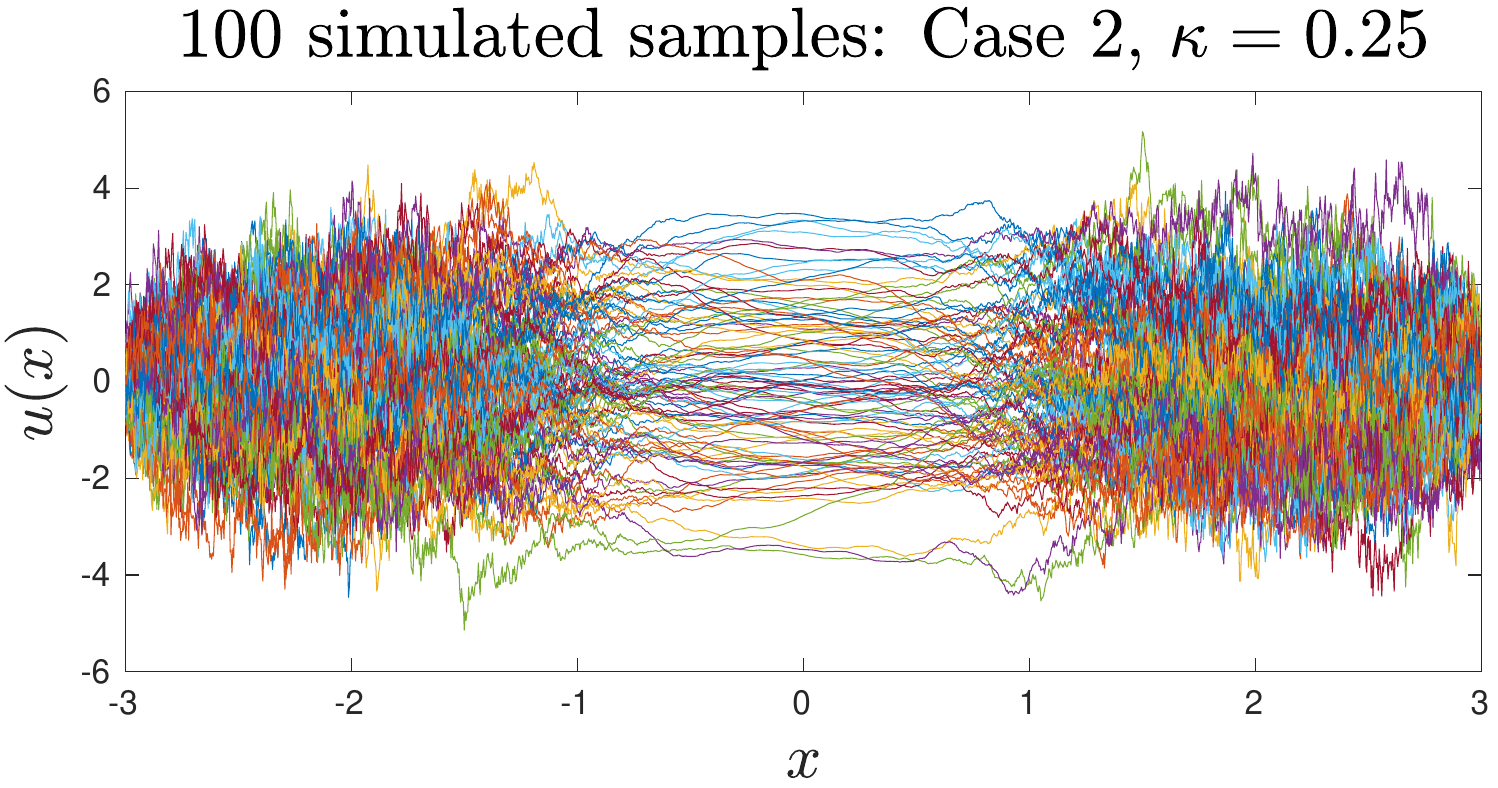} \\
	\includegraphics[width=0.45\textwidth]{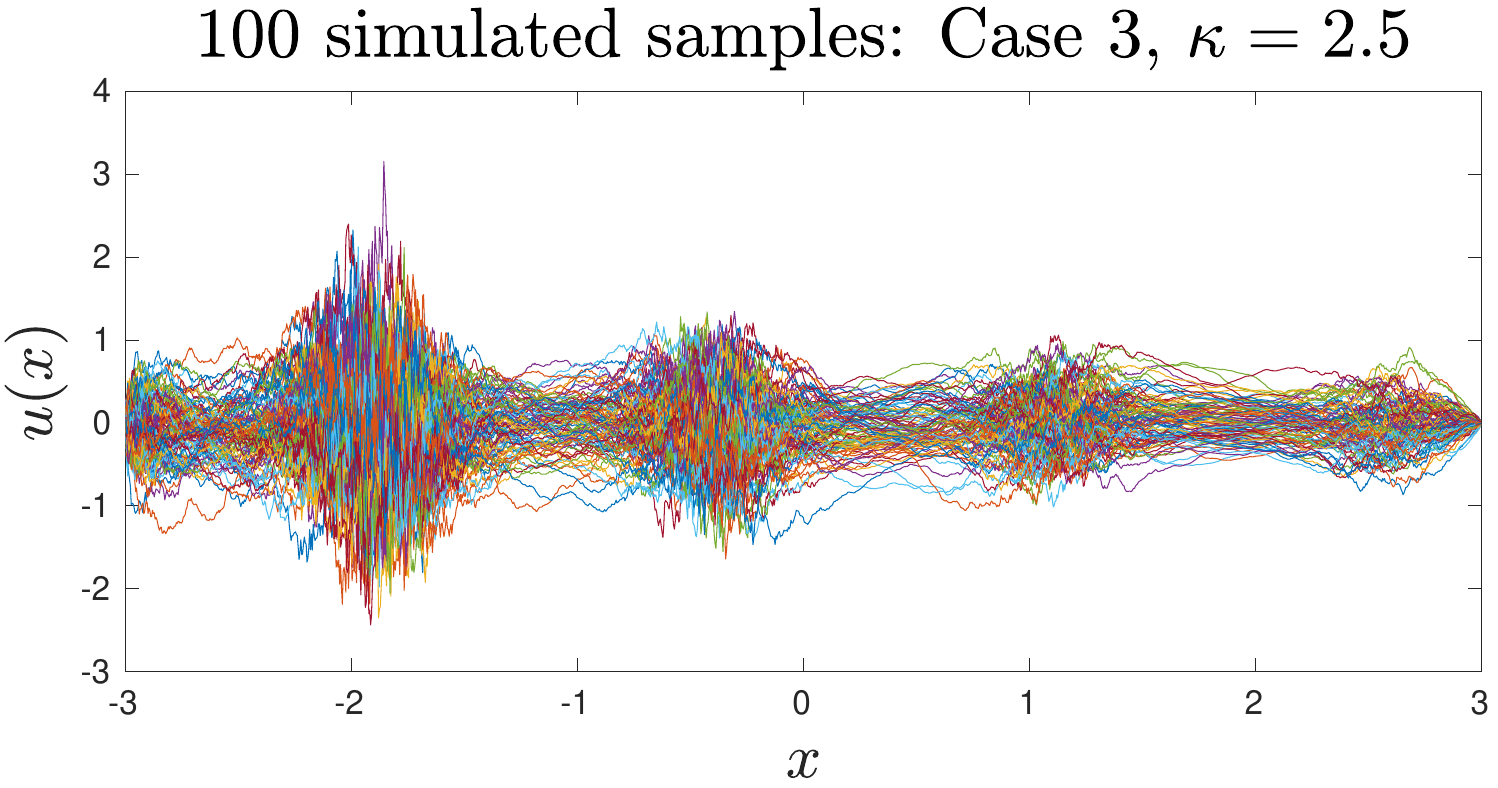} 	
	\includegraphics[width=0.45\textwidth]{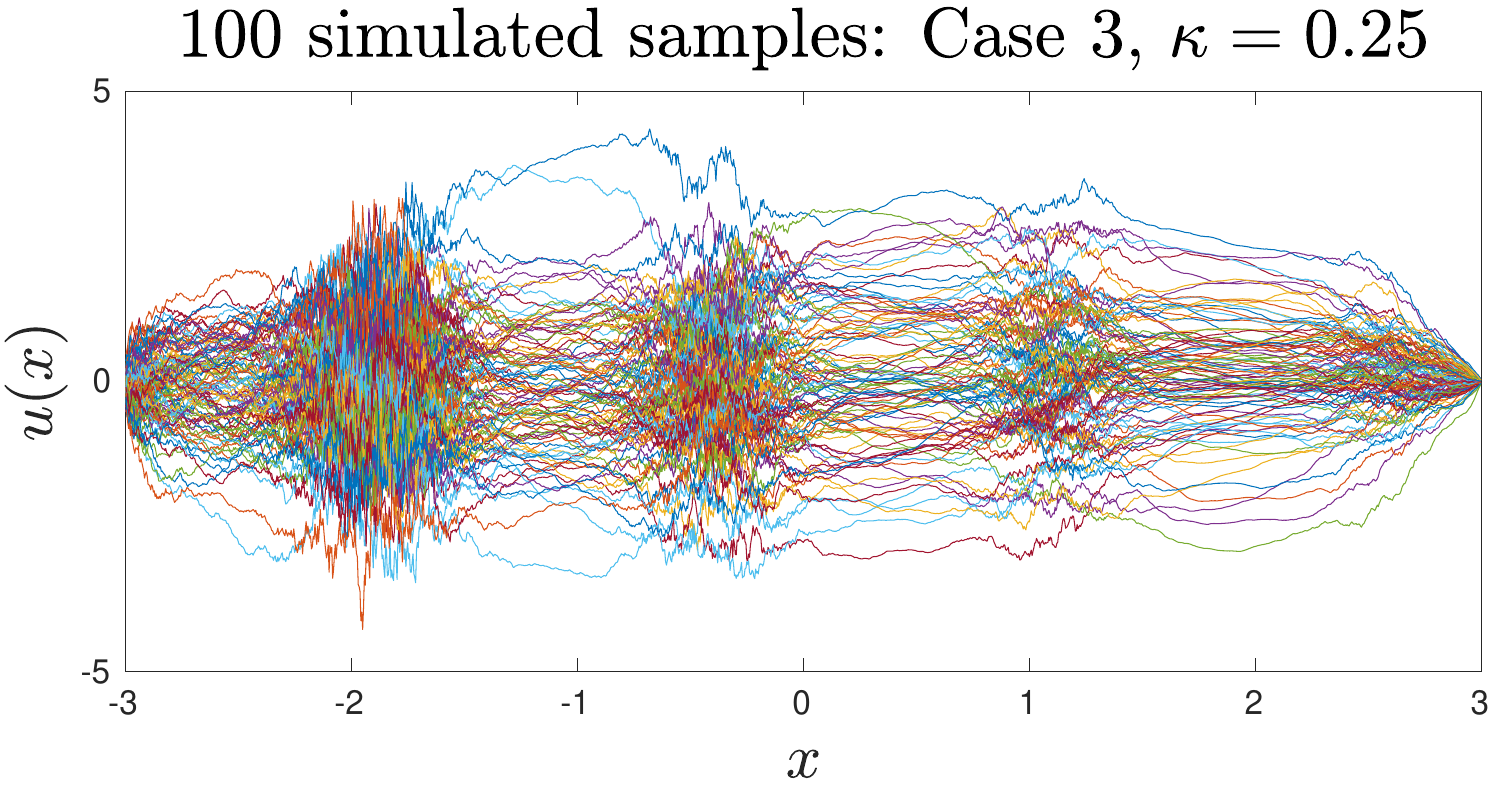}
\end{center} 
	\caption{The images show 100 generated samples (out of the 1000 samples) in all three cases when $\kappa=2.50$ (left column) and $\kappa=0.25$ (right column).} \label{fig:var_s_samples_k=2.5_and0.25} 
\end{figure}
Interestingly, for Case~1, one can choose positive constants $\kappa$, $\mu$ (or a positive function $\mu(x)$) and a radius $\delta>0$ so that, for every point $x\in\mathcal D$ sufficiently far from the boundary, the covariance satisfies $C_{\overline s}(x,y)\le C_s(x,y)\le C_{\underline s}(x,y)$ for every $y\in (x-\delta,x+\delta)$.
For example, let's take $\nu_{\underline s}=0.2$ (hence $\underline s=\nu_{\underline s}/2+1/4=  0.35$), $\kappa_{\underline s}=1.5$, $\nu_{\overline s}=1.2$ (hence $\overline s=0.85$), $\kappa_{\overline s}=2.5$, and set $\kappa_s=(\kappa_{\underline s}+\kappa_{\overline s})/2 =  2$ and $\mu=1$. We consider three sampling scenarios, namely a) constant $s=\underline s$, b) constant $s=\overline s$, and c) the piecewise-constant $s(x)$ in Case 1. Then we observe that
$C_{\overline s}(x,y)\le C_s(x,y)\le C_{\underline s}(x,y)$ when $x=-1.5,0,1.5$ and all $y\in\mathcal D$; these results are illustrated in \autoref{fig:cov_comp_case1}. In fact, with this choice of $\kappa$ and $\mu$, these inequalities are true for all $x,y\in \mathcal D$ in this particular example. 
\begin{figure}[!htbp] 
\begin{center} 
	\includegraphics[width=0.32\textwidth]{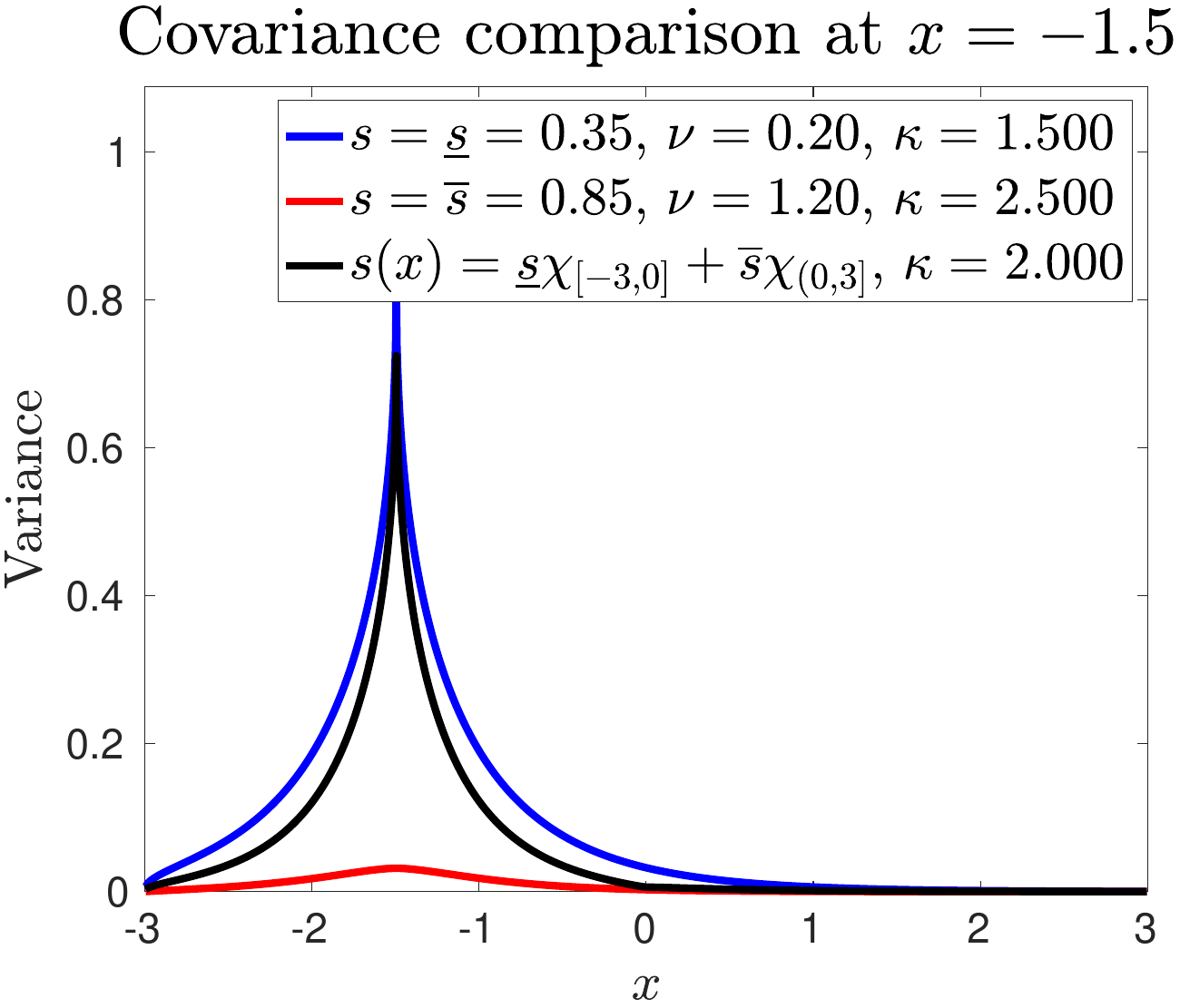} 	\includegraphics[width=0.32\textwidth]{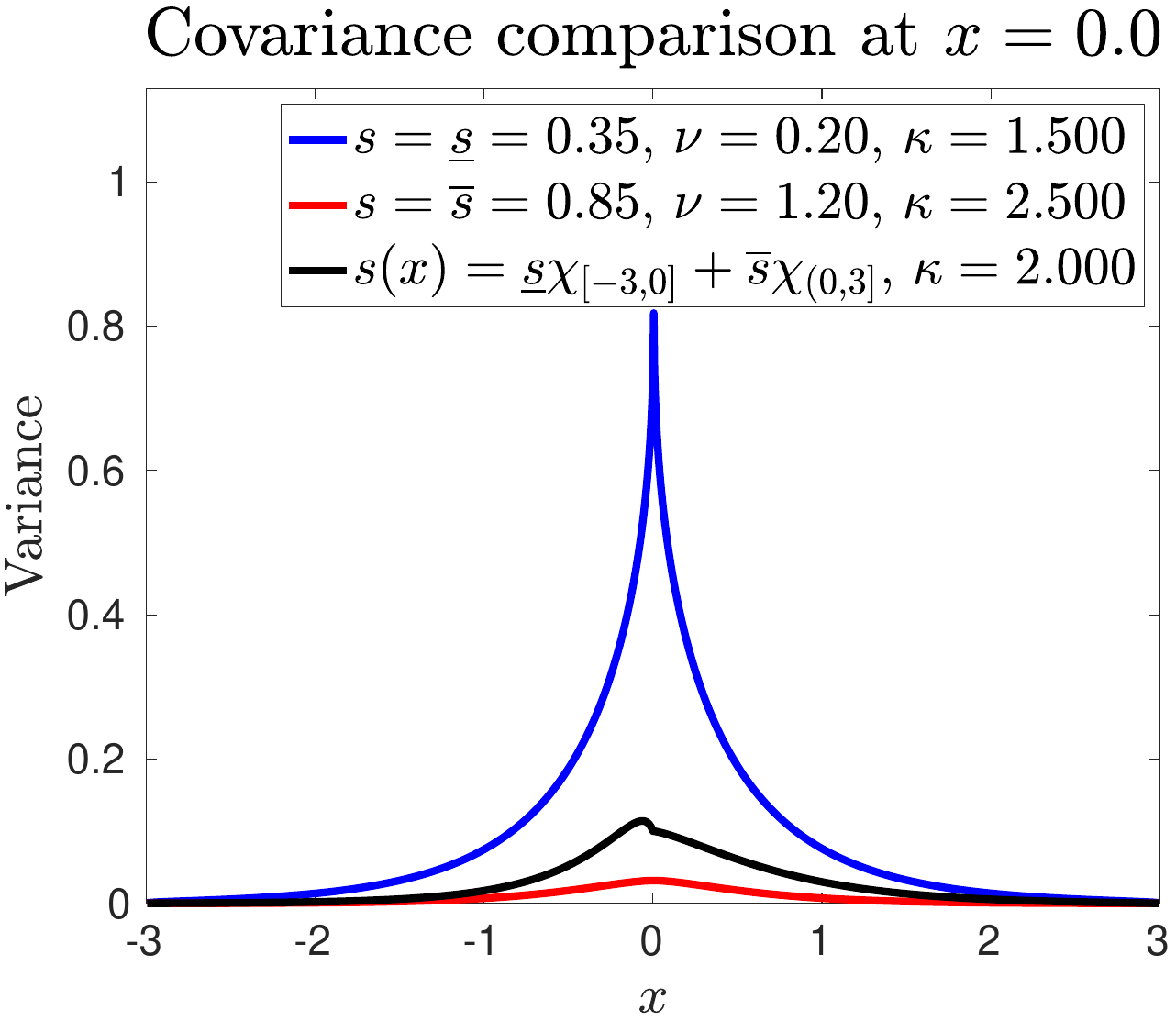}
    \includegraphics[width=0.32\textwidth]{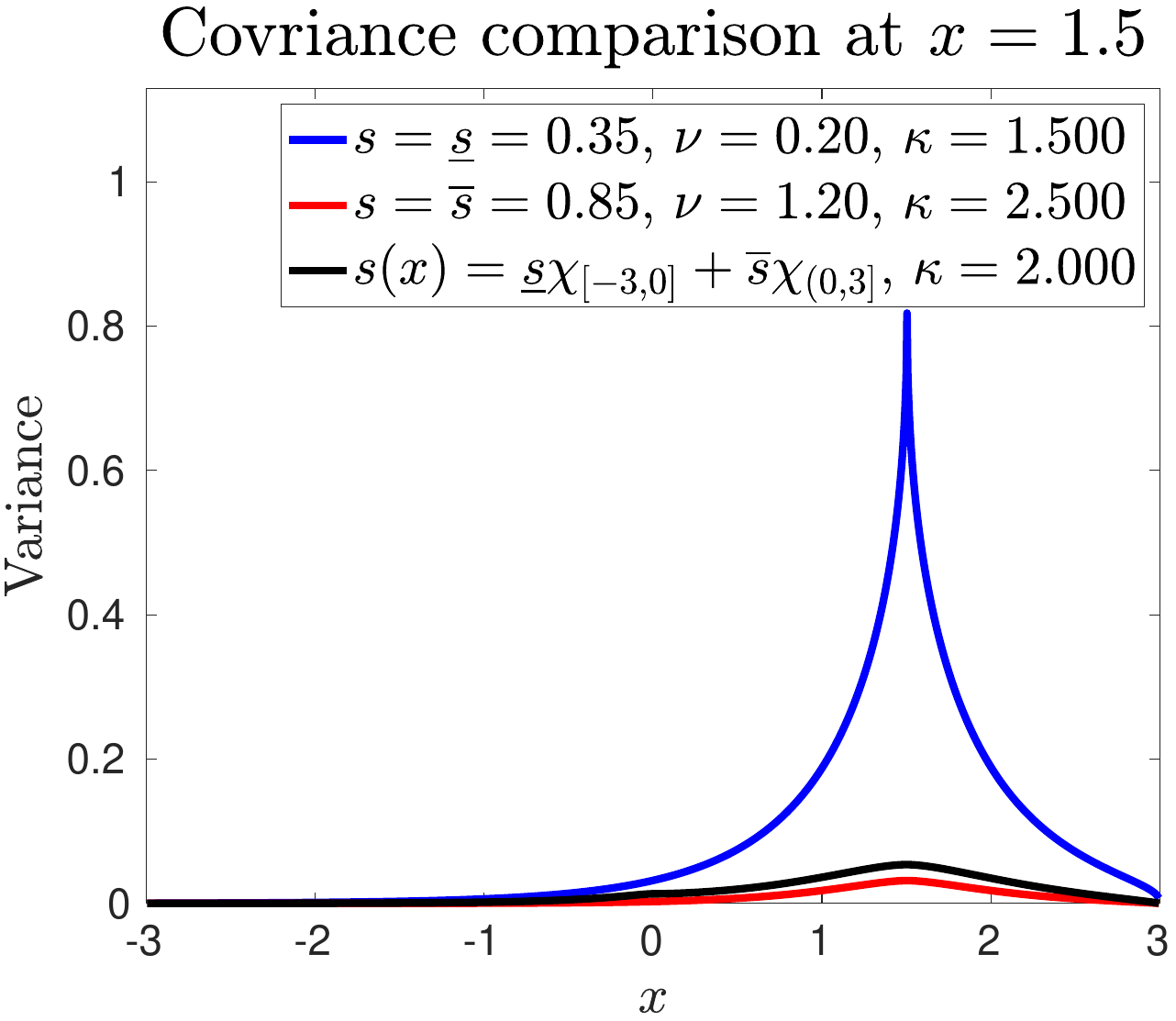}
\end{center}
	\caption{Comparison of the solution covariance profiles $C_s(x,\cdot)$ for the variable--smoothness case $s = s(x)$ in Case~1 and the two constant--smoothness cases $s = \underline{s} = 0.35$ and $s = \overline{s} = 0.85$. The panels shows the covariances at $x=-1.5$, $x=0$ and $x=1.5$. }
\label{fig:cov_comp_case1} 
\end{figure}

We repeat the same scenarios for the Gaussian bump profile (Case~2) and the oscillatory ramp profile (Case~3). The corresponding results for Cases~2 and 3 are given in \autoref{fig:cov_comp} (bottom-left and bottom-right panels, resp.) and \autoref{fig:var_s_samples_k=2.5_and0.25} (second and third rows, resp.). Finally, in \autoref{fig:var_s_3dCov_k=2.5_and0.25}, we plot the solution covariance for all three cases for both $\kappa=2.5$ and $\kappa=0.25$.

\begin{figure}[!htbp] 
\begin{center} 	 
	\includegraphics[width=0.28\textwidth]{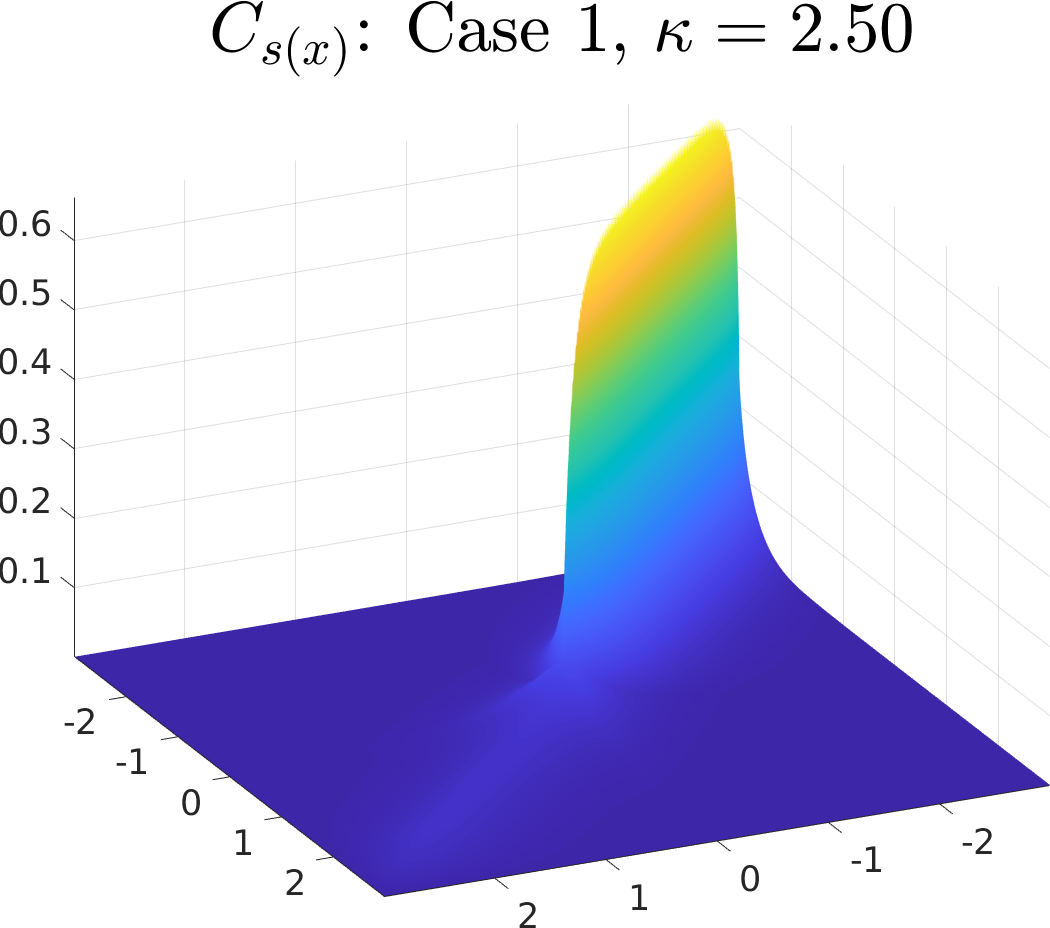} 
	\includegraphics[width=0.28\textwidth]{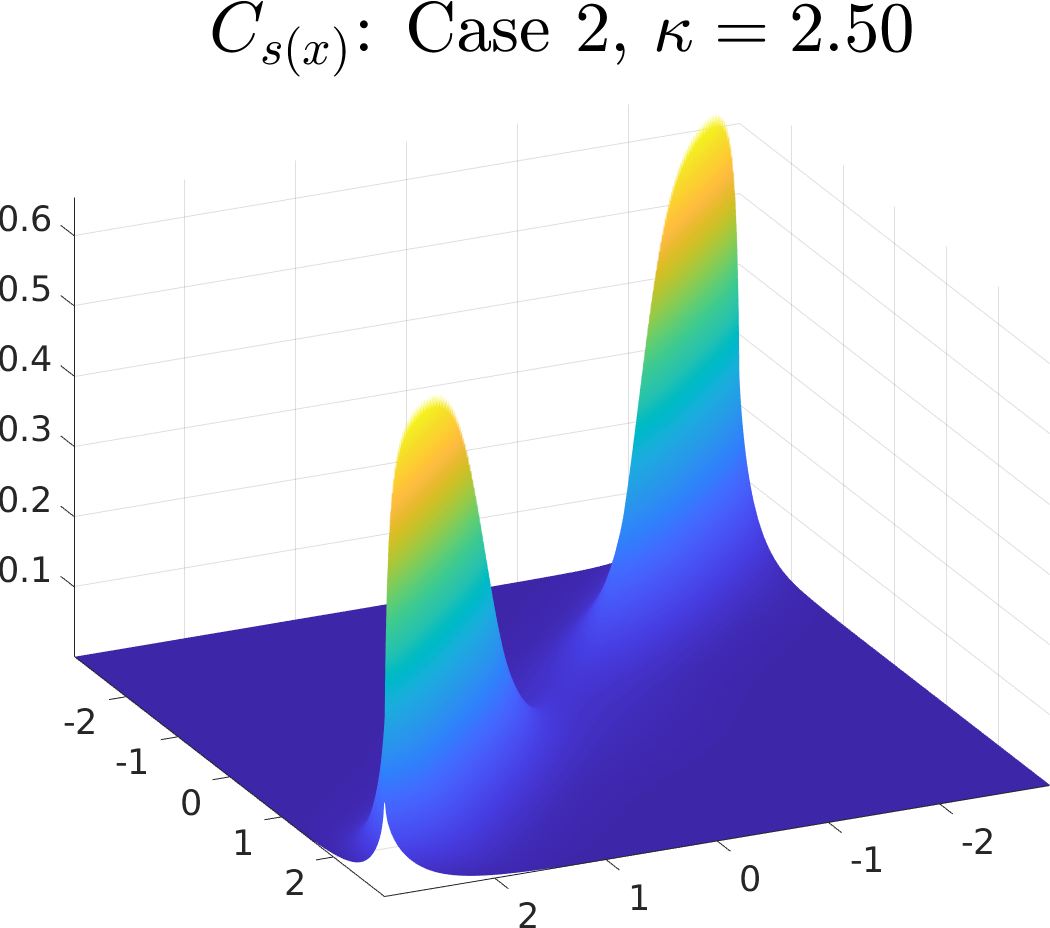}
	\includegraphics[width=0.28\textwidth]{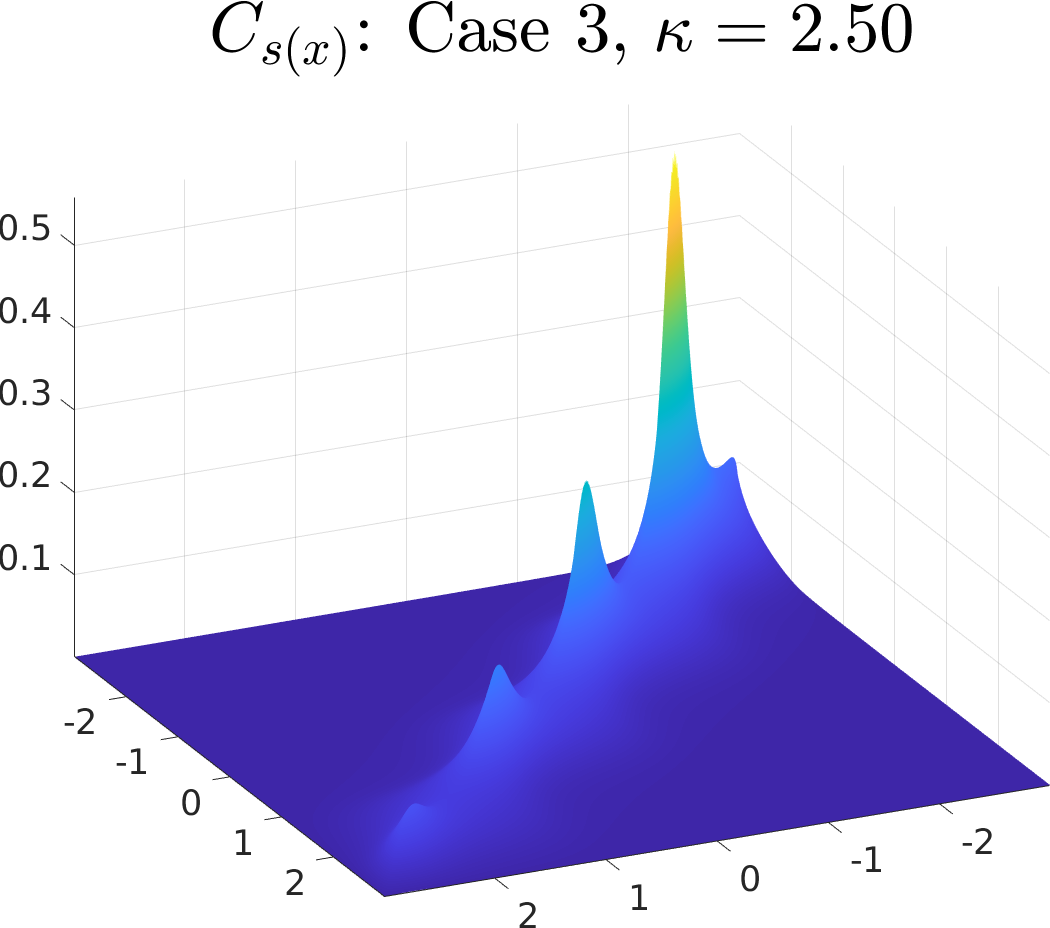}\\
	\includegraphics[width=0.28\textwidth]{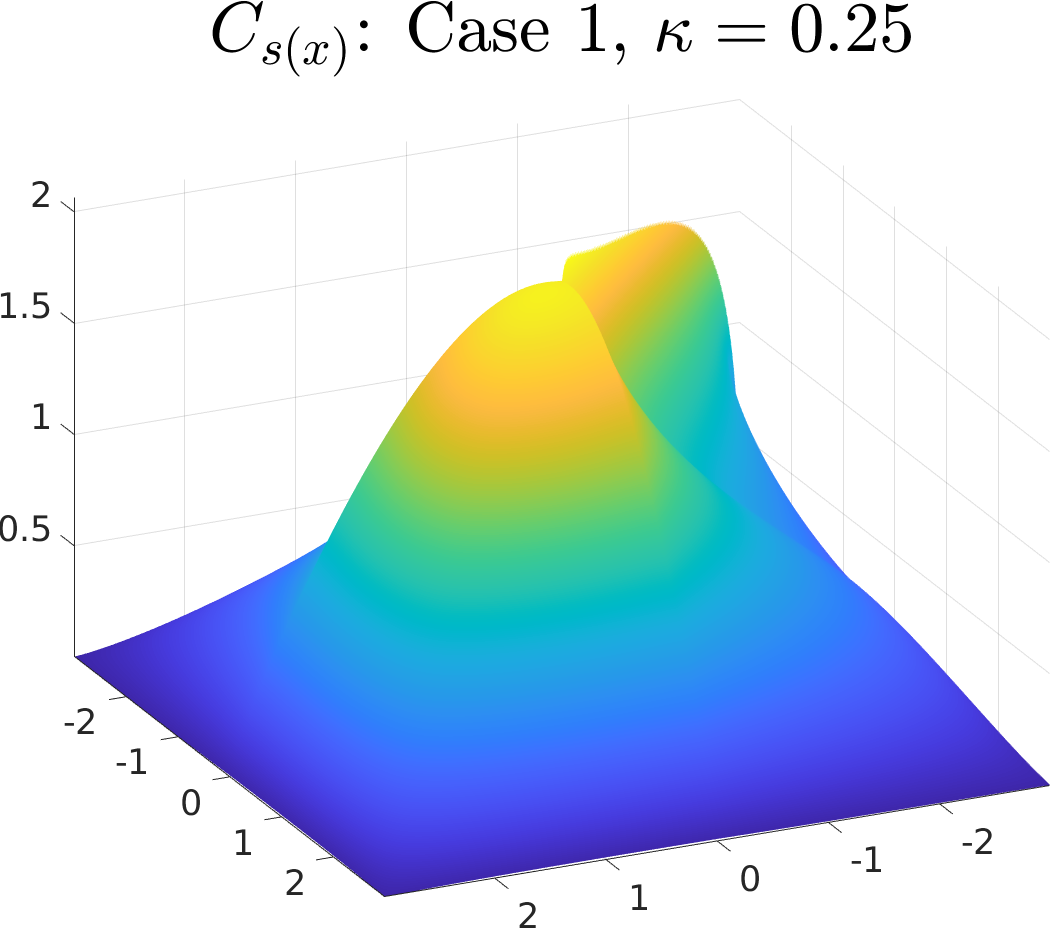}
	\includegraphics[width=0.28\textwidth]{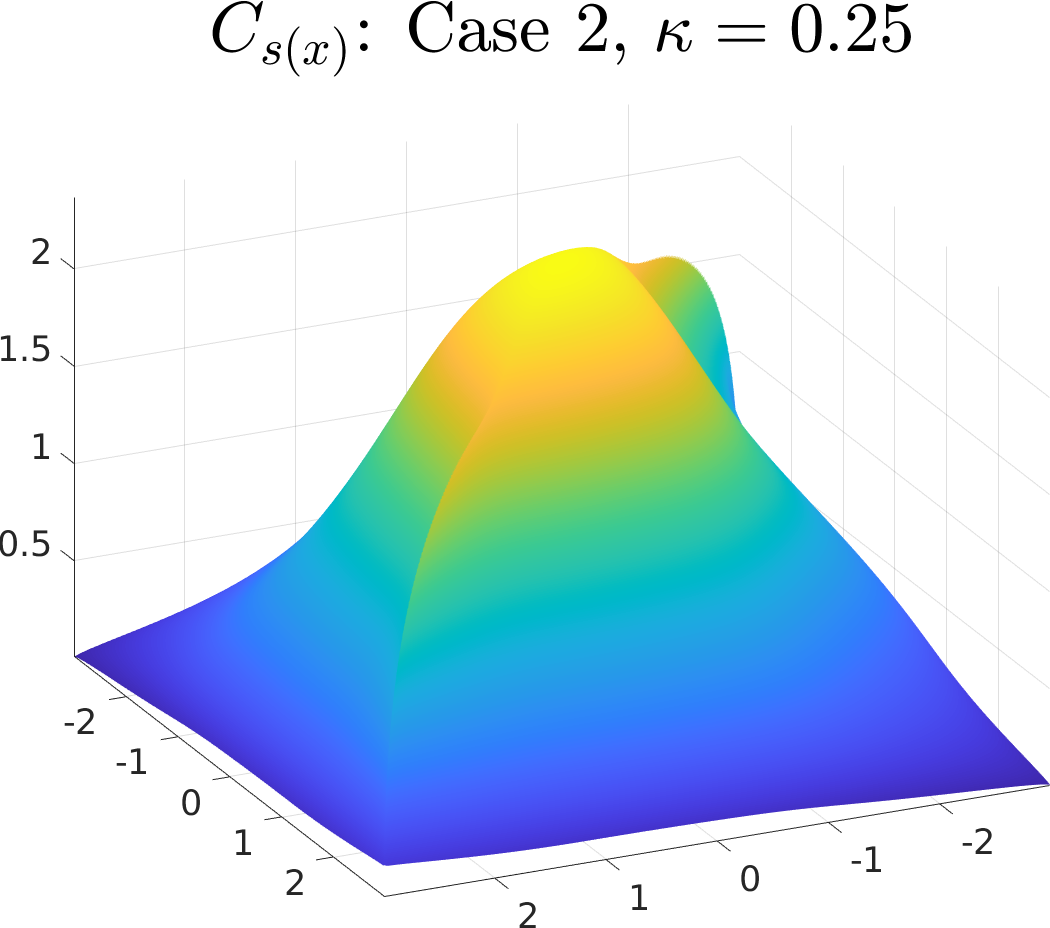} 
	\includegraphics[width=0.28\textwidth]{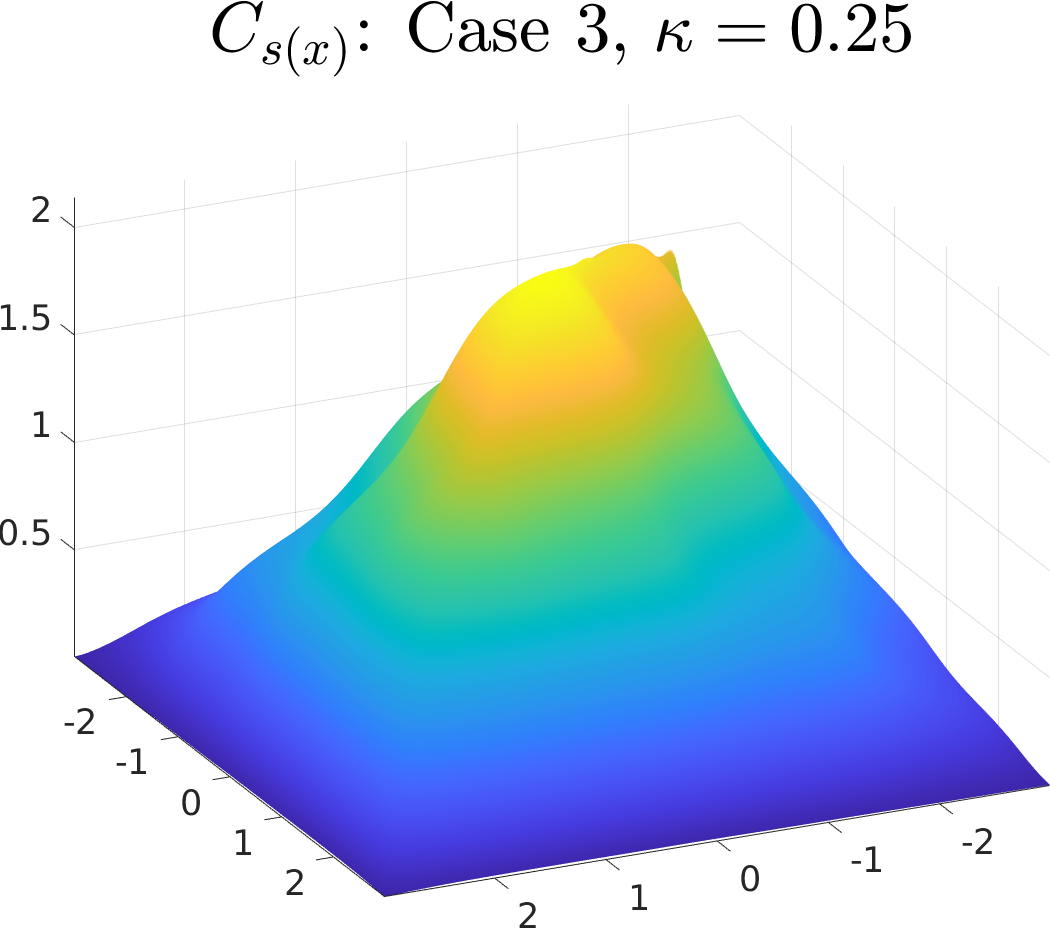}
\end{center} 
	\caption{The solution covariance $C_s$ for $\kappa=2.5$ (top row) and $\kappa=0.25$ (bottom row) for cases 1--3.} \label{fig:var_s_3dCov_k=2.5_and0.25} 
\end{figure}

\subsection{FE Convergence}
In this section, we use the same $R_{int}$ and $R_{ext}$ as in the previous section and set $\kappa=2.5$. For each Monte Carlo sample, we draw a fine-scale Gaussian vector $Z_f$ and define the fine right-hand side $b_f = L_f Z_f$, where $L_f$ is the lower-triangular Cholesky factor of the mass matrix $M_f$. The corresponding coarse problem is driven by the restricted load $b_c = P_{f,c}^\top b_f$, where $P_{f,c}$ denotes the canonical FEM injection (interpolation) operator from the coarse space to the fine space, that is, $(P_{f,c}\,u_c)_i = u_c(x_f^i)$, where $x_f$ is the fine mesh. 
Since the true solution \added[id=rev]{is unavailable, we estimate the convergence rates using a self-rate study \cite{Roache1998}.} Let $\ell$ be some given positive integer that corresponds to the fine level, that is, $h_l=2^{-\ell}$ and $U_\ell = \{U_\ell^{(i)}\}$ is the collections of samples generated at this level. Set $
E_{\ell} = \mathbb{E}\left[\bigl\| U_\ell - P_{\ell,\ell-1}\, U_{\ell-1} \bigr\|_{M_{\ell}}^2 \right]^{1/2} = \Big[\frac{1}{m}\sum_{i=1}^m (U_\ell^{(i)} - P_{\ell,\ell-1} U_{\ell-1}^{(i)})^\top M_\ell (U_\ell^{(i)} - P_{\ell,\ell-1} U_{\ell-1}^{(i)})\Big]^{1/2},$
where $\|\cdot\|_{M}$ is the mass-matrix $L_2$ norm. Alternatively, one can set $E_{\ell} = \mathbb{E}\left[\bigl\| U_\ell - P_{\ell,\ell-1}\, U_{\ell-1} \bigr\|^2 \right]^{1/2} \approx \left[ \sum_{\tau \in \mathcal{T}_\ell} h \sum_{q=1}^{Q} w_q \, \bigl( U_\ell(x_q) -  P_{\ell,\ell-1} U_{\ell-1}(x_q) \bigr)^2 \right]^{1/2}$, where $\{x_q\}$ are Gauss-quadrature points and $\{w_q\}$ are the corresponding weights. Then $\hat{r}=\log(E_{\ell-1}/E_{\ell})/\log(2)$ is an approximation of the convergence rate for large enough $\ell$. This requires available samples at three consecutive levels $\ell$, $\ell-1$ and $\ell-2$. 

\Crefrange{tbl:const_s_rhat}{tbl:var_s_rhat} report the estimated convergence rates for three different scenarios obtained from levels 9,8 and 7.
In the first scenario, the fractional order $s$ is constant, and the observed
rate is close to $2s-d/2$, in agreement with the theoretical result of
\cite{Bolin2020} for constant fractional order.  In the second scenario, $s$ varies spatially according to the three cases defined in~\eqref{eq:s_cases}, with
$\underline{s}=0.35$ and $\overline{s}=0.85$. The third scenario uses the same
variable-order profiles as in the second, but with a higher lower bound
$\underline{s}=0.65$; specifically in Case 3, we set $a=0.6705$, $b = 0.8297$, and $\omega= 0.05$ to achieve $(\underline s, \overline s) = (0.65,0.85)$. For the variable-order cases, the convergence rate is close to $2\underline s - 0.5$ in the second scenario $(\underline s, \overline s) = (0.35, 0.85)$ and closer to $2\langle s\rangle-0.5$ in the third scenario $(\underline s, \overline s) = (0.65, 0.85)$.

\begin{table}[h!]
\centering
\renewcommand{\arraystretch}{1.2}
\begin{tabular}{|c|c|c|c|c|}
\hline
$s$ 
& $0.3$ 
& $0.5$ 
& $0.7$ 
& $0.9$ \\
\hline
$r_\star \;/\; \hat r$ 
& $0.10 \;/\; 0.10$
& $0.50 \;/\;  0.51$
& $0.84 \;/\; 0.94$
& $1.12 \;/\; 1.43$ \\
\hline
\end{tabular}
\caption{Estimated strong $L_2$ convergence rates $\hat r$ for different constant fractional orders $s$, compared with the theoretical rate $r_\star$ of \sref{thm:spectral_underline_s}(ii); $r_\star=2s-\tfrac12$ for $s\le\tfrac12$.}
\label{tbl:const_s_rhat}
\end{table}
\begin{table}[h!]
\centering
\renewcommand{\arraystretch}{1.1}
\begin{tabular}{|l|c|c|c|c|}
\hline
\multirow{2}{*}{\quad $s(x)$} 
& \multicolumn{2}{c|}{$(\underline{s},\overline{s}) = (0.35,0.85)$}
& \multicolumn{2}{c|}{$(\underline{s},\overline{s}) = (0.65,0.85)$} \\ \cline{2-5}
& $r_\star \;\;/\;\; 2\langle s\rangle-\tfrac12$
& $\hat r$
& $r_\star \;\;/\;\; 2\langle s\rangle-\tfrac12$
& $\hat r$ \\ \hline
Step function
& $0.2 \;\;/\;\; 0.7$ & $0.20$
& $0.77 \;\;/\;\; 1.0$ & $0.83$ \\[0.3em]
Gaussian bump
& $0.2 \;\;/\;\; 0.47$ & $ 0.22$
& $0.77 \;\;/\;\; 0.91$ & $0.86$ \\[0.3em]
Oscillatory ramp
& $0.2 \;\;/\;\; 0.7$ & $0.26$
& $0.77 \;\;/\;\; 1.0$ & $0.93$ \\
\hline
\end{tabular}
\caption{Estimated convergence rates $\hat r$ for the three Cases in \eqref{eq:s_cases} at two choices of $(\underline{s},\overline{s})$, compared with the theoretical rate $r_\star$ ($=2\underline s-\tfrac12$ for $\underline s\le\tfrac12$) and the heuristic $2\langle s\rangle-\tfrac12$.}
\label{tbl:var_s_rhat}
\end{table}

\section{Conclusion}
\label{sec:conclusions}
We introduced and analyzed a nonlocal, integral-form generalization of Whittle–Mat\'ern Gaussian fields in which the fractional order varies in space. Working on a truncated bounded domain $\mathcal G$, we built an energy space adapted to the heterogeneous kernel induced by the modified Bessel function of the second kind with local exponent $\beta(x,y)=(s(x)+s(y))/2$ and smoothness field $\nu(x,y) = \beta(x,y) +d/2$, proved completeness of that energy space, and used form-representation and form-comparison arguments to obtain a self-adjoint positive operator with compact resolvent and two-sided spectral bounds. By means of the spectral expansion we constructed the Gaussian generalized solution, established uniqueness in the class of Gaussian generalized fields, and derived Sobolev regularity controlled by the minimal local order. In particular, our solution belongs to $H^r(\mathcal G)$ for every \replaced[id=rev]{$r<r_\star$, where $r_\star$ is given in \eqref{eq:r_star} and equals $2\underline s-\tfrac{d}{2}$ unless $d=1$ and $\underline s>1/2$}{$r<2\underline s-\tfrac{d}{2}$}, hence to $L_2(\mathcal G)$ when $\underline s>d/4$. Complementing the theory, we proposed a finite-element sampler for the integral model and presented one-dimensional numerical experiments that illustrate how spatially varying smoothness alters the solution covariance. The integral formulation leads to dense interaction matrices and nonlocal volume constraints that pose significant computational challenges in higher dimensions; in a follow-up work \cite{Ruzayqat2026sampling} we address these challenges via hierarchical low rank matrices as done in \cite{Lei2023} for nonlocal fractional Laplacian operators.
An alternative approach to reduce the computational costs would be to employ wavelet methods similar to those developed \cite{Schneider2009} for variable order pseudodifferential equations.

Our numerical experiments indicate that the integral-form, variable-order model behaves as expected and this concordance between theory and computation suggests that the framework can be embedded into extended inference and calibration pipelines, e.g., learning spatial heterogeneity from data. 
Developing these extensions is an important direction for future work. In addition, to make the models more flexible for statistical applications, an additional direction for future work is to remove the restriction $s\in(0,1)$, thereby allowing smoother fields.

\section*{Acknowledgments}
The work of HR, DB, GT and OK was supported by KAUST baseline fund. WL is partially supported by the National Natural Science Foundation of China
(Grant No. 12301496 and 12431015) and the Fundamental Research Fund for the Central Universities, University of Electronic Science and Technology of China (Grant No. Y030242063002065).

\appendix 
\label{sec:appendix}
\section{Proofs of Main Results in the Manuscript}
Assume that $s\in C(\overline{\mathcal G})$ and $0<\underline s\le s(x)\le\overline s<1, \forall x\in\overline{\mathcal G}$. Define $
\nu(x,y):=\frac d2+\beta(x,y)$.
For $x\neq y$ set $r=|x-y|$ and define
\begin{align*}
\tilde w_{-2s}(x,y)
:= \frac{1}{\pi^{d/2}}\frac{2^{1-\frac d2 + \beta(x,y)}}{|\Gamma(-\beta(x,y))|}
\kappa^{\nu(x,y)} K_{\nu(x,y)}(\kappa r)\, r^{\nu(x,y)} \quad and \quad \gamma(x,y):=\frac12\frac{\tilde w_{-2s}(x,y)}{r^{d+2\beta(x,y)}}.
\end{align*}
Constants $C,C_1,\dots$ may change from line to line but are explicit when crucial. We introduce the following crucial lemmas and proposition.

\subsection{Uniform Bessel asymptotics (uniform in index)}
We need Bessel bounds uniform for $\nu$ in the compact interval $I_\nu=[d/2+\underline s,d/2+\overline s]$.

\begin{lemma}[Uniform $K_\nu$ bounds]\label{lem:Bessel_uniform}
There exist constants $z_0>0$ and $C_{0},C_{1},C_{2}>0$ (depending only on $d,\underline s,\overline s$) such that
\begin{enumerate}[(i)]
\item For all $\nu\in I_\nu$ and $0<z\le z_0$, we have $\,
C_0\, z^{-\nu} \le K_\nu(z) \le C_1\, z^{-\nu}$.
\item For all $\nu\in I_\nu$ and $z\ge z_0$, we have $\,
K_\nu(z) \le C_2\, z^{-1/2} e^{-z}.
$
In particular, if $z_0\le z \le z_1$ for some $z_1>z_0$, then there exists a constant $c_2>0$ such that $
c_2\, z^{-1/2} e^{-z} \le K_\nu(z) \le C_2\, z^{-1/2} e^{-z}$.

\end{enumerate}
\end{lemma}

\begin{proof}
The uniform small- and large-argument asymptotics of $K_\nu$ on compact $\nu$--intervals are classical; as $z\downarrow0$, $K_\nu(z)\sim 2^{\nu-1}\Gamma(\nu)z^{-\nu}$, and as $z\to\infty$, $K_\nu(z)\sim\sqrt{\pi/(2z)}e^{-z}$, uniformly for $\nu\in I_\nu$; see \cite[Ch.~10]{Olver2010nist}, \cite{Watson1995}. Choosing $z_0>0$ and suitable constants $C_0,C_1,C_2,c_2>0$ yields the required bounds.
\end{proof}

\subsection{Two-regime kernel bounds}
This is the first crucial technical lemma we use repeatedly.

\begin{lemma}[Two-regime bounds for $\gamma$]\label{lem:gamma_two_regime}
There exist $r_0>0$ and positive constants $c_1,c_2,C_{\mathrm{tail}},C'$ (depending only on $d,\kappa,\underline s,\overline s$ and ultimately on $z_0,C_0,C_1,C_2$ of \sref{lem:Bessel_uniform}) such that
\begin{enumerate}[(a)]
\item For all $x,y\in\mathcal G$ with $0<r=|x-y|\le r_0$,
$c_1 r^{-(d+2\beta(x,y))} \le \gamma(x,y) \le c_2 r^{-(d+2\beta(x,y))}$.
\item For all $x,y\in\mathcal G$ with $r\ge r_0$, $0<\gamma(x,y)\le C_{\mathrm{tail}}$,\,  and $\,\gamma(x,y)\le C' r^{-1/2} e^{-\kappa r}$.
In particular, if $r_0 \leq r \leq r_1$ for some $r_1>r_0$, then  $c' > 0$ exists such that $ c' r^{-1/2} e^{-\kappa r} \le \gamma(x,y)\le C' r^{-1/2} e^{-\kappa r}$.
\end{enumerate}
\end{lemma}

\begin{proof}
Let $P(\beta):=\frac{1}{\pi^{d/2}}\frac{2^{1-\frac d2+\beta}}{|\Gamma(-\beta)|}\kappa^{\frac d2+\beta}$.  
Because $\beta\in[\underline s,\overline s]\subset(0,1)$, the function $\beta\mapsto P(\beta)$ is continuous on a compact set; therefore there exist $0<P_{\min}\le P_{\max}<\infty$ such that $P_{\min}\le P(\beta)\le P_{\max}$, $\forall\beta\in[\underline s,\overline s]$.
Next we apply \sref{lem:Bessel_uniform} for small argument with $z=\kappa r$. Choose $r_0 := z_0/\kappa$. For $0<r=|x-y|\le r_0$, we have $C_0 (\kappa r)^{-\nu} \le K_\nu(\kappa r)\le C_1(\kappa r)^{-\nu}$.
Therefore, $\tilde w_{-2s}(x,y) = P(\beta) K_\nu(\kappa r)\, r^{\frac d2+\beta}
\ge P_{\min} C_0 (\kappa r)^{-\nu} r^{\frac d2+\beta}=P_{\min} C_0 \kappa^{-\nu}$.
Since $\nu\in I_\nu$ is in a compact set, set $\widetilde c_1 := \min_{\nu\in I_\nu} P_{\min} C_0 \kappa^{-\nu} > 0$, so $\tilde w_{-2s}(x,y)\ge\widetilde c_1$ for all $0<r\le r_0$. Similarly upper bound gives $\tilde w_{-2s}(x,y)\le \widetilde c_2$ for some $\widetilde c_2>0$. Dividing by $2r^{d+2\beta}$ we get (a) with $c_1=\widetilde c_1/2$ and $c_2=\widetilde c_2/2$. For $r\ge r_0$ we have $z=\kappa r\ge z_0$ and \sref{lem:Bessel_uniform}(ii) yields
$K_\nu(\kappa r) \le C_2 (\kappa r)^{-1/2} e^{-\kappa r}$. Thus, $\tilde w_{-2s}(x,y)\le P_{\max} C_2 (\kappa r)^{-1/2} e^{-\kappa r} r^{\frac d2+\beta(x,y)}$. Hence
\begin{align*}
\gamma(x,y)
= \tfrac12 \tilde w_{-2s}(x,y) r^{-(d+2\beta(x,y))}
\le \tfrac12 P_{\max} C_2 \kappa^{-1/2} \, r^{-\nu(x,y)-1/2} e^{-\kappa r}.
\end{align*}
Since $\nu(x,y)\ge \nu_{\min}:=d/2+\underline s>0$ and $r\ge r_0>0$, $\gamma(x,y)\le C' r^{-1/2} e^{-\kappa r}$, 
with $C' := \tfrac12 C_2\kappa^{-1/2}\,\replaced[id=rev]{\max\big(r_0^{-\nu_{\min}},r_0^{-\nu_{\max}}\big)}{r_0^{-\nu_{\min}}}P_{\max}$ \added[id=rev]{(where $\nu_{\max}:=d/2+\overline s$; the maximum accounts for both $r_0\le1$ and $r_0>1$)} and $C_{\mathrm{tail}}:=C' r_0^{-1/2}e^{-\kappa r_0}.
$
The case when $r_0\leq r\leq r_1$ easily follows by the continuity of $\gamma(x,y) r^{1/2} e^{\kappa r}$ on compact sets.
\end{proof}

\subsection{Measurability of $\gamma$ and Finiteness of the Energy Seminorm $|\cdot|_{\mathbb V_{\kappa, s}}$ for Smooth Test Functions}
In this subsection we establish joint measurability of the kernel $\gamma(x,y)$ and that smooth compactly supported functions lie in $\mathbb V_{\kappa, s}$.

\begin{prop}\label{prop:meas_integrable}
Under the standing assumptions, the following hold.
\begin{enumerate}[(i)]
\item The function $(x,y)\mapsto\gamma(x,y)$ is measurable on $\mathcal G\times\mathcal G$.
\item For any $v\in C_c^\infty(\mathcal G)$ the double integral defining $|v|_{\mathbb V_{\kappa, s}}^2$ is finite\replaced[id=rev]{. Since membership in $\mathbb V_{\kappa, s}$ additionally requires $v=0$ a.e. in $\mathcal D^c_t$, it follows that $C_c^\infty(\mathcal D)\subset\mathbb V_{\kappa, s}$, where functions in $C_c^\infty(\mathcal D)$ are extended by zero to $\mathcal G$.}{, hence $C_c^\infty(\mathcal G)\subset\mathbb V_{\kappa, s}$.}
\end{enumerate}
\end{prop}

\begin{proof}
(i) The only potential issue is the division by $|x-y|^{d+2\beta}$ in $\gamma$ which is singular on the diagonal $x=y$, but the diagonal has Lebesgue measure zero in $\mathbb R^{2d}$ and measurability extends by defining the kernel arbitrarily on the diagonal. Therefore $\gamma$ is measurable.

(ii) Finiteness for $v\in C_c^\infty$. Fix $v\in C_c^\infty(\mathcal G)$. We split

\begin{align*}
\iint_{\mathcal G\times\mathcal G} (v(x)-v(y))^2\gamma(x,y)\,dy\,dx
= \iint_{|x-y|\le r_0} + \iint_{|x-y|>r_0} =: I_{\rm near}+I_{\rm far},
\end{align*}
with $r_0$ from \sref{lem:gamma_two_regime}. From the Lemma, $\gamma(x,y)\le C_{\mathrm{tail}}$ on $\{|x-y|>r_0\}$. Then

\begin{align*}
I_{\rm far}
\le C_{\mathrm{tail}} \iint_{|x-y|>r_0} (v(x)-v(y))^2\,dy\,dx &\le 2C_{\mathrm{tail}} \iint_{|x-y|>r_0} \big(v(x)^2+v(y)^2\big)\,dy\,dx \\
&= 4 C_{\mathrm{tail}} |\mathcal G| \int_{\mathcal G} v(x)^2\,dx <\infty,
\end{align*}
because $v\in L_2(\mathcal G)$ and $\mathcal G$ is bounded. For $|x-y|\le r_0$ we apply the near-diagonal bound $\gamma(x,y)\le c_2 |x-y|^{-(d+2\beta(x,y))}\le c_2 |x-y|^{-(d+2\underline s)}$ and also use the mean value theorem $v(x)-v(y)=\nabla v(\xi)\cdot (x-y)$, for some $\xi \in [x,y]$, to have

\begin{align*}
I_{\rm near}
&\le c_2 \|\nabla v\|_{L^\infty}^2 \iint_{|x-y|\le r_0} |x-y|^2 |x-y|^{-(d+2\underline s)}\,dy\,dx \\
&= c_2 \|\nabla v\|_{L^\infty}^2 \int_{\mathcal G} \int_{0}^{r_0} r^{2} r^{d-1} r^{-(d+2\underline s)}\,dr\,dx = c_2 \|\nabla v\|_{L^\infty}^2 |\mathcal G| \int_0^{r_0} r^{1-2\underline s}\,dr.
\end{align*}
The integral $\int_0^{r_0} r^{1-2\underline s}\,dr$ converges because the exponent $1-2\underline s>-1$ (since $\underline s<1$). Thus $I_{\rm near}<\infty$. This shows that $|v|_{\mathbb V_{\kappa, s}}^2<\infty$\replaced[id=rev]{ for every $v\in C_c^\infty(\mathcal G)$. If in addition $\operatorname{supp}v\subset\mathcal D$, then $v=0$ on $\mathcal D^c_t$ and hence $v\in\mathbb V_{\kappa, s}$.}{ and hence $v\in\mathbb V_{\kappa, s}$.}
\end{proof}

\subsection{Proof of \sref{thm:V_complete}}
\label{sec:proof_V_complete}

\begin{proof}
Let $(u_n)\subset\mathbb V_{\kappa, s}$ be Cauchy. Then by definition of the norm $\|u\|_{\mathbb V_{\kappa, s}}$ both sequences $(\kappa^{s}u_n)$ in $L_2(\mathcal G)$ and $\phi_n(x,y):=u_n(x)-u_n(y)$ in $L_2(\mathcal G\times\mathcal G,\gamma\,dx\,dy)$ are Cauchy. 
Here the Hilbert space $L_2(\mathcal G\times\mathcal G,\gamma\,dx\,dy)$ is well-defined by \sref{prop:meas_integrable}, which established measurability and integrability of $\gamma$. Hence there exist $w\in L_2(\mathcal G)$ and $\Phi\in L_2(\mathcal G\times\mathcal G,\gamma\,dx\,dy)$ with $\kappa^s u_n \to w$ in $L_2(\mathcal G)$, and $\phi_n \to \Phi $ in $L_2(\mathcal G\times\mathcal G,\gamma\,dx\,dy)$. Since $\kappa^s$ is bounded above and below by positive constants, we may set $u:=\kappa^{-s}w\in L_2(\mathcal G)$, and then $u_n\to u$ in $L_2(\mathcal G)$. 
A standard subsequence argument yields $\Phi(x,y)=u(x)-u(y)$ a.e., so $|u|_{\mathbb V_{\kappa, s}}<\infty$ and $u\in\mathbb V_{\kappa, s}$. Finally $\|u_n-u\|_{\mathbb V_{\kappa, s}}^2
=\|\kappa^s u_n-w\|_{L_2(\mathcal G)}^2+\|\phi_n-\Phi\|_{L_2(\mathcal G\times\mathcal G,\gamma\,dx\,dy)}^2 \to 0$.
Thus $\mathbb V_{\kappa, s}$ is complete.
\end{proof}

\subsection{Proof of \texorpdfstring{\sref{prop:density_moll}}{}}
\label{sec:proof_density_moll}

\begin{revblock}
Throughout this subsection $\bar v$ denotes the extension by zero to $\mathbb R^d$ of a function $v\in\mathbb V_{\kappa, s}$, $\rho\in C_c^\infty(B_1(0))$ is a fixed mollifier with $\rho\ge0$ and $\int\rho=1$, $\rho_\varepsilon(x):=\varepsilon^{-d}\rho(x/\varepsilon)$, $v_\varepsilon:=\rho_\varepsilon*\bar v$, and $\Delta_h w(x):=w(x+h)-w(x)$ for $h\in\mathbb R^d$. We shall use repeatedly that mollification commutes with finite differences, $\Delta_h v_\varepsilon=\rho_\varepsilon*\Delta_h\bar v$. Since $v=0$ a.e. on $\mathcal D^c_t$ and $\mathcal D^c_t$ contains a neighbourhood of $\partial\mathcal D$, we have $\operatorname{supp}\bar v\subseteq\overline{\mathcal D}$ and $\varepsilon_0:=\operatorname{dist}(\overline{\mathcal D},\mathbb R^d\setminus\mathcal G)>0$; the approximants of \sref{prop:density_moll} are the mollifications $v_\varepsilon$, and \emph{sufficiently small} there means $\varepsilon<\varepsilon_0$, which guarantees $\mathcal D^+_\varepsilon\subset\mathcal G$. Decreasing the radius $r_0$ of \sref{lem:gamma_two_regime} if necessary, we assume from now on that $r_0\le\min\{1,\varepsilon_0/2\}$; the conclusions of \sref{lem:gamma_two_regime} are of course preserved. Finally, we extend $s$ to $\mathbb R^d$ keeping the bounds $\underline s\le s\le\overline s$ and \eqref{eq:log_holder}, with a possibly larger constant; this is a purely notational device, since every function appearing below vanishes outside $\mathcal D^+_{2\varepsilon}\subset\mathcal G$.

For variable $s$ the kernel $\gamma$ is neither translation invariant nor homogeneous near the diagonal, where by \sref{lem:gamma_two_regime} it behaves like $|x-y|^{-(d+2\beta(x,y))}$. The proof below rests on replacing this order by the order $s(x)$ frozen at one of the two points, which costs the factor
\begin{align*}
|x-y|^{-2(\beta(x,y)-s(x))}=\exp\Big(\left(s(y)-s(x)\right)\log\frac{1}{|x-y|}\Big).
\end{align*}
This stays bounded as $|x-y|\to0$ only if the modulus of continuity $\omega$ of $s$ satisfies $\omega(r)\log(1/r)=O(1)$, that is, $\omega(r)\lesssim1/\log(1/r)$. Uniform continuity alone does not suffice; the slower-decaying modulus $\omega(r)=1/\log\log(1/r)$ still tends to $0$ as $r\downarrow0$, and is therefore admissible for a uniformly continuous $s$, yet it gives $\omega(r)\log(1/r)=\log(1/r)/\log\log(1/r)\to\infty$. This is the reason for including \eqref{eq:log_holder} among the standing assumptions, and it is the step that replaces the translation invariance and the homogeneity available for constant $s$. Most other arguments in this article use only the crude bounds of \sref{lem:gamma_two_regime}, in which the variable order is majorized and minorized by the constants $\overline s$ and $\underline s$, so that two distinct values of $s$ never have to be compared; besides the present subsection, \eqref{eq:log_holder} is genuinely used only in the proof of \sref{prop:reg_variable}, where the variable order is frozen at a point by the same mechanism.

\begin{lemma}[Frozen-order comparison]
\label{lem:frozen_order}
Assume \eqref{eq:log_holder}. Then there exist constants $0<c_\ast\le C_\ast<\infty$, depending only on $d,\kappa,\underline s,\overline s$ and $C_{\log}$, such that
\begin{align}
\label{eq:frozen}
c_\ast\,|h|^{-(d+2s(x))}\;\le\;\gamma(x,x+h)\;\le\;C_\ast\,|h|^{-(d+2s(x))},\qquad 0<|h|\le r_0 .
\end{align}
Consequently, setting for $0<\varrho\le r_0$
\begin{align}
\label{eq:N_rho}
N_\varrho(w):=\int_{|h|\le\varrho}\int_{\mathbb R^d}|\Delta_h w(x)|^2\,|h|^{-(d+2s(x))}\,dx\,dh ,
\end{align}
we have, for every measurable $w$ vanishing outside $\mathcal D^+_{\varepsilon_0/2}$,
\begin{align}
\label{eq:frozen_seminorm}
c_\ast\,N_{r_0}(w)\;\le\;\iint_{\{|x-y|\le r_0\}}(w(x)-w(y))^2\,\gamma(x,y)\,dy\,dx\;\le\;C_\ast\,N_{r_0}(w) .
\end{align}
\end{lemma}

\begin{proof}
By \sref{lem:gamma_two_regime}(a) there are $0<c_1\le c_2$ with $c_1|h|^{-(d+2\beta(x,x+h))}\le\gamma(x,x+h)\le c_2|h|^{-(d+2\beta(x,x+h))}$ for $0<|h|\le r_0$. Since $\beta(x,x+h)-s(x)=\tfrac12\left(s(x+h)-s(x)\right)$,
\begin{align*}
|h|^{-(d+2\beta(x,x+h))}=|h|^{-(d+2s(x))}\,\exp\Big(\left(s(x+h)-s(x)\right)\log\frac{1}{|h|}\Big),
\end{align*}
and, by \eqref{eq:log_holder} and $|h|\le r_0\le1$,
\begin{align*}
\left|s(x+h)-s(x)\right|\log\frac{1}{|h|}\le\left|s(x+h)-s(x)\right|\log\Big(e+\frac{1}{|h|}\Big)\le C_{\log} .
\end{align*}
Hence the exponential factor lies in $[e^{-C_{\log}},e^{C_{\log}}]$, which gives \eqref{eq:frozen} with $c_\ast=c_1e^{-C_{\log}}$ and $C_\ast=c_2e^{C_{\log}}$. For \eqref{eq:frozen_seminorm} we substitute $y=x+h$ and note that, since $r_0\le\varepsilon_0/2$ and $w$ vanishes outside $\mathcal D^+_{\varepsilon_0/2}$, the difference $\Delta_hw(x)$ vanishes for $|h|\le r_0$ unless both $x$ and $x+h$ lie in $\mathcal G$; the double integral in \eqref{eq:frozen_seminorm} and the integral \eqref{eq:N_rho} are therefore taken over the same set of pairs, and \eqref{eq:frozen} applies pointwise.
\end{proof}

We can now prove \sref{prop:density_moll}.

\begin{proof}
Let $0<\varepsilon<\varepsilon_0$. If $\operatorname{dist}(x,\mathcal D)>\varepsilon$ then $\bar v$ vanishes a.e. on $B_\varepsilon(x)$, whence $v_\varepsilon(x)=\int_{B_\varepsilon(x)}\rho_\varepsilon(x-z)\bar v(z)\,dz=0$; therefore $\operatorname{supp}v_\varepsilon\subseteq\mathcal D^+_\varepsilon$, and $\mathcal D^+_\varepsilon\subset\mathcal G$ by the definition of $\varepsilon_0$; hence $v_\varepsilon\in C_c^\infty(\mathcal G)$ and $v_\varepsilon$ vanishes on $\mathcal D^c_t\setminus\mathcal C_\varepsilon$. Moreover $\mathcal D^+_\varepsilon\downarrow\overline{\mathcal D}$ as $\varepsilon\downarrow0$, and $|\partial\mathcal D|=0$ because $\mathcal D$ is a bounded Lipschitz domain, so $|\mathcal C_\varepsilon|\to0$. There remains the convergence of the energy norms, which concerns the limit $\varepsilon\downarrow0$; we may therefore assume from now on that $\varepsilon<r_0/4$, recalling that $r_0\le\min\{1,\varepsilon_0/2\}$.

Set $w_\varepsilon:=v_\varepsilon-v$, extended by zero, so that $w_\varepsilon$ vanishes outside $\mathcal D^+_\varepsilon\subset\mathcal D^+_{\varepsilon_0/2}$ and $\|w_\varepsilon\|_{L_2}\to0$ by the $L_2$-continuity of mollification. For the mass term, $\|\kappa^{s(\cdot)}w_\varepsilon\|_{L_2(\mathcal G)}\le\big(\sup_{\overline{\mathcal G}}\kappa^{s}\big)\|w_\varepsilon\|_{L_2}\to0$. For the seminorm we split at the radius $r_0$, as in the classical argument. On $\{|x-y|>r_0\}$ the kernel is bounded by a constant $C_{\rm tail}$ by \sref{lem:gamma_two_regime}(b), whence
\begin{align*}
\iint_{\{|x-y|>r_0\}}(w_\varepsilon(x)-w_\varepsilon(y))^2\gamma(x,y)\,dy\,dx\le 4\,C_{\rm tail}\,|\mathcal G|\,\|w_\varepsilon\|_{L_2}^2\longrightarrow0 .
\end{align*}
By \eqref{eq:frozen_seminorm} it therefore suffices to show that $N_{r_0}(w_\varepsilon)\to0$, and we split this quantity at the mollification scale $N_{r_0}(w_\varepsilon)=N_{\varepsilon}(w_\varepsilon)+\big(N_{r_0}(w_\varepsilon)-N_{\varepsilon}(w_\varepsilon)\big)$.
Note that, by \eqref{eq:frozen_seminorm} applied to $\bar v$, we have $N_{r_0}(\bar v)\le c_\ast^{-1}|v|_{\mathbb V_{\kappa, s}}^2<\infty$, so that $N_{\varrho}(\bar v)\to0$ as $\varrho\downarrow0$ by dominated convergence. This finiteness is what drives both steps below.

Since $|\Delta_hw_\varepsilon|^2\le2|\Delta_hv_\varepsilon|^2+2|\Delta_h\bar v|^2$, the contribution of $\bar v$ is at most $2N_\varepsilon(\bar v)\to0$. For the contribution of $v_\varepsilon$ we exploit that $\int_{\mathbb R^d}\left(\rho_\varepsilon(y+h)-\rho_\varepsilon(y)\right)dy=0$, which allows us to subtract the value at $x$
\begin{align*}
\Delta_hv_\varepsilon(x)=\int_{\mathbb R^d}\big[\rho_\varepsilon(y+h)-\rho_\varepsilon(y)\big]\,\big[\bar v(x-y)-\bar v(x)\big]\,dy .
\end{align*}
We now bound the first bracket, for $|h|\le\varepsilon$, in two ways. Since $\operatorname{supp}\rho_\varepsilon\subseteq\overline{B_\varepsilon(0)}$, the term $\rho_\varepsilon(y)$ vanishes unless $|y|\le\varepsilon$, while $\rho_\varepsilon(y+h)$ vanishes unless $|y+h|\le\varepsilon$, hence unless $|y|\le|y+h|+|h|\le2\varepsilon$; the bracket therefore vanishes for $|y|>2\varepsilon$, and the integral may be restricted to the ball $\{|y|\le2\varepsilon\}$. Moreover $\nabla\rho_\varepsilon(z)=\varepsilon^{-d-1}(\nabla\rho)(z/\varepsilon)$, so that $\|\nabla\rho_\varepsilon\|_{L_\infty}=\varepsilon^{-d-1}\|\nabla\rho\|_{L_\infty}$, and the mean value theorem gives $\left|\rho_\varepsilon(y+h)-\rho_\varepsilon(y)\right|\le|h|\,\|\nabla\rho_\varepsilon\|_{L_\infty}=|h|\,\varepsilon^{-d-1}\|\nabla\rho\|_{L_\infty}$.
Applying the Cauchy--Schwarz inequality on $\{|y|\le2\varepsilon\}$, whose volume is $c_d\,\varepsilon^{d}$, we obtain
\begin{align*}
|\Delta_hv_\varepsilon(x)|^2
\le |h|^2\varepsilon^{-2d-2}\|\nabla\rho\|_{L_\infty}^2\cdot c_d\varepsilon^{d}\int_{|y|\le2\varepsilon}|\bar v(x-y)-\bar v(x)|^2\,dy ,
\end{align*}
that is,
\begin{align*}
|\Delta_hv_\varepsilon(x)|^2\le C\,\frac{|h|^2}{\varepsilon^{2}}\,G_\varepsilon(x),\qquad
G_\varepsilon(x):=\varepsilon^{-d}\int_{|y|\le2\varepsilon}|\bar v(x-y)-\bar v(x)|^2\,dy ,
\end{align*}
with $C=C(d,\rho)$. Because $2-2s(x)\ge2-2\overline s>0$, we have $\int_{|h|\le\varepsilon}|h|^{2-(d+2s(x))}dh\le C(d,\overline s)\,\varepsilon^{2-2s(x)}$, and therefore
\begin{align*}
\int_{|h|\le\varepsilon}\int_{\mathbb R^d}|\Delta_hv_\varepsilon(x)|^2\,|h|^{-(d+2s(x))}\,dx\,dh
\le C\int_{\mathbb R^d}G_\varepsilon(x)\,\varepsilon^{-2s(x)}\,dx .
\end{align*}
It is precisely here that the frozen order $s(x)$, rather than $\beta(x,y)$, makes the computation possible (the exponent no longer depends on the integration variable $h$). For $|y|\le2\varepsilon$ and $\varepsilon<1$ we have $\varepsilon^{-d}\le2^{d}|y|^{-d}$ and $\varepsilon^{-2s(x)}\le2^{2\overline s}|y|^{-2s(x)}$, so that, by Tonelli's theorem and the substitution $h=-y$,
\begin{align*}
\int_{\mathbb R^d}G_\varepsilon(x)\,\varepsilon^{-2s(x)}\,dx
\le C\int_{|y|\le2\varepsilon}\int_{\mathbb R^d}|\bar v(x-y)-\bar v(x)|^2\,|y|^{-(d+2s(x))}\,dx\,dy
= C\,N_{2\varepsilon}(\bar v),
\end{align*}
which tends to $0$ as $\varepsilon\downarrow0$ since $2\varepsilon<r_0$. As for the coarse scales $\varepsilon<|h|\le r_0$, fix $h$ and put $g_h:=\Delta_h\bar v\in L_2(\mathbb R^d)$ and $\mu_h(x):=|h|^{-2s(x)}$, so that $\Delta_hw_\varepsilon=\rho_\varepsilon*g_h-g_h$, $|h|^{-(d+2s(x))}=|h|^{-d}\mu_h(x)$ and
\begin{align*}
N_{r_0}(w_\varepsilon)-N_{\varepsilon}(w_\varepsilon)=\int_{\varepsilon<|h|\le r_0}F_\varepsilon(h)\,\frac{dh}{|h|^{d}},
\qquad
F_\varepsilon(h):=\int_{\mathbb R^d}\left|\rho_\varepsilon*g_h-g_h\right|^2\mu_h\,dx .
\end{align*}
For each fixed $h\neq0$ the weight $\mu_h$ is bounded above and below by positive constants, so $L_2(\mu_h\,dx)$ carries the same topology as $L_2$ and $F_\varepsilon(h)\to0$ as $\varepsilon\downarrow0$, again by the $L_2$-continuity of mollification. To dominate $F_\varepsilon$ uniformly, we use Jensen's inequality (recall $\rho_\varepsilon\ge0$ and $\int\rho_\varepsilon=1$) together with Tonelli's theorem,
\begin{align*}
\int_{\mathbb R^d}|\rho_\varepsilon*g_h|^2\mu_h\,dx
\le\int_{\mathbb R^d}|g_h(u)|^2\Big(\int_{\mathbb R^d}\rho_\varepsilon(z)\,\mu_h(u+z)\,dz\Big)du ,
\end{align*}
and observe that, for $|z|\le\varepsilon<|h|\le r_0\le1$,
\begin{align*}
\frac{\mu_h(u+z)}{\mu_h(u)}=\exp\Big(2\left(s(u+z)-s(u)\right)\log\frac{1}{|h|}\Big)
\le\exp\Big(2\left|s(u+z)-s(u)\right|\log\Big(e+\frac{1}{\varepsilon}\Big)\Big)\le e^{2C_{\log}},
\end{align*}
where the last inequality follows from \eqref{eq:log_holder} and $|z|\le\varepsilon$. This is the only place where the restriction $|h|>\varepsilon$ is used, and it is what makes the argument work for a variable order. Consequently $F_\varepsilon(h)\le2\big(1+e^{2C_{\log}}\big)\int|g_h|^2\mu_h\,dx$, and the majorant is integrable, $\int_{|h|\le r_0}\Big(\int_{\mathbb R^d}|g_h|^2\mu_h\,dx\Big)\frac{dh}{|h|^{d}}=N_{r_0}(\bar v)<\infty$.
Dominated convergence therefore gives $N_{r_0}(w_\varepsilon)-N_{\varepsilon}(w_\varepsilon)\to0$, which completes the proof.
\end{proof}
\end{revblock}

\subsection{Proof of \sref{thm:V_kappa_equiv}}

\begin{proof}
Write $\gamma_{\kappa}(x,y)=C(\beta)\,\kappa^{\nu}K_{\nu}(\kappa r)r^{-\nu}$ with $r=|x-y|$ and $\nu=d/2+\beta$. The argument splits into the $L_2$ mass term comparison and the seminorm comparison.

\emph{Mass term.} Since $s\in C(\overline{\mathcal G})$ and $0<\underline s\le s\le\overline s<1$, the continuous map $x\mapsto(\kappa/\kappa')^{s(x)}$ attains positive extrema on $\overline{\mathcal G}$, so there are $0<m_0\le M_0<\infty$ with $m_0\le(\kappa/\kappa')^{s(x)}\le M_0$ for all $x$. Multiplying by $\kappa'^{s(x)}|v(x)|$, squaring and integrating yields
\begin{align*}
m_0^2\|\kappa'^{s(\cdot)}v\|_{L_2(\mathcal G)}^2\le\|\kappa^{s(\cdot)}v\|_{L_2(\mathcal G)}^2\le M_0^2\|\kappa'^{s(\cdot)}v\|_{L_2(\mathcal G)}^2.
\end{align*}

\emph{Seminorm.} Let $\nu_{\min}=d/2+\underline s$, $\nu_{\max}=d/2+\overline s$ and take $z_0,C_0,C_1,C_2$ from \sref{lem:Bessel_uniform}. Set $r_{\min}=\min\{z_0/\kappa,z_0/\kappa'\}$, and $r_{\max}=\max\{z_0/\kappa,z_0/\kappa'\}$. Since $\mathcal G$ is bounded, $r\in(0,|\mathcal G|]$, and we consider three regimes.

(i) If $0<r\le r_{\min}$ then $\kappa r,\kappa' r\le z_0$ and the small--$z$ bounds give $C_0(\kappa r)^{-\nu}\le K_\nu(\kappa r)\le C_1(\kappa r)^{-\nu}$ (and likewise for $\kappa'$). Hence there exist $c_{\rm near},C_{\rm near}>0$ (depending only on $d,\underline s,\overline s$ and the Bessel constants) such that for  $0<r\le r_{\min}$,  $c_{\rm near}\,\gamma_{\kappa'}(x,y)\le\gamma_{\kappa}(x,y)\le C_{\rm near}\,\gamma_{\kappa'}(x,y)$.
(ii) If $r_{\min}<r<r_{\max}$ then one of $\kappa r,\kappa' r$ is $\le z_0$ and the other $\ge z_0$. The ratio $\gamma_\kappa(x,y)/\gamma_{\kappa'}(x,y)=\kappa^{\nu}K_\nu(\kappa r)/\kappa'^{\nu}K_\nu(\kappa' r)$ is continuous on the compact set $[r_{\min},r_{\max}]\times[\nu_{\min},\nu_{\max}]$, so it is bounded above and below by positive constants $c_{\rm mixed},C_{\rm mixed}$. Using the mass-term bounds to control any residual $\kappa$--powers if needed, we obtain $c_{\rm mixed}\,\gamma_{\kappa'}(x,y)\le\gamma_{\kappa}(x,y)\le C_{\rm mixed}\,\gamma_{\kappa'}(x,y)$ for $r_{\min}<r<r_{\max}$.
(iii) If $r\ge r_{\max}$ then $\kappa r,\kappa' r\ge z_0$ and the large--$z$ bound $K_\nu(z)\lesssim z^{-1/2}e^{-z}$ yields for suitable $\underline C,\overline C>0$, $
\underline C\,\kappa^{\nu-1/2}r^{-(\nu+1/2)}e^{-\kappa r}\le\gamma_\kappa(x,y)\le\overline C\,\kappa^{\nu-1/2}r^{-(\nu+1/2)}e^{-\kappa r}$,
and similarly for $\gamma_{\kappa'}$. Dividing these displays and using $\nu\in[\nu_{\min},\nu_{\max}]$ together with the mass bounds to absorb powers of $(\kappa/\kappa')$, we get constants $c_{\rm tail},C_{\rm tail}>0$ (depending only on $d,\underline s,\overline s,m_0,M_0,|\mathcal G|$) such that $c_{\rm tail}\,\gamma_{\kappa'}(x,y)\le\gamma_{\kappa}(x,y)\le C_{\rm tail}\,\gamma_{\kappa'}(x,y)$ for $r\ge r_{\max})$. Combining the three regimes gives global constants $0<c\le C<\infty$ (depending only on $d,\underline s,\overline s,m_0,M_0,|\mathcal G|$) with
\begin{align*}
c\,\gamma_{\kappa'}(x,y)\le\gamma_{\kappa}(x,y)\le C\,\gamma_{\kappa'}(x,y)\quad\text{for a.e. }(x,y)\in\mathcal G\times\mathcal G.
\end{align*}
Multiplying by $(v(x)-v(y))^2$ and integrating yields $c|v|_{\mathbb V_{\kappa',s}}^2\le|v|_{\mathbb V_{\kappa,s}}^2\le C|v|_{\mathbb V_{\kappa',s}}^2$. Together with the mass-term estimate above this proves the claimed norm equivalence.
\end{proof}

\subsection{Proof of \sref{thm:form_rep}}
\label{sec:proof_form_rep}

\begin{lemma}[Equivalence of form norm and $\mathbb V_{\kappa, s}$-norm]\label{lem:norm_equiv}
Let $\mathcal A^{s(\cdot)}(\cdot,\cdot)$ be the bilinear form,  define the form-norm $\|u\|_{\mathcal A^{s(\cdot)}}:=\big(\mathcal A^{s(\cdot)}(u,u)+\|u\|_{L_2(\mathcal G)}^2\big)^{1/2}$,
and recall the energy norm $\|u\|_{\mathbb V_{\kappa, s}}:=\mathcal A(u,u)^{1/2}$. Suppose $\mathcal A^{s(\cdot)}$ is coercive, that is, there exists $c_{\rm coerc}>0$ with $\mathcal A^{s(\cdot)}(u,u) \ge c_{\rm coerc} \|u\|_{L_2(\mathcal G)}^2$, $\forall u\in\mathbb V_{\kappa, s}$,
where in our setting we can take \replaced[id=rev]{$c_{\rm coerc}=\min(\kappa^{2\underline s},\kappa^{2\overline s})$, since the mass term gives $\mathcal A^{s(\cdot)}(u,u)\ge\int_{\mathcal D}\kappa^{2s(x)}u(x)^2\,dx\ge\inf_{x}\kappa^{2s(x)}\|u\|_{L_2(\mathcal G)}^2$ and $s\mapsto\kappa^{2s}$ is increasing for $\kappa\ge1$ and decreasing for $\kappa<1$}{$c_{\rm coerc}=\min(1,\kappa^{2\underline s})$}. Then for all $u\in\mathbb V_{\kappa, s}$ the following norm-equivalence holds
$\|u\|_{\mathbb V_{\kappa, s}} \le \|u\|_{\mathcal A^{s(\cdot)}} \le \sqrt{1+1/c_{\rm coerc}}\; \|u\|_{\mathbb V_{\kappa, s}}$.
Consequently completeness of $(\mathbb V_{\kappa, s},\|\cdot\|_{\mathbb V_{\kappa, s}})$ implies completeness of $(\mathbb V_{\kappa, s},\|\cdot\|_{\mathcal A^{s(\cdot)}})$ and hence the form $\mathcal A^{s(\cdot)}$ is closed.
\end{lemma}

\begin{proof}
Since $\|u\|_{\mathcal A^{s(\cdot)}}^2=\mathcal A^{s(\cdot)}(u,u)+\|u\|_{L_2(\mathcal G)}^2 \ge \mathcal A^{s(\cdot)}(u,u)=\|u\|_{\mathbb V_{\kappa, s}}^2$, the left inequality is immediate. For the right inequality we use the coercivity $\mathcal A^{s(\cdot)}(u,u)\ge c_{\rm coerc}\|u\|_{L_2(\mathcal G)}^2$. Notice that $\mathcal A^{s(\cdot)}(u,u)=\|u\|_{\mathbb V_{\kappa, s}}^2$. In particular, $\mathcal A^{s(\cdot)}(u,u) \ge \replaced[id=rev]{\min(\kappa^{2\underline s},\kappa^{2\overline s})}{\min(1,\kappa^{2\underline s})} \|u\|_{L_2(\mathcal G)}^2$. Equivalently $\|u\|_{L_2(\mathcal G)}^2 \le \mathcal A^{s(\cdot)}(u,u)/c_{\rm coerc} = \|u\|_{\mathbb V_{\kappa, s}}^2 / c_{\rm coerc}$. Then $\|u\|_{\mathcal A^{s(\cdot)}}^2 = \mathcal A^{s(\cdot)}(u,u) + \|u\|_{L_2(\mathcal G)}^2 \le (1+1/c_{\rm coerc}) \|u\|_{\mathbb V_{\kappa, s}}^2$. Taking square roots gives the right-hand inequality. The norm equivalence implies completeness is preserved, since if $(u_n)$ is Cauchy in $\|\cdot\|_{\mathcal A^{s(\cdot)}}$ then it is Cauchy in $\|\cdot\|_{\mathbb V_{\kappa, s}}$ (by the left inequality) and thus converges in $\mathbb V_{\kappa, s}$; conversely completeness in $\|\cdot\|_{\mathbb V_{\kappa, s}}$ implies completeness in $\|\cdot\|_{\mathcal A^{s(\cdot)}}$ using the right inequality. Thus the form is closed.
\end{proof}

Now, we are ready to prove \sref{thm:form_rep}.
\begin{proof}
Parts (i) and  (ii) are standard consequences of the first representation theorem (Kato--Friedrichs) for closed, symmetric, coercive forms; see \cite[Ch. VI]{Kato1995}. In particular the form-to-operator correspondence gives a positive self-adjoint operator $A$ with $\operatorname{Dom}(A^{1/2})=\mathbb V_{\kappa, s}$ and $\mathcal A^{s(\cdot)}(u,v)=\langle A^{1/2}u,A^{1/2}v \rangle_{L_2(\mathcal G)}$. We verify hypotheses of Kato--Friedrichs first representation theorem.
\begin{itemize}
\item $\mathcal A^{s(\cdot)}$ is closed by \sref{lem:norm_equiv}.

\item $\mathcal A^{s(\cdot)}$ is symmetric and bilinear on $\mathbb V_{\kappa, s}$ by bilinearity and symmetry of $\gamma$.

\item $\mathcal A^{s(\cdot)}$ is continuous and coercive on $\mathbb V_{\kappa, s}$. For continuity, applying the Cauchy--Schwarz inequality twice
\begin{align*}
\lvert\mathcal A^{s(\cdot)}(u,v)\rvert
&\le \bigg\lvert\int_{\mathcal D}\kappa^{2s(x)}u(x)v(x)\,dx\bigg\rvert
  + \bigg\lvert\iint_{\mathcal G\times\mathcal G}(u(x)-u(y))(v(x)-v(y))\gamma(x,y)\,dy\,dx\bigg\rvert\\
&\le \|\kappa^{s}u\|_{L^2(\mathcal G)}\|\kappa^{s}v\|_{L^2(\mathcal G)}
  + |u|_{\mathbb V_{\kappa,s}}\,|v|_{\mathbb V_{\kappa,s}}\le 2\,\|u\|_{\mathbb V_{\kappa,s}}\,\|v\|_{\mathbb V_{\kappa,s}}.
\end{align*}

For coercivity see the proof of \sref{lem:norm_equiv}.

\end{itemize}
For the statement in (ii) with $f\in L_2(\mathcal G)$, it defines a linear bounded functional $v\in L_2(\mathcal G)\mapsto \langle f,v\rangle_{L_2(\mathcal G)}$, hence $L_2(\mathcal G)$ is (canonically) embedded into the dual space $V_{\kappa, s}'$ and therefore the Riesz isomorphism associated to $\mathcal A^{s(\cdot)}$ yields a unique $u\in\mathbb V_{\kappa, s}$ with $\mathcal A^{s(\cdot)}(u,v)=\langle f,v\rangle_{L_2(\mathcal G)}$ for all $v\in\mathbb V_{\kappa, s}$; equivalently $u=A^{-1}f$. The same argument (replacing $f$ by a general $F\in\mathbb V_{\kappa, s}'$) gives the more general $\mathbb V_{\kappa, s}'$-right-hand-side claim, since the map $u\mapsto\mathcal A^{s(\cdot)}(u,\cdot)$ is an isomorphism $\mathbb V_{\kappa, s}\to\mathbb V_{\kappa, s}'$ by coercivity. \added[id=rev]{The two uses of the Riesz representation theorem (in $L_2(\mathcal G)$ and in $\mathbb V_{\kappa,s}$) and the resulting extension of $A^{-1}$ to $\mathbb V_{\kappa,s}'$, including white-noise right-hand sides, are summarized in \sref{rem:riesz_gelfand}.} 
\end{proof}

\subsection{Proof of \sref{prop:form_compare_compact}}
\begin{proof}
By \sref{lem:gamma_two_regime} (using the small- and large-argument asymptotics of the modified Bessel function $K_\nu$ and uniformity in $\beta\in[\underline s,\overline s]$) there exists $r_0\in(0,1]$ and constants $c_1,c_2>0$, $C_{\mathrm tail}\ge0$ such that for every $x,y\in\mathcal G$ with $r:=|x-y|$,
\begin{align}
\label{eq:gamma_bounds}
c_1\, r^{-(d+2\underline s)}\mathbf{1}_{\{r\le r_0\}}(x,y)
\le \gamma(x,y)
\le c_2\, r^{-(d+2\overline s)}\mathbf{1}_{\{r\le r_0\}}(x,y) \;+\; C_{\mathrm tail}\,\mathbf{1}_{\{r>r_0\}}(x,y).
\end{align}
We choose $r_0\le1$ so that the algebraic ordering of power-law kernels is uniform on $\{r\le r_0\}$.
From the left inequality in \eqref{eq:gamma_bounds} we obtain, \replaced[id=rev]{for every $v\in L_2(\mathcal G)$ (all the integrals below are integrals of non-negative functions, so that both sides are well defined in $[0,\infty]$ and Tonelli's theorem applies)}{for every $v\in C_c^\infty(\mathcal G)$},
\begin{align*}
|v|_{\mathbb V_{\kappa, s}}^2
\;=\; \iint_{\mathcal G\times\mathcal G} (v(x)-v(y))^2 \gamma(x,y)\,dy\,dx
\;\ge\; c_1\iint_{r\le r_0} \frac{(v(x)-v(y))^2}{r^{d+2\underline s}}\,dy\,dx.
\end{align*}
Introduce the full fractional seminorm $|v|_{H^{\underline s}(\mathcal G)}^2
= \iint_{\mathcal G\times\mathcal G} \frac{(v(x)-v(y))^2}{r^{d+2\underline s}}\,dy\,dx
= I_{\le r_0} + I_{>r_0}$,
where $I_{\le r_0}$ and $I_{>r_0}$ are the integrals over $\{r\le r_0\}$ and $\{r>r_0\}$ respectively. The previous display gives
$I_{\le r_0} \le c_1^{-1}\,|v|_{\mathbb V_{\kappa, s}}^2$.
For the integral $I_{>r_0}$ note that for $r>r_0$ we have the trivial bound $r^{-(d+2\underline s)}\le r_0^{-(d+2\underline s)}$. Using $(v(x)-v(y))^2\le 2|v(x)|^2+2|v(y)|^2$ and Fubini,
\begin{align*}
I_{>r_0}
\le r_0^{-(d+2\underline s)} \iint_{r>r_0} \big(2|v(x)|^2 + 2|v(y)|^2\big)\,dy\,dx
= 4 r_0^{-(d+2\underline s)} |\mathcal G| \,\|v\|_{L^2(\mathcal G)}^2.
\end{align*}
Set $C := 4 r_0^{-(d+2\underline s)} |\mathcal G|$. Combining the two pieces yields
$c_1 |v|_{H^{\underline s}}^2
\le \,|v|_{\mathbb V_{\kappa, s}}^2 + c_1 C \|v\|_{L^2(\mathcal G)}^2$.
By continuity of $s$ on $\overline{\mathcal G}$ there exist constants $m:=\inf_{x\in\mathcal G}\kappa^{s(x)}>0$, $M:=\sup_{x\in\mathcal G}\kappa^{s(x)}<\infty$, and thus $\|v\|_{L^2(\mathcal G)}^2 \le \frac{1}{m^2} \|\kappa^{s(\cdot)}v\|_{L_2(\mathcal{G})}^2$. Therefore, we have
\begin{align*}
c_1 |v|_{H^{\underline s}}^2
\le \,|v|_{\mathbb V_{\kappa, s}}^2 + c_1C \|v\|_{L^2(\mathcal G)}^2\le \,|v|_{\mathbb V_{\kappa, s}}^2 + \frac{c_1C}{m^2} \|\kappa^{s(\cdot)}v\|_{L_2(\mathcal{G})}^2. 
\end{align*}
Divide both sides by $1 + c_1C/m^2$ to obtain
$ \frac{c_1}{1 + c_1C/m^2}\,|v|_{H^{\underline s}}^2
\le |v|_{\mathbb V_{\kappa, s}}^2 + \|\kappa^{s(\cdot)}v\|_{L_2(\mathcal{G})}^2.
$
Now divide both sides of the inequality $\|\kappa^{s(\cdot)}v\|_{L_2(\mathcal{G})}^2\ge m^2\|v\|_{L^2(\mathcal G)}^2$ by $1 + c_1C/m^2$, we obtain 
\begin{align*}
\frac{m^2}{1 + c_1C/m^2}\,\|v\|_{L^2(\mathcal G)}^2 \le \frac{\|\kappa^{s(\cdot)}v\|_{L_2(\mathcal{G})}^2}{1 + c_1C/m^2} \leq |v|_{\mathbb V_{\kappa, s}}^2 + \|\kappa^{s(\cdot)}v\|_{L_2(\mathcal{G})}^2.
\end{align*}
Set $c_1' := \tfrac12\min \{\frac{c_1}{1 + c_1C/m^2},\;\frac{m^2}{1 + c_1C/m^2}\} > 0$. Then adding the two inequalities above and multiplying by $\tfrac12$ yields $c_1' \mathcal B_{-}(v,v) = c_1'\big(|v|_{H^{\underline s}}^2 + \|v\|_{L^2(\mathcal G)}^2\big)
\le |v|_{\mathbb V_{\kappa, s}}^2 + \|\kappa^{s(\cdot)}v\|_{L_2(\mathcal{G})}^2
= \mathcal A^{s(\cdot)}(v,v)$,
which proves the left inequality in \eqref{eq:form_two_sided}. As for the upper bound, using right inequality in \eqref{eq:gamma_bounds} we get
\begin{align*}
|v|_{\mathbb V_{\kappa, s}}^2
\le c_2\iint_{r\le r_0}\frac{(v(x)-v(y))^2}{r^{d+2\overline s}}\,dy\,dx
\;+\; C_{\mathrm tail}\iint_{r> r_0} (v(x)-v(y))^2\,dy\,dx.
\end{align*}
The first term is bounded above by $c_2\,|v|_{H^{\overline s}}^2$, while the second term is controlled by an $L^2$-term as before. Thus there exist constants $c_2>0$ and $C\ge0$ such that $|v|_{\mathbb V_{\kappa, s}}^2  \le c_2 |v|_{H^{\overline s}}^2 + C \|v\|_{L^2(\mathcal G)}^2$.
Adding the mass term and using $\|\kappa^{s(\cdot)}v\|_{L_2(\mathcal{G})}^2\le M^2\|v\|_{L^2}^2$ yields the right inequality in \eqref{eq:form_two_sided} with a suitable $c_2'>0$.
\added[id=rev]{Since the derivation above only used the pointwise kernel bounds \eqref{eq:gamma_bounds} together with Tonelli's theorem, the inequalities \eqref{eq:form_two_sided} hold directly for every $v\in\mathbb V_{\kappa, s}$, with no density argument being needed; by \sref{prop:meas_integrable} we also have $C_c^\infty(\mathcal D)\subset\mathbb V_{\kappa, s}$, and by \sref{prop:density_moll} every element of $\mathbb V_{\kappa, s}$ is approximated in the energy norm by functions in $C_c^\infty(\mathcal G)$ supported in an arbitrarily small neighbourhood of $\overline{\mathcal D}$.} Finally the embedding $\mathbb V_{\kappa, s}\hookrightarrow H^{\underline s}(\mathcal G)$ follows by the left inequality and it is continuous; since $\underline s>0$ and $\mathcal G$ is bounded, Rellich--Kondrachov gives compactness of $\mathbb V_{\kappa, s}\hookrightarrow L_2(\mathcal G)$, as required. As a result, the associated operator $A$ has compact resolvent and discrete spectrum \cite{Kato1995,Reed1978}. 
\end{proof}

\subsection{Proof of \sref{prop:gen_weak_solution_rhs=W}}
\begin{proof}
\replaced[id=rev]{By construction,}{Clearly} $u$ is a centered Gaussian indexed by $L^2(\mathcal G)$\added[id=rev]{ (each $u(\varphi)$ is an $L_2(\Omega;\mathbb R)$-limit of centered Gaussian random variables, hence centered Gaussian)}. Now fix $v\in\mathbb V_{\kappa,s}$. Choose any sequence $v_n\in\operatorname{Dom}(A)$ with $v_n\to v$ in the $\mathbb V_{\kappa,s}$-norm (for example $v_n=(I+\tfrac1n A)^{-1}v$). For each $n$ the element $A v_n$ lies in $L_2(\mathcal G)$ (by \sref{prop:form_compare_compact}), hence $u(A v_n)$ is defined and $
u(A v_n) = (1/\mu)\mathcal W(A^{-1}(A v_n)) = (1/\mu)\mathcal W(v_n)$.
Because $v_n\to v$ in $\mathbb V_{\kappa,s}$. Continuity of $\mathcal W:L_2(\mathcal G)\to L_2(\Omega; \mathbb R)$ therefore yields $u(A v_n) = \frac{1}{\mu}\,\mathcal W(v_n) \xrightarrow[n\to\infty]{L_2(\Omega; \mathbb R)} \frac{1}{\mu}\,\mathcal W(v)$.
Hence the limit $\lim_{n} u(A v_n)$ exists in $L_2(\Omega; \mathbb R)$ and is independent of the particular approximating sequence. Therefore define, for each $v\in\mathbb V_{\kappa,s}$, $\langle A u, v\rangle_{\mathbb V_{\kappa,s}',\mathbb V_{\kappa,s}}
:= \lim_{n\to\infty} u(A v_n)
= (1/\mu)\,\mathcal W(v)$, with the limit taken in $L_2(\Omega;\mathbb R)$. Note that by definition $\langle A u, v\rangle_{\mathbb V_{\kappa,s}',\mathbb V_{\kappa,s}}$ denotes the action of $A u$ on $v$ which is equivalent to the action of $u$ on $Av$ by linearity of $u$.
\end{proof}

\subsection{Proof of \sref{thm:spectral_underline_s}}
\begin{proof}
\medskip\noindent\textbf{(i)}
Fix $\varphi\in L_2(\mathcal G)$ and expand $\varphi=\sum_{k\ge1}\langle\varphi,\phi_k\rangle_{L_2(\mathcal G)}\phi_k$ in the orthonormal basis $(\phi_k)$. Consider the partial sums $S_N(\varphi):=\frac{1}{\mu}\sum_{j=1}^N \xi_j \lambda_j^{-1}\,\langle\phi_j,\varphi\rangle_{L_2(\mathcal G)}$. Each $S_N(\varphi)$ is a centered Gaussian random variable in $L_2(\Omega; \mathbb R)$ with variance $\operatorname{Var}\big(S_N(\varphi)\big)
\;=\; \mathbb E\big[|S_N(\varphi)|^2\big]
\;=\; (1/\mu^2)\sum_{j=1}^N \lambda_j^{-2}\,|\langle\phi_j,\varphi\rangle_{L_2(\mathcal G)}|^2$.
Since $\sum_j |\langle\phi_j,\varphi\rangle_{L_2(\mathcal G)}|^2=\|\varphi\|_{L_2(\mathcal G)}^2<\infty$ and $\lambda_j^{-2}\le\lambda_1^{-2}$, the variance series is bounded by $\lambda_1^{-2}\|\varphi\|_{L_2(\mathcal G)}^2/\mu^2$. Hence $(S_N(\varphi))_N$ is Cauchy in $L_2(\Omega; \mathbb R)$ and converges in $L_2(\Omega; \mathbb R)$. Linearity in $\varphi$ is preserved in the limit, and the variance bound yields continuity of the map $\varphi\mapsto u(\varphi)$ as a map from $L_2(\mathcal G)\to L_2(\Omega; \mathbb R)$. This proves (i).

\medskip\noindent\textbf{(ii)}
We compute the mean-square $H^r(\mathcal G)$-norm of $u$. Using the expansion of $u$ in $(\phi_j)$ and independence of the $\xi_j$, $\mathbb E\big[\|u\|_{H^r(\mathcal G)}^2\big] \;=\; (1/\mu^2)\sum_{j=1}^\infty \lambda_j^{-2}\,\|\phi_j\|_{H^r(\mathcal G)}^2$.
To show finiteness of the right-hand side for $r<2\underline s-\tfrac d2$ we need two ingredients, namely (A) a spectral lower bound given by \eqref{eq:lambda_growth}, and (B) an interpolation estimate for the eigenfunctions. In particular, the interpolation is taken between $L_2(\mathcal G)$ and $H^{\underline s}(\mathcal G)$. By standard interpolation (Gagliardo–Nirenberg interpolation), for $0\le r\le\underline s < 1$, $\|\phi_j\|_{H^r(\mathcal G)} \le \|\phi_j\|_{L_2(\mathcal G)}^{1-r/\underline s}\,\|\phi_j\|_{H^{\underline s}(\mathcal G)}^{r/\underline s}
= \|\phi_j\|_{H^{\underline s}(\mathcal G)}^{r/\underline s}$,
since $\|\phi_j\|_{L_2(\mathcal G)}=1$. From the variational identity $\lambda_j=\mathcal A^{s(\cdot)}(\phi_j,\phi_j)$ and the lower bound in \eqref{eq:form_two_sided} we have $\lambda_j \gtrsim \|\phi_j\|_{H^{\underline s}(\mathcal G)}^2$ (up to harmless additive constants coming from mass terms), so for large $j$, $\|\phi_j\|_{H^r(\mathcal G)}^2 \lesssim \lambda_j^{r/\underline s}$.
Therefore, we have $\mathbb E\big[\|u\|_{H^r(\mathcal G)}^2\big] \lesssim \sum_{j=1}^\infty \lambda_j^{-2 + r/\underline s}$.
Using \eqref{eq:lambda_growth} we obtain for large $j$, $\lambda_j^{-2 + r/\underline s} \lesssim j^{(2\underline s/d)\left(-2 + r/\underline s\right)}
= j^{(-4\underline s + 2r)/d}$.
The series $\sum_{j\ge1} j^{(-4\underline s + 2r)/d}$ converges precisely when $r<2\underline s - \tfrac d2$.
\added[id=rev]{The interpolation step requires $r\le\underline s$, which covers the full range $r<2\underline s-\tfrac d2$ whenever $d\ge2$, or $d=1$ and $\underline s\le\tfrac12$, since then $2\underline s-\tfrac d2\le\underline s$; in these cases $r_\star=2\underline s-\tfrac d2$ and the proof of (ii) is complete. In the remaining case $d=1$ and $\underline s>\tfrac12$, eigenfunction regularity beyond the form domain is required. Each eigenfunction satisfies $\phi_j=\lambda_j A^{-1}\phi_j$, so \sref{prop:reg_variable}(ii), which applies since $\underline s>\tfrac12>\tfrac d4$ and the domains are $C^\infty$ (its proof is independent of the present theorem), gives $\|\phi_j\|_{H^\sigma(\mathcal G)}\le C_\sigma\lambda_j$ for every $\sigma<r_g$. For $r=(1-\theta)\underline s+\theta\sigma$ with $\theta\in(0,1)$, interpolation between $H^{\underline s}(\mathcal G)$ and $H^{\sigma}(\mathcal G)$ together with $\|\phi_j\|_{H^{\underline s}(\mathcal G)}\lesssim\lambda_j^{1/2}$ yields $\|\phi_j\|_{H^r(\mathcal G)}\le\|\phi_j\|_{H^{\underline s}(\mathcal G)}^{1-\theta}\|\phi_j\|_{H^{\sigma}(\mathcal G)}^{\theta}\lesssim\lambda_j^{(1+\theta)/2}$. Hence $\mathbb E\big[\|u\|_{H^r(\mathcal G)}^2\big]\lesssim\sum_{j\ge1}\lambda_j^{\theta-1}\lesssim\sum_{j\ge1} j^{2\underline s(\theta-1)}$ by \eqref{eq:lambda_growth}, which converges precisely when $\theta<1-\tfrac1{2\underline s}$, that is, $r<\underline s+(\sigma-\underline s)\big(1-\tfrac1{2\underline s}\big)$. Letting $\sigma\uparrow r_g$ gives the claimed range $r<r_\star$ in \eqref{eq:r_star}.}
Thus for every $r<r_\star$ we have $\mathbb E\big[\|u\|_{H^r(\mathcal G)}^2\big]<\infty$, which implies $u\in H^r(\mathcal G)$ almost surely. Setting $r=0$ yields the particular statement $u\in L_2(\mathcal G)$ a.s. whenever $\underline s>d/4$.

\medskip\noindent\textbf{(iii)} Assume $\underline s>d/4$, we show $A u=(1/\mu)\mathcal W$ in the generalized sense. \replaced[id=rev]{For any $\varphi\in \operatorname{Dom}(A)$, write the spectral representation $A\varphi=\sum_k \lambda_k\,\langle\varphi,\phi_k\rangle_{L_2(\mathcal G)}\phi_k$ (convergent in $L_2(\mathcal G)$ since $\varphi\in\operatorname{Dom}(A)$). Using the definition of the action $u(\cdot)$ from (i) and orthonormality,}{For any $\varphi\in L_2(\mathcal G)$, write the spectral representation $A\varphi=\sum_k \lambda_k\,\langle\varphi,\phi_k\rangle_{L_2(\mathcal G)}\phi_k$ (convergent in $L_2(\mathcal G)$). Using the definition of $u$ from (i) and orthonormality,}
\begin{align*}
\replaced[id=rev]{u(A\varphi)}{\langle A u,\varphi\rangle_{L_2(\mathcal G)} = \langle u, A\varphi\rangle_{L_2(\mathcal G)}}
&= \frac{1}{\mu}\sum_{j=1}^\infty \xi_j \lambda_j^{-1}\,\langle\phi_j, A\varphi\rangle_{L_2(\mathcal G)} = \frac{1}{\mu}\sum_{j=1}^\infty \xi_j\,\langle\phi_j,\varphi\rangle_{L_2(\mathcal G)} = \frac{1}{\mu}\,\mathcal W(\varphi),
\end{align*}
where the final equality is the standard expansion of white noise in the orthonormal basis. All series converge in $L_2(\Omega; \mathbb R)$ by the variance bounds from part (i), so the identity holds in $L_2(\Omega; \mathbb R)$. Uniqueness among centered Gaussian generalized fields follows by testing any candidate solution $v$ against the eigenfunctions \replaced[id=rev]{$\phi_j=A(\lambda_j^{-1}\phi_j)$ with $\lambda_j^{-1}\phi_j\in\operatorname{Dom}(A)$, so that the relation $v(\phi_j)=v\big(A(\lambda_j^{-1}\phi_j)\big)=(1/\mu)\mathcal W(\lambda_j^{-1}\phi_j)=(1/\mu)\lambda_j^{-1}\xi_j$ determines all the coefficients $v(\phi_j)$}{$\phi_j$: the relation $\lambda_j\,\langle v,\phi_j\rangle_{L_2(\mathcal G)}=(1/\mu)\xi_j$ forces the coefficients $\langle v,\phi_j\rangle_{L_2(\mathcal G)}$ to equal $(1/\mu)\lambda_j^{-1}\xi_j$ for every $j$}, so $v$ and $u$ have identical finite-dimensional distributions and coincide as Gaussian generalized fields.
\end{proof}

\subsection{Proof of \sref{prop:pathwise_coincidence}}
\begin{proof}
Fix $\omega$ in the full-measure set where $\mathcal W(\omega)\in\mathbb V_{\kappa,s}'$ (which is true for $\underline s>d/2$).  
By the deterministic form representation (\sref{thm:form_rep}(ii)) there exists a unique
$u(\omega)\in\mathbb V_{\kappa,s}$ solving the variational equation $
\mathcal A^{s(\cdot)}\big(u(\omega),v\big)
=(1/\mu)\,\langle\mathcal W(\omega),v\rangle_{\mathbb V_{\kappa,s}',\mathbb V_{\kappa,s}}$, for all $v\in\mathbb V_{\kappa,s}$,
and $u(\omega)=\tfrac{1}{\mu}A^{-1}\mathcal W(\omega)$.  Take an arbitrary test function
$\varphi\in L_2(\mathcal G)$. Since $A^{-1}\varphi\in\operatorname{Dom}(A)\subset\mathbb V_{\kappa,s}$,
we may use $v=A^{-1}\varphi$ in the variational identity $\mathcal A^{s(\cdot)}\big(u(\omega),A^{-1}\varphi\big)
=(1/\mu)\,\langle\mathcal W(\omega),A^{-1}\varphi\rangle_{\mathbb V_{\kappa,s}',\mathbb V_{\kappa,s}}$.
By the form--operator relation $\mathcal A^{s(\cdot)}(w,z)=\langle Aw,z\rangle_{L_2(\mathcal G)}$ (valid for $w\in\operatorname{Dom}(A)$, $z\in\mathbb V_{\kappa,s}$) the left-hand side equals
$\langle Au(\omega),A^{-1}\varphi\rangle_{L_2(\mathcal G)}$. Using self-adjointness of $A$ and $A(A^{-1}\varphi)=\varphi$ we obtain $
\langle u(\omega),\varphi\rangle_{L_2(\mathcal G)}
=(1/\mu)\,\langle\mathcal W(\omega),A^{-1}\varphi\rangle_{\mathbb V_{\kappa,s}',\mathbb V_{\kappa,s}}$.
The right-hand side is precisely the defining action of the generalized solution $u$ on $\varphi$ in \sref{prop:gen_weak_solution_rhs=W}.
Thus for this $\omega$ and every $\varphi\in L_2(\mathcal G)$,
$\langle u(\omega),\varphi\rangle_{L_2(\mathcal G)}=u(\omega)(\varphi)$.
\end{proof}

\subsection{Proof of \sref{prop:discrete_wellposed}}
\begin{revblock}
\begin{proof}
The discretization is conforming ($\mathbb V_{\kappa,s}(\mathcal T)\subset\mathbb V_{\kappa,s}$ is a closed (finite-dimensional) subspace). The restriction of $\mathcal A^{s(\cdot)}(\cdot,\cdot)$ to $\mathbb V_{\kappa,s}(\mathcal T)$ therefore remains continuous and coercive with the same constants as in \sref{thm:form_rep}, so existence and uniqueness of the discrete solution follow from the Lax--Milgram theorem; equivalently, the associated stiffness matrix is symmetric positive definite. The load is almost surely finite, that is, each $\mathcal W(\psi_i)$, $i=1,\dots,N$, is a well-defined Gaussian random variable since $\psi_i\in L_2(\mathcal G)$ (see \autoref{subsec:sampling}), hence $\mathcal W(V)$ is defined for all $V\in\mathbb V_{\kappa,s}(\mathcal T)$ simultaneously on a set of probability one.
\end{proof}
\end{revblock}

\begin{revblock}
\subsection{Proof of \sref{prop:reg_variable}}
\label{app:proof_reg_variable}

The local boundedness step in the proof of \sref{prop:reg_variable} relies on the following result, which we quote from \cite{Cozzi2017} specialised to $p=2$.

\begin{lemma}[Local boundedness for nonlocal equations {\cite[Theorem~2.3, Theorem~8.1]{Cozzi2017}}]\label{lem:cozzi_bound}
Let $\mathcal{U}\subset\mathbb R^d$ be a bounded open set, $s_0\in(0,1)$, and let
$K:\mathbb R^d\times\mathbb R^d\to[0,\infty]$ be a symmetric measurable kernel satisfying
the two-sided bound
\begin{equation}\label{eq:cozzi_kernel}
  \frac{\Lambda^{-1}\,\chi_{B_{r_0}}(x{-}y)}{|x{-}y|^{d+2s_0}}
  \;\le\; K(x,y) \;\le\; \frac{\Lambda}{|x{-}y|^{d+2s_0}},
  \qquad x,y\in\mathbb R^d,
\end{equation}
for some $\Lambda\ge 1$ and $r_0>0$.
Let $u\in H^{s_0}(\mathcal{U})$ with
$\int_{\mathbb R^d}|u(y)|\,(1+|y|)^{-d-2s_0}\,dy<\infty$
be a weak solution of
$-\mathcal{L}u = f$ in $\mathcal{U}$, where $\mathcal L$ is any nonlocal operator of the integral form $\mathcal{L}u(x) := \mathrm{P.V.}\!\int_{\mathbb R^d}\!\bigl(u(x)-u(y)\bigr)\,K(x,y)\,dy$.
Define the nonlocal tail
\begin{equation}\label{eq:tail_def}
  \mathrm{Tail}(u;x_0,R)
  := R^{2s_0}\!\int_{\mathbb R^d\setminus B_R(x_0)}\!|u(y)|\,|x_0{-}y|^{-d-2s_0}\,dy.
\end{equation}
\begin{enumerate}[(i)]
\item If $d<2s_0$, then $u\in L^\infty_{\mathrm{loc}}(\mathcal{U})$
by the fractional Morrey embedding $H^{s_0}(\mathcal{U})\hookrightarrow L^\infty(\mathcal{U})$.
\item If $d\ge 2s_0$ and $f\in L_q(\mathcal{U})$ for some $q>d/(2s_0)$, then
$u\in L^\infty_{\mathrm{loc}}(\mathcal{U})$.
More precisely, for any $x_0\in\mathcal{U}$ and $0<R<\mathrm{dist}(x_0,\partial\mathcal{U})/4$,
\begin{equation}\label{eq:cozzi_Linfty}
  \|u\|_{L_\infty(B_R(x_0))}
  \le C\,R^{-d/2}\,\|u\|_{L_2(B_{2R}(x_0))}
    + \mathrm{Tail}(u;x_0,R)
    + C\,R^{2s_0-d/q}\,\|f\|_{L_q(B_{2R}(x_0))},
\end{equation}
with $C=C(d,s_0,\Lambda,q)$.
\end{enumerate}
\end{lemma}

We remark that the kernel bound \eqref{eq:cozzi_kernel} differs from Cozzi's convention by a $(1{-}s_0)$ normalisation factor (included there for consistency with the local limit $s_0\to 1^-$); for fixed $s_0\in(0,1)$ this factor is absorbed into $\Lambda$. When only $f\in L_q(\mathcal U)$ is available (Cozzi states the result for right-hand sides bounded pointwise), the forcing integral in the Caccioppoli inequality is estimated via H\"older's inequality and the fractional Sobolev embedding, which produces the $R^{2s_0-d/q}\|f\|_{L_q}$ term in \eqref{eq:cozzi_Linfty} and requires precisely $q>d/(2s_0)$.

Throughout the proof below, $\omega(\rho):=\sup\{|s(x)-s(y)|:x,y\in\overline{\mathcal G},\,|x-y|\le\rho\}$ denotes the modulus of continuity of $s$, $B_\rho(x)$ a ball of radius $\rho$ centered at $x$, and for a frozen order $s_0\in[\underline s,\overline s]$ we write $\gamma_{s_0}$ for the constant-order kernel obtained from $\gamma$ by replacing $\beta(x,y)$ with $s_0$, and $\mathcal A^{s_0}$ for the associated constant-order bilinear form.

\begin{proof}[Proof of \sref{prop:reg_variable}]
Fix $x_0\in\overline{\mathcal D}$ and set $s_0:=s(x_0)$. By \sref{lem:gamma_two_regime}(a) and the frozen-order comparison argument of \sref{lem:frozen_order}, the variable-order kernel satisfies
\begin{equation}\label{eq:kernel_comparability}
  c\,|x{-}y|^{-(d+2s_0)} \;\le\; \gamma(x,y) \;\le\; C\,|x{-}y|^{-(d+2s_0)},
  \qquad x,y\in B_\rho(x_0),\;\; 0<|x{-}y|\le\rho,
\end{equation}
for $\rho$ small enough, with constants $c,C>0$ depending on $\omega(\rho)$.

Now choose $\eta\in C_c^\infty(B_\rho(x_0))$ with $\eta\equiv 1$ on $B_{\rho/2}(x_0)$, and set $w:=\eta u$.
Writing $\mathcal A^{s_0}(w,\phi)=\mathcal A^{s_0}(\eta u,\phi)$ and adding and subtracting the variable-order bilinear form $\mathcal A^{s(\cdot)}(\eta u,\phi)$ yields
\begin{equation}\label{eq:localized_frozen_identity}
  \mathcal A^{s_0}(w,\phi)
  = \langle \eta f,\phi\rangle_{L_2(\mathcal G)}
  + \mathcal{C}_{\mathrm{om}}^\eta(u,\phi)
  + \mathcal{E}^{s_0}(u,\phi),
  \qquad \phi\in\mathbb V_{\kappa,s},
\end{equation}
where the commutator and freezing errors are
\[
  \mathcal{C}_{\mathrm{om}}^\eta(u,\phi)
  := \iint_{\mathcal G\times\mathcal G}
  \bigl(\eta(x)-\eta(y)\bigr)
  \bigl(u(y)\phi(x)-u(x)\phi(y)\bigr)
  \gamma_{s_0}(x,y)\,dy\,dx,
\]
\begin{align*}
  \mathcal{E}^{s_0}(u,\phi)
  &:= \int_{\mathcal D}\bigl(\kappa^{2s_0}-\kappa^{2s(x)}\bigr)\eta(x)u(x)\phi(x)\,dx \\
  &\quad + \iint_{\mathcal G\times\mathcal G}
  \bigl(u(x)-u(y)\bigr)\bigl(\eta(x)\phi(x)-\eta(y)\phi(y)\bigr)
  \bigl(\gamma_{s_0}(x,y)-\gamma(x,y)\bigr)\,dy\,dx.
\end{align*}
Next, we bound $|\mathcal{E}^{s_0}(u,\phi)|$. The error $\mathcal{E}^{s_0}$ consists of a mass term and a jump-kernel term. As for the mass term,
since $|s(x)-s_0|\le\omega(\rho)$ on $B_\rho(x_0)\supset\operatorname{supp}\eta$, we have
$|\kappa^{2s_0}-\kappa^{2s(x)}|=\kappa^{2s_0}|1-\kappa^{2(s(x)-s_0)}|\le C\,\omega(\rho)$
(using $|e^t-1|\le C|t|$ for bounded $t$), so
\[
  \Bigl|\int_{\mathcal D}\bigl(\kappa^{2s_0}-\kappa^{2s(x)}\bigr)\eta u\,\phi\,dx\Bigr|
  \le C\,\omega(\rho)\,\|u\|_{L_2(\mathcal G)}\|\phi\|_{L_2(\mathcal G)}.
\]
As for the jump-kernel term, we split the domain
of the bilinear integral involving $\gamma_{s_0}-\gamma$ into
$\{|x-y|\le\rho\}\cap(\operatorname{supp}\eta\times\mathcal G)$ and $\{|x-y|>\rho\}\cap(\operatorname{supp}\eta\times\mathcal G)$.

On $\{|x-y|\le\rho\}$ with $x\in B_\rho(x_0)$, the kernel difference satisfies
$$|\gamma_{s_0}(x,y)-\gamma(x,y)|\le C\,\omega(\rho)\,(|\log|x{-}y||+1)\,|x-y|^{-(d+2s_0)},$$
since $\beta(x,y)=s_0+O(\omega(\rho))$ and both the normalization constant $c(\beta)$ and the power $|x-y|^{-2\beta}$ are Lipschitz in $\beta$ with a logarithmic correction from differentiating $|x-y|^{-2\beta}$ in $\beta$. For any $\delta>0$, the logarithmic factor is absorbed by $|x-y|^{-\delta}$, giving
$|\gamma_{s_0}(x,y)-\gamma(x,y)|\le C_\delta\,\omega(\rho)\,|x-y|^{-(d+2s_0+\delta)}$.
Therefore, by Cauchy--Schwarz,
\begin{align*}
  &\Bigl|\iint_{|x-y|\le\rho}(u(x)-u(y))(\eta\phi(x)-\eta\phi(y))(\gamma_{s_0}-\gamma)\,dy\,dx\Bigr|\\
  &\quad\le C_\delta\,\omega(\rho)
  \Bigl(\iint\frac{|u(x)-u(y)|^2}{|x-y|^{d+2s_0+\delta}}\,dy\,dx\Bigr)^{1/2}
  \Bigl(\iint\frac{|\eta\phi(x)-\eta\phi(y)|^2}{|x-y|^{d+2s_0+\delta}}\,dy\,dx\Bigr)^{1/2}\\
  &\quad\le C_\delta\,\omega(\rho)\,\|u\|_{H^{s_0+\delta/2}(\mathcal G)}\,\|\phi\|_{H^{s_0+\delta/2}(\mathcal G)}.
\end{align*}

On $\{|x-y|>\rho\}$, one can rewrite the kernel as
$\gamma(x,y)=\tfrac{1}{2}P(\beta)\,K_\nu(\kappa r)\,r^{-\nu}\,r^{-(d+2\beta)}\,r^{d/2+\beta+\nu-\nu}$, i.e., with $r=|x-y|$, $\beta=\beta(x,y)$, $\nu=d/2+\beta$, define the map $F(\beta,r):=P(\beta)\,K_{d/2+\beta}(\kappa r)\,r^{d/2+\beta}\,r^{-(d+2\beta)}$ so that $\gamma=\tfrac{1}{2}F(\beta(x,y),r)$ and $\gamma_{s_0}=\tfrac{1}{2}F(s_0,r)$, where $P(\beta)$ is the smooth prefactor from the proof of \sref{lem:gamma_two_regime}.
By the two-regime kernel bounds (\sref{lem:gamma_two_regime}(b)), for $r\ge r_0>0$ and any $\beta\in[\underline s,\overline s]$,
$|F(\beta,r)|\le C\,r^{-1/2}\,e^{-\kappa r}$,
so in particular both $\gamma_{s_0}$ and $\gamma$ are bounded on $\{r>\rho\}$ (with $\rho\ge r_0$). By the mean value theorem in $\beta$,
$|F(s_0,r)-F(\beta,r)|=|\partial_\beta F(\tilde\beta,r)|\,|s_0-\beta|$
for some $\tilde\beta$ between $s_0$ and $\beta$.
Since $x\in\operatorname{supp}\eta\subset B_\rho(x_0)$, we have $|s_0-\beta(x,y)|\le\omega(\rho)$.
It remains to verify that $\partial_\beta F(\tilde\beta,r)$ is bounded for $r\ge\rho$.
The $\beta$-derivative of $F$ involves $\partial_\beta P(\beta)$ (which is bounded), together with $\partial_\nu K_\nu(z)|_{z=\kappa r}$ and
$\partial_\beta[r^{-d/2-\beta}]=(-\log r)\,r^{-d/2-\beta}$.
For the Bessel derivative, the Nicholson-type bound $|\partial_\nu K_\nu(z)|\le C(1+|\log z|)\,K_\nu(z)$ (see \cite{Olver2010nist}) combined with the large-argument bound $K_\nu(z)\le C\,z^{-1/2}e^{-z}$ for $z\ge z_0$ (\sref{lem:Bessel_uniform}(ii)) gives $|\partial_\nu K_\nu(\kappa r)|\le C\,(1+\log(\kappa r))\,r^{-1/2}\,e^{-\kappa r}$ for $r\ge\rho$.
The logarithmic factors $\log r$ and $\log(\kappa r)$ are absorbed by the exponential decay on $\{r\ge\rho\}$, yielding
$|\partial_\beta F(\tilde\beta,r)|\le C_\rho$ uniformly for $r\ge\rho$ and $\tilde\beta\in[\underline s,\overline s]$.
Therefore $|\gamma_{s_0}(x,y)-\gamma(x,y)|\le C_\rho\,\omega(\rho)$ uniformly on $\{|x-y|>\rho\}$.

Returning to the bilinear integral, the integrand is bounded by $C_\rho\,\omega(\rho)\,|u(x)-u(y)|\,|\eta(x)\phi(x)-\eta(y)\phi(y)|$, and Cauchy--Schwarz with the $L_2(\mathcal G)$ norms of $u$ and $\phi$ (using $\mathcal G$ bounded) gives a contribution $\le C_\rho\,\omega(\rho)\,\|u\|_{L_2(\mathcal G)}\|\phi\|_{L_2(\mathcal G)}$.
Combining and relabelling $\delta$ yields
\begin{equation}\label{eq:localized_error_bound}
  |\mathcal{E}^{s_0}(u,\phi)|
  \le C_{\delta,\rho}\,\omega(\rho)
  \bigl(\|u\|_{H^{s_0+\delta}(\mathcal G)}\|\phi\|_{H^{s_0+\delta}(\mathcal G)}
  + \|u\|_{L_2(\mathcal G)}\|\phi\|_{L_2(\mathcal G)}\bigr).
\end{equation}

We are now ready to prove \eqref{eq:var_reg_interior}. Fix $x_0\in\mathcal D'$ and choose $\rho>0$ with $B_\rho(x_0)\subset\mathcal D$. Set $s_0=s(x_0)$. Since $\underline{s}'>d/4$, shrinking $\rho$ if necessary ensures $s_0>d/4$ for every center in the cover of $\overline{\mathcal D'}$. We apply \sref{lem:cozzi_bound} with $K=\gamma$ and $\mathcal U=B_\rho(x_0)$.
The kernel comparability \eqref{eq:kernel_comparability} verifies
hypothesis \eqref{eq:cozzi_kernel}, since the lower bound holds for $|x{-}y|\le\rho$ (taking $r_0=\rho$),
while the upper bound extends to all of $\mathbb R^d$ because for $|x{-}y|>\rho$ the
far-field kernel bound (\sref{lem:gamma_two_regime}(b)) gives
$\gamma(x,y)\le C_\rho\,|x{-}y|^{-1/2}e^{-\kappa|x{-}y|}$.
We distinguish two cases depending on the spatial dimension. In case $d<2s_0$, part (i) of \sref{lem:cozzi_bound} applies, since the fractional Morrey embedding gives $u\in L^\infty_{\mathrm{loc}}(B_\rho(x_0))$ directly. In case $d\ge 2s_0$, part (ii) of \sref{lem:cozzi_bound} applies with $q=2$.
The admissibility condition $q>d/(2s_0)$, i.e.\ $2>d/(2s_0)$, is equivalent to $s_0>d/4$,
which holds by hypothesis.
Estimate \eqref{eq:cozzi_Linfty} (with $R=3\rho/4$) gives
\[
  \|u\|_{L_\infty(B_{3\rho/4}(x_0))}
  \le C\bigl(\|u\|_{L_2(B_\rho(x_0))}+\rho^{2s_0-d/2}\|f\|_{L_2(B_\rho(x_0))}+\mathrm{Tail}(u;x_0,\rho)\bigr),
\]
where the tail \eqref{eq:tail_def} is finite since $u$ is supported in $\overline{\mathcal G}$.

Next, on $B_{\rho/2}(x_0)$, $\eta\equiv 1$ so the commutator $\mathcal{C}_\mathrm{om}^\eta$ vanishes.
The freezing error \eqref{eq:localized_error_bound} contributes an $L_2$ perturbation of size $O(\omega(\rho))$.
The right-hand side of \eqref{eq:localized_frozen_identity} therefore belongs to $L_2(B_{\rho/2}(x_0))$, and the constant-order interior estimate for $(\kappa^2-\Delta)^{s_0}$ gives
\[
  \|u\|_{H^{2s_0-\varepsilon}(B_{\rho/2}(x_0))}
  \le C_{\varepsilon,\rho}\bigl(\|f\|_{L_2(B_\rho(x_0))}
  + \|u\|_{H^{\underline s}(\mathcal G)}
  + \|u\|_{L_\infty(B_{3\rho/4}(x_0))}\bigr).
\]

Finally, cover $\overline{\mathcal D'}$ with finitely many balls $B_{\rho/2}(x_i)$ with $s(x_i)\ge\underline{s}'-\omega(\rho)$.
The inequality $\|u\|_{L_2(\mathcal G)}\leq \|u\|_{H^{\underline s}(\mathcal G)}$, the energy estimate $\|u\|_{H^{\underline s}(\mathcal G)}\le c\|f\|_{L_2(\mathcal G)}$ (which follows from the left inequality in \eqref{eq:form_two_sided} and $\|u\|_{\mathbb V_{\kappa,s}}^2=\mathcal A^{s(\cdot)}(u,u)=\langle f,u\rangle_{L_2(\mathcal G)}\le\|f\|_{L_2(\mathcal G)}\|u\|_{L_2(\mathcal G)}$ together with the continuous embedding $\mathbb V_{\kappa,s}\hookrightarrow L_2(\mathcal G)$), and the local boundedness bound control all right-hand side terms by $\|f\|_{L_2(\mathcal G)}$.
Summing the local estimates and absorbing the $\omega(\rho)$ loss proves \eqref{eq:var_reg_interior}.

Next, we prove the global regularity \eqref{eq:var_reg_global}.
We cover $\overline{\mathcal D}$ by finitely many balls $B_\rho(x_i)$. When $x_0\in\partial\mathcal D$,
the frozen problem $r^+ A_{s_0}w=g$ on the $C^\infty$ domain $\mathcal D$ falls within the scope of \cite[Theorem~7.1]{Grubb2015} (with $a=s_0$, $p=2$, $s=2s_0$, where $r^+$ is the restriction-to-$\mathcal D$ operator, $r^+v:=v|_{\mathcal D}$). The frozen integral-form operator $A_{s_0}$ coincides with the fractional power $(\kappa^2-\Delta)^{s_0}$, which by \cite[Lemma~2.9 and Example~2.8]{Grubb2015} is a classical pseudodifferential operator of order $2s_0$ satisfying the $\mu$-transmission condition with $\mu=s_0$ and factorization index $s_0$; its homogeneous Dirichlet realization is the operator described in \cite[Example~7.2]{Grubb2015}. \cite[Theorems~5.4 and~7.1, Corollary~5.5]{Grubb2015} then yield the decomposition $w=\operatorname{dist}(x,\partial\mathcal D)^{s_0}v_0+w_{\rm reg}$ with $v_0\in H^{s_0-\varepsilon}(\mathcal D)$ and $w_{\rm reg}\in H^{2s_0-\varepsilon}(\mathcal D)$, and combining with the embedding of $\mu$-transmission spaces into standard Sobolev spaces gives
\[
  w \in H^{t'}(\mathcal D) \qquad\text{for every } t'<\min(2s_0,\,s_0+\tfrac{1}{2}).
\]
The threshold $s_0+\tfrac{1}{2}$ reflects $\operatorname{dist}(x,\partial\mathcal D)^{s_0}\in H^\sigma(0,1)$ iff $\sigma<s_0+\tfrac{1}{2}$. When $x_0\in\mathcal D$ with $B_\rho(x_0)\subset\mathcal D$, there is no boundary constraint and the frozen operator gives the full $2s_0$ Sobolev gain. Finally, since $\underline s\le s(x_0)$ for every center, combining the estimates across the finite cover gives the rate $\min(2\underline s,\,\underline s+\tfrac{1}{2})$.
\end{proof}
\end{revblock}

\subsection{Proof of \sref{prop:FEM_MS_rate}}
\begin{proof}
Setting $t=0$ in \eqref{eq:scott_zhang}, squaring and taking expectation to both sides and noticing that $\mathbb E \|u\|_{H^r(\mathcal G)}^2 < \infty$ for \replaced[id=rev]{$r<r_\star$ by \sref{thm:spectral_underline_s}(ii)}{$r<2\underline s -d/2$}, we get \eqref{eq:FEM_L2_error}.
\end{proof}

\subsection{Proof of \sref{prop:local-quad-error}}
\begin{proof}
Our argument is based on the derivative-free quadrature error estimates \cite{davis1975interpolation}. We shall check whether each component in the integrand in \eqref{eq:transformed-integrand_adj} or \eqref{eq:transformed-integrand_ident} can be analytically extended to a closed ellipse $\mathcal E_\rho$ with focus points $0$ and $1$ and where $\rho>\tfrac12$ is the sum of semimajor and semiminor axes. \replaced[id=rev]{By hypothesis $s$ extends to a bounded analytic function on a complex neighborhood $\mathcal N$ of $\overline{\mathcal G}$, and the affine maps $T_i$ together with $\xi(\zeta)=\zeta^{1/(3-2\overline s)}$ and $\eta(t)$ are analytic away from the origin, so the composed map $(x,y)=(T_i(\xi(\zeta)),T_{i+1}(\xi(\zeta)\eta(t)))$, and with it $\beta$ and $\Phi$, is analytic in each variable on any region avoiding $\zeta=0$ (respectively $t=0$) whose image stays in $\mathcal N$. The factors $\zeta^{\frac{2(\overline s-\beta)}{3-2\overline s}}$ and $t^{\frac{2\overline s-2\beta}{2-2\overline s}}$ have nonnegative, in general noninteger, exponents and hence a branch point at the origin. We therefore work on ellipses avoiding it. For $\rho\in(\tfrac12,1)$, let $\mathcal E_\rho^\delta\subset\mathcal E_\rho\setminus\{0\}$ be the slightly shrunken confocal ellipse whose closure excludes the branch cut $(-\infty,0]$; the integrand is analytic on $\mathcal E_\rho^\delta$ in each variable, uniformly in the other variable in $[0,1]$, and the principal branches of the power factors are bounded there since their exponents are nonnegative and their moduli are bounded by $C\sup_{\mathcal E_\rho}|\zeta|^{\Re\alpha}e^{\pi|\Im\alpha|}$ with $\alpha$ ranging over the (bounded) exponent values. This choice restricts the admissible $\rho$ but preserves the geometric rate.}{To this end, we first note that the final transformation $(x,y) = (T_i(\xi(\zeta)), T_{i+1}(\xi(\zeta)\eta(t)))$ is component-wise analytic since the variable order $s(x)$ is an analytic function, hence, $\beta$ and the whole integrand is analytic in the complex plane.} Now we bound the integrand for each component. Letting $\rho\in(\tfrac12,1)$, there hold
\begin{align*}
  \sup_{\substack{\zeta\in \mathcal E_\rho \\ t\in[0,1]}} |h^{1 - 2\beta} \, \zeta^{\frac{2\overline{s} - 2\beta}{3 - 2\overline{s}}} t^{\frac{2\overline{s} - 2\beta}{2 - 2\overline{s}}} \Phi(\zeta,t)| \le C h^{1-2\overline s}
  \text{ and } \sup_{\substack{t\in \mathcal E_\rho \\ \zeta\in[0,1]}} |h^{1 - 2\beta} \, \zeta^{\frac{2\overline{s} - 2\beta}{3 - 2\overline{s}}} t^{\frac{2\overline{s} - 2\beta}{2 - 2\overline{s}}} \Phi(\zeta,t)| \le C h^{1-2\overline s} .
\end{align*}
Thus we invoke Theorem~5.3.15 in \cite{Sauter2011}\added[id=rev]{, applied on the shrunken ellipses $\mathcal E_\rho^\delta$ (with $2\rho$ replaced by the corresponding sum of semiaxes, absorbed into $\rho$),} to get  $
 \displaystyle |\mathcal A^2_{\tau,\tau'}(\psi_j,\psi_i) - Q_{\tau,\tau'}^n| \le Ch^{1-2\overline s}(2\rho)^{-2n}$.
\end{proof}

%
%
%
%
%
%
%
%

\bibliographystyle{plainnat}
\bibliography{references1}
\end{document}